\documentclass{amsart}

\usepackage{anyfontsize,fix-cm}
\usepackage[T1]{fontenc}
\usepackage{comment}
\usepackage[vietnamese,english]{babel}
\usepackage{amssymb, amsthm, enumerate, mathtools, microtype}

\usepackage{enumitem}
\usepackage{mathabx}

\usepackage[backref, pdfencoding=auto,  psdextra]{hyperref}
\usepackage[abbrev, nobysame]{amsrefs} 
\makeatletter
\def\bibtex@style{amsrn-mod}
\makeatother

\renewcommand{\MR}[1]{} 
\renewcommand{\PrintDOI}[1]{}
\BibSpec{arxiv}{%
    +{}  {\PrintAuthors}                {author}
    +{,} { \textit}                     {title}
    +{:} { \textit}                     {subtitle}
    +{,} { \PrintContributions}         {contribution}
    +{.} { \PrintPartials}              {partial}
    +{,} { }                            {journal}
    +{}  { \PrintDatePV}                {date}
     +{,} { \eprint}        {url}
}

\usepackage{tikz-cd}
\usetikzlibrary{babel, shapes.geometric}
\usetikzlibrary {arrows.meta,positioning}
\newtheorem{theorem}{Theorem}[section]
\newtheorem{lemma}[theorem]{Lemma}
\newtheorem{question}[theorem]{Question}
\newtheorem{corollary}[theorem]{Corollary}
\newtheorem{proposition}[theorem]{Proposition}

\newtheorem{Theorem}[theorem]{Theorem}

\newtheorem{thmx}{Theorem}

\newtheorem*{thm}{\header}
\newcommand{\header}{Theorem}
\newenvironment{prevthm}[1]
 {\renewcommand\header{Theorem #1}\begin{thm}}
 {\end{thm}}

\theoremstyle{definition}

\newtheorem*{example}{Example}

\theoremstyle{remark}

\newtheorem*{remark}{Remark}
\newtheorem*{remarks}{Remarks}
\newtheorem*{Induced representation}{Induced representation}
\newtheorem*{Change of coefficients}{Change of coefficients}

\let\mge\gg
\newcommand{\zz}{\Z} 
\renewcommand{\qq}{\Q} 
\newcommand{\hh}{\mathbb{H}} 

\newcommand{\onto}{\twoheadrightarrow}

\newcommand{\CD}{CD}
\newcommand{\W}{W}
\newcommand{\M}{\mathcal{M}}

\newcommand{\ub}{\overline{\gb}}
\newcommand{\lb}{\underline{\gb}}
\newcommand{\nb}{\gb}

\DeclareMathOperator{\tr}{tr}
\DeclareMathOperator{\logtor}{logtor}

\DeclareMathOperator{\Stab}{Stab}
\DeclareMathOperator{\Out}{Out}
\DeclareMathOperator{\Ore}{Ore}

\DeclareMathOperator{\Lk}{Lk}

\DeclareMathOperator{\St}{St}
\DeclareMathOperator{\rk}{rk}
\DeclareMathOperator{\Th}{Th}

\def\R{\mathbb{R}}
\def\Z{\mathbb{Z}}

\def\Q{\mathbb{Q}}

\newcommand{\cn}{\mathcal{N}}
\newcommand{\cu}{\mathcal{U}}
\newcommand{\cl}{\mathcal{L}}
\newcommand{\co}{\mathcal{O}}

\newcommand{\rr}{\mathbb{R}} 
\newcommand{\nn}{\mathbb{N}}

\newcommand{\gb}{\beta} 
\newcommand{\gd}{\delta} 
\newcommand{\gD}{\Delta} 

\newcommand{\gf}{\varphi} 
\newcommand{\gF}{\Phi}

\newcommand{\gs}{\sigma} 
\newcommand{\gS}{\Sigma} 

\newcommand{\gt}{\tau}

\def\gG{\Gamma}
\def\gg{\gamma}
\renewcommand{\ge}{\epsilon} 

\newcommand{\F}{\mathbb{F} }
\newcommand{\Fp}{\F_{p}} 

\newcommand{\Cone}[1]{K_{{#1}}}
\newcommand{\hc}[1]{h\Cone{#1}}

\newcommand{\phc}[1]{\pi_{1}(\hc{#1})}

\newcommand{\Lhc}[2]{\hc{#2}^{#1}}
\newcommand{\pLhc}[2]{\pi_{1}(\Lhc{#1}{#2})}

\renewcommand{\d}{\partial} 

\DeclarePairedDelimiter\abs{\lvert}{\rvert}

\newcommand{\into}{\hookrightarrow}

\title{Homology growth, hyperbolization, and fibering}

\author{Grigori Avramidi}
\address{Max Planck Institute for Mathematics\\
Bonn\\
Germany, 53111}
\email{gavramidi@mpim-bonn.mpg.de}

\author{Boris Okun}
\address{Department of Mathematical Sciences\\
University of Wisconsin--Milwaukee\\
Milwaukee, WI 53201}
\email{okun@uwm.edu}

\author{Kevin Schreve}
\address{Department of Mathematics\\Louisiana State University\\Lockett Hall\\Baton Rouge, LA 70806}
\email{kschreve@lsu.edu}

\begin{document}\begin{abstract}
    We introduce a hyperbolic reflection group trick which builds closed aspherical manifolds out of compact ones and preserves hyperbolicity, residual finiteness, and---for almost all primes $p$---$\Fp$-homology growth above the middle dimension.
    We use this trick, embedding theory and manifold topology to construct Gromov hyperbolic $7$-manifolds that do not virtually fiber over a circle out of graph products of large finite groups.
\end{abstract}

\maketitle

\section{Introduction}

By the Chern--Gauss--Bonnet theorem, an even dimensional hyperbolic manifold always has non-zero Euler characteristic.
One geometric consequence of this is that such manifolds cannot fiber over the circle.
On the other hand, all odd-dimensional manifolds have vanishing Euler characteristic, so, at least in principle, odd-dimensional hyperbolic manifolds can fiber.
A remarkable discovery of the last few decades is that the later possibility is actually realized in dimension three: All closed, hyperbolic $3$-manifolds have a finite cover that fibers over a circle \cite{a13}.
There has been some recent progress towards determining whether or not some analogue of this phenomenon persists in higher dimensions.
In a geometric direction, Italiano, Martelli, and Migliorini constructed the first finite volume hyperbolic $5$-manifolds that fiber over a circle \cite{imm22}.
These examples are non-compact and no closed $5$-dimensional hyperbolic manifolds are known to fiber, see however \cite{fujiwara21} for an example which is closed nonpositively curved with isolated flats.
In a more algebraic direction, Kielak \cite{k20a} and Fisher \cite{f21} showed that for a large class of groups, the existence of a virtual $\F$-homological fibering is controlled by vanishing of certain skew field Betti numbers.
For $\F=\Q$ these Betti numbers are the $L^2$-Betti numbers, and for general fields they have an interpretation as a measure of $\F$-homology growth of finite covers.
But, there is a curious\footnote{Curious, because the fiber of a hypothetical fibration would be a closed, aspherical manifold with infinite $\Out(\pi_1)$ and no $\zz^2$-subgroups.
There are no known examples of such manifolds (note that the fibers in [33] have $\zz^2$-subgroups).}
dearth of closed, odd dimensional examples that do not virtually fiber, even if one passes from the hyperbolic to the more flexible Gromov hyperbolic setting.
The goal of the present paper is to address this.
We use $\Fp$-homology growth (for odd $p$) to prove:
\begin{thmx}\label{nofibertheorem}
    There exists a closed, odd-dimensional, aspherical manifold $\M$ with word hyperbolic fundamental group that does not virtually fiber over a circle.
\end{thmx}

It is generally difficult to construct exotic, high dimensional, closed aspherical manifolds with word hyperbolic fundamental group and this paper is no exception; we can't produce any examples above dimension $7$.
We have two (very similar) constructions of such $\M$.
The first is conceptually simpler, but has the downside of not producing anything explicit (we don't know if the non-fibering example is in dimension $5$ or $7$).
The second uses a few more tools, but has the advantage of producing explicit $7$-dimensional examples.

We will refer to closed, aspherical manifolds with word hyperbolic fundamental groups as \emph{Gromov hyperbolic manifolds}, and from now on will say such groups are \emph{hyperbolic}.
In dimension three they are precisely the closed manifolds that can be given a hyperbolic structure, while in higher dimensions they form a larger class that includes all the locally CAT$(-1)$ manifolds.
Whether every Gromov hyperbolic manifold has a locally CAT$(-1)$ metric seems to be an interesting open question.
We do not know the answer for our examples.

\subsection*{Relation to previous non-fibering results}

Higher rank, irreducible, locally symmetric spaces do not virtually fiber over a circle because their fundamental groups do not surject onto $\zz$ by the Margulis normal subgroup theorem.
This gives many nonpositively curved (but no negatively curved) examples of closed aspherical manifolds that do not fiber (e.g.
the $5$-manifolds obtained as finite volume quotients of $SL(3,\R)/SO(3)$).

Another way to see that many of these groups do not surject onto $\zz$ is to observe that they have Kazhdan's Property (T).
The fundamental groups of some even (but not odd) dimensional, negatively curved, locally symmetric spaces also have this property.
Moreover, there are additional random constructions of hyperbolic groups with Property (T) \cite{z03} and also less random ones \cite{lmw19}, but we are not aware of any that produce fundamental groups of odd dimensional, closed, aspherical manifolds.

In a more combinatorial direction (more relevant for the present paper), in \cite{aos21} we constructed closed, locally CAT(0) $n$-manifolds $M_0$ that do not virtually fiber in all odd dimensions $n\geq 7$.
The fundamental groups of these manifolds are finite index subgroups of right-angled Coxeter groups and the fibering obstruction is fast $\Fp$-homology growth.
The manifolds $M_0$ are not Gromov hyperbolic because their fundamental groups contain $\zz^2$.

Theorem \ref{nofibertheorem} provides the first examples of odd dimensional, Gromov hyperbolic manifolds that do not virtually fiber.
Moreover, we can arrange $\pi_1(\M)$ to be special in Haglund and Wise's sense \cite{hw08}*{Section 3}.
In particular, this implies $\pi_1(\M)$ is residually finite, does not have Property (T), and in fact has finite index subgroups with arbitrarily large first Betti number.

\subsection*{Homology growth and virtual fibering}

In the examples we produce for Theorem \ref{nofibertheorem} the mechanism obstructing virtual fibering is, again, fast homology growth.
To keep track of it, fix a field $\F$, look at the infimum of normalized $\F$-Betti numbers of all finite covers of $\M' \to \M$ (normalized by the degree $\abs{\M'\to \M}$ of the cover)
\[
\beta^{\inf}_{k}(\M;\F):=\inf_{\substack{\M'\to \M,\\
\abs{\M'\to \M}<\infty}} \frac{b_{k}(\M';\F)}{ \abs{\M'\to \M}}
\]
and note that virtually fibering $\M$ over a circle would give covers $\M'\to \M$ of arbitrarily large degree whose $\F$-Betti numbers $b_{k}(\M';\F)$ are bounded by a uniform constant (the sum of the $\F$-Betti numbers of the fiber), which would imply $\beta_{k}^{\inf}(\M;\F)=0$.
So, positivity of this number obstructs virtual fibering.

\subsection*{Relation to $L^2$-Betti numbers}

For a finite complex $X$ with residually finite fundamental group and coefficient field $\F =\Q$, the numbers $\beta^{\inf}_k(X;\Q)$ are closely related to the analytically defined $L^2$-Betti numbers $b_{k}^{(2)}(X)$.
In that situation, the proof of L\"uck's approximation theorem implies (see Theorem \ref{quantlueck})
\[
b_k^{(2)}(X)=\sup_{X'\to X}\frac{\beta_k^{\inf}(X';\Q)}{\abs{X'\to X}},
\]
where the supremum is taken over all finite covers.
\subsection*{Pinching homology growth}

For other fields $\F $, we replace $\Q$ by $\F $ on the right hand side, take the resulting quantity
\[
\lb_{k}(X;\F) := \sup_{X'\to X}\left(\inf_{X''\to X'} \frac{b_k(X'';\F)}{\abs{X''\to X}}\right),
\]
where the sup is over finite covers $X'\to X$ and the inf is over further finite covers $X''\to X'$, and call it the \emph{lower $\F $-homology growth}.
It follows directly from this max-min definition that the quantity $\lb$ is multiplicative in finite covers, and that it obstructs virtual fibering.
One also has the multiplicative quantity $\ub$ obtained by interchanging the roles of inf and sup in the definition, which we call the \emph{upper $\F $-homology growth}.
For $\F =\Q$ it gives the same answer (by Theorem \ref{quantlueck}) but in general we only know that $\beta^{\inf}\leq \lb\leq \ub$, although we are not aware of any examples in which the second inequality is strict.
\begin{question}
    Is there a finite complex $X$ with $\lb_k(X;\Fp ) \neq \ub_k(X;\Fp )$ for some prime $p$?
\end{question}

\subsection*{Special fundamental groups}

In \cite{aos21} we used specific properties of right-angled Coxeter groups to help compute $\Fp$-homology growth for the manifolds $M_0$.
The fundamental groups of the manifolds $\M$ we produce for Theorem \ref{nofibertheorem} are not commensurable to right-angled Coxeter groups, but they are special, hence embed in right-angled Artin groups, and this turns out to be good enough to estimate the homology of covers of $\M$.
For such groups, the numbers $\beta^{\inf}$ have another, more algebraic, ``skew field'' interpretation that is convenient for doing Mayer--Vietoris computations and, in particular, shows that they are integers.
Namely, the group ring $\F \pi_1 (\M)$ embeds in a nice\footnote{Obtained by picking a bi-invariant order on the group and taking the division closure of the group ring in the Malcev--Neumann skew field of formal power series with well-ordered support.}
skew field $D_{\F \pi_1 (\M)}$, one can do all the Mayer--Vietoris arguments for homology with coefficients in that skew field, and the Betti numbers $b_{k}(\M;D_{\F\pi_1(\M)})$ obtained from this homology coincide with the infimum of the normalized $\F$-Betti numbers.
Moreover, these algebraically defined skew field Betti numbers are multiplicative in finite covers, which implies \emph{for subgroups of right-angled Artin groups} that we don't need to take the sup in the definition of $\lb$, i.e. 
\[
\lb_{k}(\M;\F)=\beta^{\inf}_k(\M;\F)=b_{k}(\M;D_{\F\pi_1(\M)})\in\Z.
\]

So, the quantity $\lb$ can be thought of as a multiplicative extension of the skew field Betti number from this special setting to situations where (nice enough) skew fields do not exist.
Using skew fields to study $L^2$-Betti numbers (and vice-versa) originated in work of Linnell \cite{l93} and has been recently developed by Henneke--Kielak \cite{hk21}, Jaikin-Zapirain \cites{j19,j21}, and others.
Though we haven't seen the above equality before in the literature, we prove it by combining a number of previously known results, see Section \ref{s:skewfields} for more details.
\begin{remark}
    In more general situations, the inequality $\beta^{\inf}_k(X;\F ) \leq \lb_{k}(X;\F )$ is often strict.
    For instance, a wedge sum $X$ of two hyperbolic homology $3$-spheres has $\lb_1(X;\Q)=b^{(2)}_{1}(X) =1$ but $b_1(X;\Q) = 0$, and hence $\beta^{\inf}_1(X;\Q) = 0$.
\end{remark}

\subsection*{Homology growth without chains}

Previous works on homology growth of a space $X$ usually consider descending chains of subgroups $\pi_1(X) = G > G_1 > G_2 > \dots $ and the homology of the corresponding covers $X_i$.
Typically, additional assumptions are placed on the subgroups $G_i$, such as normality, $\bigcap_i G_i = 1$, or $[G:G_i]$ being some prime power.
It was surprising to us that $\lb_k(X)$ and $\ub_k(X)$ worked just as well while avoiding many of the headaches that come with using chains; with $\F_p$-coefficients it is generally unknown whether the normalized homology growth of such a chain has a limit, or depends on the choice of chain, etc.
Even if $G$ is a subgroup of a right-angled Artin group, the identification $\beta^{\inf}_k(X;\F)=b_{k}(X;D_{\F\pi_1(X)})$ has no analogue for residual chains of normal finite index subgroups, though some approximation results are known, see e.g.
\cite{blls14}*{Theorem 4.3}.

\subsection*{The only $\F $-homological, virtual fibering obstructions in a special setting}

A special case of a result of Fisher \cite{f21}, building on work of Kielak \cite{k20a}, shows---for a finite aspherical complex $X$ whose fundamental group embeds in a right-angled Artin group---that if the lower $\F $-homology growth (he uses the skew field definition) in dimensions $\leq k$ vanishes then $X$ has a finite cover $X'$ which maps to a circle with $FP_{k}(\F)$ homotopy fiber.
So, non-vanishing of the lower $\F $-homology growth is the \emph{only} $\F$-homological obstruction in dimensions $\leq k$ to virtual fibering in this setting.
We observe that non-vanishing of the upper $\F $-homology growth is also an $\F $-homological virtual fibering obstruction and hence Fisher's result leads to
\begin{thmx}\label{p:infsup}
    \label{only if}
    If $X$ is a finite, aspherical complex whose fundamental group embeds in a right-angled Artin group, then $\lb_{\leq k}(X;\F )=0$ if and only if $\ub_{\leq k}(X;\F )=0$.
\end{thmx}
This theorem may seem quite formal, but it is useful in practice because it relates two numbers $\lb$ and $\ub$ that have very different advantages.
\begin{itemize}
    \item If the fundamental group embeds in a right-angled Artin group, then $\lb_k(X;\Fp)$ are integers, and as a consequence of this integrality differ from the $L^2$-Betti numbers $b^{(2)}_k(X)=\lb_k(X;\Q)$ at only finitely many exceptional primes, while
    \item vanishing of $\ub_k(X;\F )$ controls homology of virtually all finite covers: it says that for any $\gd>0$ there is a ``$\gd$-good'' finite cover $X'\to X$ such that \emph{all} further finite covers $X''\to X'\to X$ have normalized Betti number $\frac{b_k(X'';\F )}{\abs{X''\to X}}$ bounded by $\gd$.
\end{itemize}
\begin{remark}
    It is tempting to wonder if some of these phenomena hold for more general fundamental groups: Higher rank, irreducible locally symmetric spaces do not fiber.
    Do they have fast $\Fp$-homology growth (in either the $\lb$ or the $\ub$ sense) for some prime $p$?
\end{remark}

\subsection*{The Singer conjecture and the $\F$-Singer property}

For closed aspherical $n$-manifolds $M^n$, the Singer conjecture predicts that the $L^2$-Betti numbers $b^{(2)}(M)$ vanish outside the middle ($=n/2$) dimension, and in particular that all the $L^2$-Betti numbers of a closed, odd-dimensional, aspherical manifold vanish.
This conjecture suggests that rational homology growth shouldn't give virtual fibering obstructions in odd dimensions.
But, the situation is different for $\Fp$-homology growth.
Let us single out the following homology growth vanishing property for $n$-manifolds
\begin{itemize}
    \item \emph{upper $\F$-Singer property}: $\ub_k(M^n;\F)=0$ for $k> n/2$,
\end{itemize}
and similarly define the {\emph{lower $\F$-Singer property}} using lower homology growth.
For closed manifolds these properties imply by Poincar\'e duality that upper (or lower) homology growth is concentrated in the middle dimension.
If $M$ is a closed aspherical manifold with special fundamental group then, by Theorem \ref{only if} and Poincar\'e duality, the upper and lower properties are equivalent and in this situation we will refer to both as the $\F$-Singer property.

In \cite{aos21} we built---for every odd prime $p$---closed, locally CAT(0) manifolds with special fundamental groups that do not satisfy the $\Fp$-Singer property in all odd dimensions $\geq 7$ (and all even dimensions $\geq 14$).
Much of the mathematical content of the present paper amounts to producing Gromov hyperbolic ones.
Our main result is
\begin{thmx}\label{maintheorem}
    \leavevmode
    \begin{enumerate}
        \item For any odd prime $p$, there is a closed, aspherical, $n$-manifold $\M^n$ of dimension either $n=5$ or $n=7$ with special hyperbolic fundamental group such that $\lb_k(\M;\Fp)>0$ for some $k$.
        \item For large primes, such $7$-manifolds definitively exist.
    \end{enumerate}
\end{thmx}

Proving it turned out to be more subtle (and interesting) than we initially expected, because applying strict hyperbolization procedures (e.g.
the Charney--Davis strict hyperbolization) directly to our previous examples kills the golden goose: the homology cycles responsible for fast growth get hyperbolized in the process and, as a result, instead of getting a linear number of homology cycles in covers one gets a sublinear number of more complicated cycles.
To get an idea of how hyperbolization can destroy $L^2$-Betti numbers, note that strict hyperbolization applied to a $2$-dimensional cube complex (e.g.
the complex $8\times 8$) amounts to connect-summing each square with a higher genus surface.
A Mayer--Vietoris argument shows that this process has the same effect on the second $L^2$-Betti number as removing the squares, leaving a $1$-dimensional complex with vanishing second $L^2$-Betti number.
There are a number of elements that go into our construction of $\M$, which we now describe.
We shall see how ensuring that $\M$ is Gromov hyperbolic prevents us from making examples in dimensions $>7$.

\subsection*{Graph products}

In the construction, the starring role is played by graph products $G_L$ of groups modeled on a flag complex $L$.
For each vertex $v$ of $L$, pick a group $G_v$ and define
\[
G_L:=\bigast_v G_v/\langle [g,g']=1\text{ if } g\in G_v, g'\in G_{v'} \text{ and } v \text{ is adjacent to } v'\rangle.
\]

Graph products of $\Z/2$'s are right-angled Coxeter groups $W_L$, graph products of $\Z$'s are right-angled Artin groups $A_L$, but what we use are graph products of $\Z/m$'s for large $m$.
They virtually embed in the corresponding right-angled Artin group $A_L$, and can either be thought of as deformations of $W_L$ whose homology growth can be estimated, or deformations of $A_L$ that have a chance of being hyperbolic.

\subsection*{Hyperbolicity}

A graph product of finite groups $G_L$ acts properly, cocompactly on a CAT(0) cubical complex of dimension $\dim L+1$.
If the triangulation of $L$ has no empty squares, then this complex can be given a $G_L$-invariant CAT$(-1)$ metric, so in that case $G_L$ is CAT$(-1)$ and hence hyperbolic.

\subsection*{Homology}

In \cite{aos21} we computed the $\F $-homology growth of right-angled Artin groups.
The computation given there works identically for $\lb$ and $\ub$ and shows
\[
\lb_k(A_L;\F )=\tilde b_{k-1}(L;\F )=\ub_k(A_L;\F ),
\]
where $\tilde b$ denotes the reduced Betti number of $L$.
In particular, $\lb$ and $\ub$ agree for right-angled Artin groups and differ from the usual $L^2$-Betti number at finitely many primes determined by the topology of the underlying complex $L$.

A graph product $G_L$ of large $\zz/m$'s has approximately the same homology growth as the corresponding right-angled Artin group $A_L$, i.e. 
\[
\lb_k(G_L;\F )\sim \lb_k(A_L;\F ),
\]
where the error is on the order of $\abs{L}/m$ (Corollary \ref{c:modpl2}.)

Somewhat surprisingly, the argument in this paper is conceptually simpler; our computation relies on cell counting, whereas in \cite{aos21} we needed to compute some homology.
Of course, in this paper we don't get (or need) an exact computation of $\lb_k(G_L;\F )$.
\subsection*{Embedding theory}

One can construct non-compact aspherical $n$-manifolds that have a specified fundamental group by embedding that group in a right-angled Coxeter group of the form $W_{S^{n-1}}$, as the groups $W_{S^{n-1}}$ all act properly on (topological) $\rr^n$.

This method works well for graph products of finite groups $G_L$ because by commensurability results of \cite{js01} they virtually embed in the right-angled Coxeter group $W_{OL}$, where $OL$ is a more complicated flag complex called the \emph{octahedralization} of $L$, obtained from $L$ by doubling the set of vertices and replacing each $k$-simplex $v_0*\dots*v_k$ by $2^{k+1}$ $k$-simplices $v_0^{\pm}*\dots*v_k^{\pm}$.
With Davis, in \cite{ados16} we determined, for a $d$-dimensional ($d\neq 2$) flag complex $L$, that $OL$ embeds as a full subcomplex of some flag triangulation of $S^{2d}$ if and only if $H_d(L;\F_{2})=0$.
This implies that $W_{OL}$ is a subgroup of $W_{S^{2d}}$.
In summary, if $d\neq 2$ and $b_d(L;\F_{2})=0$, then $G_L$ virtually embeds in some $W_{S^{2d}}$.

\subsection*{Construction of a $7$-manifold with boundary}

We now exploit the fact that embedding theory for $OL$ only depends on $\F_{2}$-homology of $L$ while $\Fp$-homology growth of $G_L$ is sensitive to the prime $p$.
The $3$-dimensional Moore space $L=D^3\cup_p S^2,$ ($p$ odd) has a flag no-square triangulation by \cite{ps09}, and for any such triangulation it follows from what we have said that the graph product $G_L$ of large $\zz/m$'s
\begin{itemize}
    \item is hyperbolic,
    \item has $\lb_{4}(G_L;\Fp)>0$,
    \item virtually embeds in some $W_{S^6}$, and hence
    \item has a finite index subgroup $\Gamma$ that acts properly and freely on $\rr^7$.
\end{itemize}
The quotient manifold $\rr^7/\Gamma$ has finite type but is not compact.
What saves us is that the construction of the manifold also produces, as a byproduct, a codimension three spine.
This codimension three lets us compactify a regular neighborhood of this spine using a $\pi$-$\pi$ version of Siebenmann's thesis (Theorem \ref{t:pi-pi}.) The upshot is a compact aspherical $7$-manifold $(N,\d N)$ with hyperbolic fundamental group and $\Fp$-homology growth in dimension four.
It remains to produce a closed aspherical $7$-manifold with these properties.

\subsection*{On dimensions}

Why did we start with a $3$-complex? On one hand, \foreignlanguage{vietnamese}{Nguy\~\ecircumflex{}n Phan} and the first author recently constructed examples \cite{ap21} showing that the assumption $d\neq 2$ in the embedding theory cannot, in general, be avoided: there are flag $2$-complexes $L$ with $b_2(L;\F_2)=0$ whose octahedralizations do not PL embed in $S^4$.
On the other hand, flag no-square triangulations of arbitrary polyhedra are only known to exist in dimensions at most three \cite{ps09}.
While there are constructions \cites{js03, o13} of flag no-square $d$-complexes $L$ in all dimensions $d$, we do not know how to arrange these to have $b_d(L;\F_{2})=0$ and $b_d(L;\Fp)\neq 0$.
So, we only know how to make the method work in one dimension!
\begin{question}
    Are there flag no-square $d$-complexes $L$ for $d>3$ which have $b_d(L;\F_{2})=0$ and $b_d(L;\Fp)\neq 0$?
\end{question}

\subsection*{A hyperbolic reflection group trick}

To obtain a closed manifold we do a combination of the Davis reflection group trick and the Charney--Davis strict hyperbolization.
This hyperbolic reflection group trick works in any dimension, and may be of independent interest, because it preserves hyperbolicity, residual finiteness and other pleasant features.
The input to this trick is a compact $n$-manifold $(N,\d N)$ with a flag triangulation $\d$ of the boundary $\d N$, and a choice of Charney--Davis hyperbolized $n$-cube $\CD^{n}$.
The output is a closed manifold $hP_{\d}^{N}$, which we will sometimes simply denote by $\M$, obtained as follows.
First, build the right-angled Coxeter group corresponding to the flag triangulation $\d$, and take the commutator quotient of the corresponding Davis complex.
This results in a finite cube complex, which is a manifold except at finitely many singular points, which have links isomorphic to $\d$.
Now replace the cubes of this cube complex by $\CD^{n}$, and then replace small neighborhoods of the singular points by copies of $N$.

We summarize some of the properties of the hyperbolic reflection group trick:
\begin{thmx}\label{t:hypdavistrick}
    Given a compact $n$-manifold $(N,\d N)$ with a flag triangulation $\d$ of the boundary, the hyperbolic reflection group trick produces a closed $n$-manifold $\M=hP^N_{\d}$ satisfying:
    \begin{enumerate}
        \item $N$ is a retract of $\M$,
        \item If $N$ is aspherical then $\M$ is aspherical,
        \item If $N$ is $\F $-aspherical\footnote{A space is \emph{$\F $-aspherical} if its universal cover has the same $\F $-homology as a point.}
        then $\M$ is $\F $-aspherical,
        \item If $\pi_1(N)$ is hyperbolic then $\pi_1(\M)$ is hyperbolic,
        \item If $\pi_1(N)$ is virtually special hyperbolic then $\pi_1(\M)$ is virtually special hyperbolic,
        \item If $\pi_1(N)$ is residually finite then $\pi_1(\M)$ is residually finite.
    \end{enumerate}
\end{thmx}
Note that for $n>5$ any flag triangulation of $\d N$ always has empty squares, so the usual reflection group trick never produces a Gromov hyperbolic manifold.

The first part of Theorem \ref{maintheorem} follows by applying the hyperbolic reflection group trick to our seed manifold $N^{7}$ with hyperbolic fundamental group and $\lb_4(N; \F_p) \ne 0$.
The resulting manifold $\M$ can be cut along walls down to Charney--Davis pieces and copies of $N$.
The walls have virtually special fundamental groups which lets us relate their homology growth in all finite covers to homology growth in the restricted class of finite covers induced from $\M$.
This lets us do a Mayer--Vietoris type argument and show that either $\M$ has non-vanishing $\lb$ or a lower odd-dimensional closed, locally CAT$(-1)$ manifold appearing as an intersection of walls in the construction has non-vanishing $\ub$ (and hence $\lb$ by Theorem \ref{only if}).
In either case, we get a non-fibering example.

For the second part of Theorem \ref{maintheorem} we need better control on homology growth of the walls.
We can achieve it by restricting our triangulation $\d$ to be a barycentric subdivision of a triangulation of the boundary.
This lets us arrange so that our walls are themselves Charney--Davis hyperbolizations from a finite set of cubical complexes that depends only on the dimension $n$.
Then choosing an appropriate $\CD^{n}$ and applying recent work of Ontaneda shows that these walls admit Riemannian metrics of sufficiently pinched negative curvature, which by a result of Donnelly--Xavier implies that $L^{2}$-Betti numbers of the walls vanish outside of the two middle dimensions.
Since the fundamental groups of these walls are virtually special, their $\lb$ differs from $b^{(2)}$ at only finitely many primes, and the cutting argument proves the following theorem.
\begin{thmx}\label{t:barytrick}
    For each dimension $n$ there is a choice of Charney--Davis piece $CD^n$ and a corresponding finite collection of exceptional primes $S_n$, such that for any compact $n$-manifold with boundary $(N,\d N)$ and any triangulation $\d$ which is a barycentric subdivision of a triangulation of the boundary, the result of the hyperbolic reflection group trick $\M=hP_{\d}^{N}$ satisfies the following inequalities for $k>n/2$:
    \begin{enumerate}
        \item\label{i:b2}
        $b_k^{(2)}(N)\leq b_k^{(2)}(\M)$,

        \item\label{i:lbQ}
        $\lb_{k}(N;\Q)\leq \lb_k(\M;\Q)$ and $\ub_{k}(N;\Q)\leq \ub_k(\M;\Q)$,

        \item\label{i:lb}
        $\lb_k(N;\Fp)\leq \lb_k(\M;\Fp)$ and $\ub_k(N;\Fp)\leq \ub_k(\M;\Fp)$ for $p \notin S_n$.
    \end{enumerate}
\end{thmx}

The second part of Theorem \ref{maintheorem} follows from Theorem \ref{t:barytrick} applied to our seed manifold $N^{7}$.
\begin{remark}
    When the fundamental group of the input $\pi_1(N)$ is virtually special hyperbolic, then the fundamental group of the output $\pi_1(\M)$ is, as well, so we have access to the skew field $D_{\F \pi_1 (\M)}$.
    Then, the entire cutting argument can be carried out using this ambient skew field (and its sub-skew fields corresponding to subgroups of $\pi_1(\M)$) and leads to an alternate proof of the first inequality in Theorem \ref{t:barytrick}(3).
\end{remark}

As explained in \cite{os16}, the (usual) reflection group trick implies that the Singer conjecture is equivalent to the statement that $L^{2}$-Betti numbers of a compact aspherical manifold, possibly with boundary, vanish above the middle dimension.
The hyperbolic reflection group trick recovers this and also shows that the Singer conjecture for Gromov hyperbolic manifolds is equivalent to the statement that $L^2$-Betti numbers of compact aspherical manifolds with hyperbolic fundamental groups vanish above the middle dimension.

\subsection*{Rationally aspherical manifolds}

When $L$ is a flag triangulation of the $3$-sphere the graph product of large finite groups $G_{S^3}$ has $b_4^{(2)}(G_{S^3})>0$.
In this case, the van Kampen embedding theory method does not produce a $7$-dimensional thickening since $b_3(S^3;\F _2)\neq 0$, and we suspect that no such thickening exists.
However, since finite index torsion free subgroups $\Gamma$ of $G_{S^3}$ are duality groups \cites{dm02, ddjmo10}, we can use the rational homotopy method from \cite{a18} to at least produce a rational thickening, i.e. a rationally aspherical, compact $7$-manifold with boundary $(N,\d N)$, non-vanishing $b_4^{(2)}$ and fundamental group $\Gamma$.

Moreover, if we start with a flag no-square triangulation of $S^3$ (which do exist) then the resulting group will be hyperbolic.
Feeding this seed manifold $N^7$ into the hyperbolic reflection group trick (with barycentrically subdivided boundary) gives
\begin{thmx}\label{rational}
    There is a closed, rationally aspherical $7$-manifold $\M$ with special hyperbolic fundamental group and $b_4^{(2)}(\M)\neq 0$.
\end{thmx}
Theorems \ref{maintheorem} and \ref{rational} are quite different both in the input used to obtain the examples and in their conclusions.
The first produces genuinely aspherical manifolds, while the second produces examples in which the actual $L^2$-Betti numbers are not concentrated in the middle dimension.
Nonetheless, forgetting some of the information they provide, we can put them in a single context.
We have obtained for $\F =\Q$ or $\Fp$ for \emph{odd} primes $p$ an example of a closed, $\F $-aspherical manifold that does not satisfy the $\F $-Singer property.
The remaining case is $p=2$.
\begin{question}

    Does every $\F _2$-aspherical manifold with residually finite fundamental group satisfy the (upper or lower) $\F _2$-Singer property?
\end{question}
\subsection*{Acknowledgements}

We thank the referee for a very detailed report and useful comments which improved the paper, in particular they found a gap in our arguments in Section \ref{s:induction} which has now been fixed.
We thank Andrei Jaikin-Zapirain for answering multiple questions of ours concerning the skew field approach to homological growth.
Most of Section \ref{s:skewfields} came from trying to understand his answers.
The first author thanks \foreignlanguage{vietnamese}{T\^am Nguy\~\ecircumflex{}n Phan} for explaining van Kampen's approach to embedding theory and asking for an interpretation of embedding theory for octahedralizations in terms of it, which lead to Appendix \ref{a:embedding}.
We thank Sam Fisher and Igor Belegradek for helpful comments on an earlier draft, and Andrei Jaikin-Zapirain and Clara L{\"o}h for providing useful references.
The first and second authors thank the Max Planck Institut f\"ur Mathematik for its support and excellent working conditions.
The third author was supported by the NSF grant DMS-2203325. \tableofcontents
\subsection*{Plan of the paper}

Section \ref{s:homologygrowth} assembles some basic facts on homology growth.
Section \ref{s:skewfields} sets up the skew field theory we need and Section \ref{s:applications} gives consequences of this theory for homology growth and proves Theorem \ref{only if}.
We estimate homology growth of graph products in Section \ref{s:graphproducts}, thicken their classifying spaces to manifolds in Section \ref{s:thickenings}, construct closed aspherical manifolds via a hyperbolic reflection group trick in Section \ref{s:hrt} (proving Theorem~\ref{t:hypdavistrick}(1)--(4)), and show the results have virtually special fundamental groups (Theorem \ref{t:hypdavistrick}(5)) in Section \ref{s:special}.
In Section \ref{mvsection}, we derive the Mayer--Vietoris inequalities needed for our induction arguments.
Section \ref{s:induction} carries out the basic inductive cutting argument and proves Theorem \ref{nofibertheorem} and Theorem \ref{maintheorem}(1).
Section \ref{s:proofs} deals with the barycentric version of the hyperbolic reflection group trick, proves Theorem \ref{t:barytrick} and uses it to establish Theorems \ref{maintheorem}(2) and \ref{rational}.
The appendices discuss residual finiteness of the hyperbolic (and the usual) reflection group trick (proving Theorem \ref{t:hypdavistrick}(6)), the relation between $L^2$-Betti numbers and rational homology growth, and the embedding theory for octahedralizations.

\section{Upper and lower homology growth}\label{s:homologygrowth}

Our first goal is to understand the normalized Betti numbers of finite covers of a complex $X$ and how they vary as we pass to further covers.
Somewhat surprisingly, a number of basic but useful properties of this can be established by thinking of the normalized Betti numbers as a function on the partially ordered\footnote{The partial order on covers of $X$ is defined by $X''> X'$ whenever $X''$ is a cover of $X'$.}
set of all finite covers, looking at upper and lower limits over this set, and using the fact that any two finite covers have a further finite cover lying above both of them.
In this section we record the basic properties of such limits.

\subsection*{Limits over directed posets}

Let $(C, <)$ be a partially ordered set.
Suppose it is directed (for any $x, y \in C$ there exists $z \in C$ such that $x <z $ and $y < z$).
A subset $D$ is \emph{cofinal} if for any $x \in C$ there exists $y \in D$ such that $x \leq y $.
We want to define various notions of limits of bounded real-valued functions on $C$.
The basic building blocks are taking $\inf$ or $\sup$ over a subset.
There are two basic observations: smaller subsets produce smaller $\sup$ and if the function is decreasing then $\inf$ can be computed over any cofinal subset.

Applying $\sup$ to tails (subsets of the form $C_{\geq x}=\{y \in C \mid y\geq x \}$) defines an operation on bounded functions $f \mapsto f^{\sup}$:
\[
f^{\sup}(x)= \sup_{C_{\geq x}} f.
\]
This converts any function to a decreasing one.

Define $f^{\inf}:=-(-f)^{\sup}$.
The following is immediate from the observations:
\begin{lemma}
    $ \overline f:= (f^{\sup})^{\inf } $ and $ \underline f := (f^{\inf})^{\sup}$ are constant functions.
\end{lemma}
We will call $\overline f$ and $ \underline f$ the upper and lower limits of $f$.

We collect in the following lemma all the properties of upper and lower limits that we need.
They will be primarily used in obtaining the Mayer--Vietoris inequalities in Section \ref{mvsection}.
\begin{lemma}\label{l:posetineq}
    Let $f, g$ be bounded functions.
    \begin{enumerate}
        \item If $f \leq g$, then $\underline f\leq\underline g$ and $ \overline f \leq \overline g$.
        \item(Almost additivity)\label{i:sub}
        \[
        \underline f+\underline g\leq\underline{f+g} \leq \underline f+\overline g\leq\overline {f+g} \leq \overline{f} +\overline{g}.
        \]

        \item(Restriction)\label{i:cofinal}
        Let $D$ be a cofinal subset of $C$.
        Then
        \[
        \underline f\leq\underline{f|_{D}}\leq \overline{f|_{D}}\leq \overline f.
        \]
        \item Let $T$ be a tail of $C$.
        Then
        \[
        \underline f=\underline{f|_{T}}\leq \overline{f|_{T}}= \overline f.
        \]
    \end{enumerate}
\end{lemma}
\begin{proof}
    (1) is immediate.
    The last inequality in (2) follows from general subadditivity of $\sup$ and additivity of $\inf$ for decreasing functions.
    Using $\underline f=-\overline{(-f)}$ the following trick then proves the third inequality in (2):
    \[
    \underline f + \overline g=\overline{f+g-f}-\overline{(-f)}\leq\overline{f+g}+\overline{(-f)}-\overline{(-f)}=\overline{f+g}.
    \]
    Note that setting $g=0$ in this inequality gives $\underline f\leq\overline f$, which justifies the upper and lower notation and proves the middle inequalities in (3) and (4).

    Since for a subset $D$ of $C$, $D_{\geq x} = C_{\geq x} \cap D$, the basic observations give the last inequality in (3).
    If $T$ is itself a tail of $C$, then its tails are tails of $C$, hence
    \[
    ( f|_{T})^{\sup}= f^{\sup}|_{T},
    \]
    and we obtain the last equality in (4).
    The remaining (in)equalities are obtained from these by flipping signs.
\end{proof}

\subsection*{Normalized Betti numbers as a function on a set of covers}

For a complex $X$, let $C_{X}$ denote the poset of finite covers of $X$.
The normalized $k$-th $\F$-Betti numbers of such covers define a function on this poset, which we will denote by $\nb_k(X; \F)$, i.e. if $X' \to X$ is a finite cover, then
\[
\nb_k(X; \F)(X'):=\frac{b_k(X';\F )}{\abs{X'\to X}}.
\]
For much of this section, $k$ and $\F $ will be unimportant, and then we will omit one (or both) of them from the notation.
The function $\nb_{k}( X; \F)$ is bounded by the number of $k$-cells in $X$.

The \emph{upper} and \emph{lower $\F $-homology growth}, $\ub_{k}(X; \F)$ and $\lb_{k}(X; \F)$ are the upper and lower limits over $C_{X}$ of this function, more explicitly:
\begin{align*}
    \ub_{k}(X;\F )&:=\inf_{X'\to X}\left(\sup_{X''\to X'} \frac{b_k(X'';\F )}{\abs{X''\to X}}\right),\\
    \lb_{k}(X;\F )&:=\sup_{X'\to X}\left(\inf_{X''\to X'} \frac{b_k(X'';\F )}{\abs{X''\to X}}\right).
\end{align*}
More generally, given a map $h: X \to Y$, and a finite cover $\pi: Y' \rightarrow Y$, the pullback
\[
h^*(Y') = \{(x,y'): h(x) = \pi(y')\} \subset X \times Y'
\]
is a finite cover of $X$ of the same degree.
We define the \emph{restricted} homology growths of $X$, $\ub_{k}^{Y}(X; \F)$ and $\lb_{k}^{Y}(X; \F)$, by taking the above limits over the subset of covers pulled back from $Y$.
Note that for $Y=X$, $\ub^{X}(X)=\ub(X)$ and $\lb^{X}(X)=\lb(X)$.

\subsection*{Rational homology growth and $\gd$-good covers}

We next describe a $\gd$-pinching theorem for normalized rational Betti numbers.
It is a consequence of the proof of L\"uck's approximation theorem in \cite{l94} and can be thought of as a quantitative variant of that theorem.
Other quantitative versions of L\"uck approximation also appear in \cite{cw03} and \cite{lu22}.

We need the following basic linear algebraic lemma.
\begin{lemma}[\cite{l94}*{Theorem 3.4(1)}]\label{smalleigs}
    Suppose $\gD$ is an $N\times N$ matrix with integer entries.
    Let $N_{\ge}$ be the number of eigenvalues $\lambda$ with $\abs{\lambda}\in(0,\ge]$, counted with multiplicity.
    If $\ge<1$ then
    \[
    \frac{N_{\ge}}{N} \leq \frac{\log\abs{\gD}}{\log(\ge^{-1})}
    \]
\end{lemma}
\begin{proof}
    Look at the characteristic polynomial $\det(t-\gD)=t^b q(t)$ where $q(t)=\prod_{0<\mu_i\leq\ge} (t-\mu_i)\prod_{\ge<\mu_i\leq \abs{\gD}}(t-\mu_i)$.
    Since $\gD$ has integer entries, the number $q(0)$ is a non-zero \emph{integer}, hence
    \[
    1\leq \abs{q(0)} \leq \ge^{N_{\ge}}\abs{\gD}^N.
    \]
    Rearranging to $(\ge^{-1})^{N_{\ge}}\leq\abs{\gD\abs{^N$ and taking logs gives $N_{\ge}\log(\ge^{-1})\leq N\log}\gD}$.
\end{proof}
\begin{theorem}[$\gd$-pinching $\qq$-homology growth]\label{quantlueck}
    Let $X$ be a finite complex.
    Given $\gd>0$ there is a finite cover $X_{\gd}$ such that the function $\beta_k(X;\qq)$ is $\gd$-pinched above $X_{\gd}$.
    More explicitly, if $X'$ and $X''$ are finite covers of $X_{\gd}$ then
    \[
    \abs*{\frac {b_k(X';\Q)}{\abs{X'\to X} }- \frac{b_k(X'';\Q)}{\abs{X''\to X} } } \leq\gd.
    \]
\end{theorem}
\begin{proof}
    Let $N$ be the number of $k$-cells of $X$.
    There is a finite constant $D$ such that the norm of the combinatorial Laplacian $\gD'$ acting on $C_k(X';\qq)$ of finite covers of $X$ is uniformly bounded by $D$ (see \cite{l94}*{Lemma 2.5}) independent of the cover.
    Choose $0<\ge<1$ satisfying
    \[
    \ge N < \gd/2, \qquad \ge + \frac{\log D}{\log (\ge^{-1})} < \gd/N.
    \]
    Next choose $r$ so that polynomial $f(x) =(1-x/D)^{r}$ satisfies
    \[
    f(\ge)< \ge.
    \]
    Then, since $f$ is monotone decreasing on $[0,D]$,
    \[
    \chi_{0}\leq f\leq \chi_{[0,\ge]}+\ge \text{ on $[0,D]$}.
    \]
    So, for any finite cover $X'\to X$ we have
    \[
    b_{k}(X';\Q)=\tr \chi_{0}(\gD') \leq \tr f(\gD') \leq \tr (\chi_{[0,\ge]}+\ge)(\gD')= b_k(X';\Q)+N'_{\ge}+ \ge N',
    \]
    where $N'$ is the number of $k$-cells of $X'$, and $N'_{\ge}$ is the number of eigenvalues in the interval $(0,\ge]$ of the Laplacian $\gD'$ acting on $C_k(X';\qq)$.
    Since this combinatorial Laplacian has integer entries, Lemma \ref{smalleigs} implies
    \[
    N'_{\ge} \leq \frac{\log D}{\log (\ge^{-1})}N'.
    \]
    Hence,
    \[
    b_{k}(X';\Q)\leq \tr(f(\gD'))\leq b_k(X';\Q)+ \gd \abs{X'\to X}.
    \]
    Since $f$ is a polynomial, there is a radius $R$ such that the support of $f(\gD') e$ is in the $R$-neighborhood of $e$ for each cell $e$ in $X'$.

    Let $\hat X\to X$ denote the universal residually finite cover and choose its finite quotient $X_{\gd}$, so that $\hat X\to X_{\gd}$ is injective on $R$-balls.
    Then any finite cover $X'\to X_{\gd}$ is also injective on $R$-balls, hence
    \[
    \tr f(\gD') = \tr(f(\gD_{\gd}))\abs{X' \to X_{\gd}}.
    \]
    Combining this with the above inequality shows that for any such cover $X'$ the normalized Betti numbers lie in the $\gd$-interval $\left[ \frac {\tr(f(\gD_{\gd}))} {\abs{X_{\gd}\to X} } -\gd, \frac{ \tr(f(\gD_{\gd}))} {\abs{X_{\gd}\to X}} \right]$, which proves the claim.
\end{proof}
\begin{corollary}
    For any finite complex $X$ we have $\lb_*(X;\Q) = \ub_*(X;\Q)$.
\end{corollary}

If we fix lifts $\hat e$ in $\hat X$ of $k$-cells $e$ in $X$, then the injectivity on $R$-balls implies that $\tr(f(\gD_{\gd}))/\abs{X_{\gd} \to X}= \sum_{e \in X^{(k)}} \langle f(\hat \gD)\hat e,\hat e\rangle$.
As we vary the polynomial $f$, letting $r \to \infty$, the quantity on the right of this equation converges to the von Neumann dimension of the space of $L^2$-harmonic $k$-cycles on $\hat X$ (see \cite{l94}*{Lemma 2.7}).
Therefore the $k$-th rational homology growth can be identified with the von Neumann dimension of this space.
In particular, if $\pi_1(X)$ is residually finite, then $\hat X$ is the universal cover, and we have
\begin{corollary}
    For any finite complex $X$ with residually finite fundamental group we have
    \[
    \lb_*(X;\Q)= b_*^{(2)}(X)= \ub_*(X;\Q).
    \]
\end{corollary}

For other coefficients, we only have the inequality $\lb_*(X;\F)\leq\ub_*(X;\F)$.
It follows directly from the definitions that the interval $[\lb(X;\F),\ub(X;\F)]$ has the following interpretation: It is the smallest closed interval $[a,b]$ such that for every $\gd>0$ there is a finite cover $X'\to X$ such that for any further finite cover $X''\to X'$ the normalized Betti numbers $b(X'';\F)/\abs{X''\to X}$ lie in the interval $[a-\gd, b+\gd]$.

\subsection*{Connectedness}

If $X$ is disconnected, then its homology growth is the sum of the homology growth of its components, as the following lemma shows.
\begin{lemma}
    If $X=Y\amalg Z$, then
    \[
    \lb(X)=\lb(Y)+\lb(Z),
    \]
    \[
    \ub(X)=\ub(Y)+\ub(Z).
    \]
\end{lemma}
\begin{proof}
    Clearly, the normalized Betti numbers of a finite cover of $X$ is the sum of the normalized Betti numbers of its restrictions to $Y$ and $Z$.
    The issue is that in general finite covers $Y' \to Y$ and $Z' \to Z$ do not combine to a cover of $X$ unless they have same degree, as our definition requires the degree to be constant.
    However, we can equalize degrees by replacing $Y'$ with $\abs{Z' \to Z}$ disjoint copies of $Y'$ and replacing $Z'$ with $\abs{Y' \to Y}$ disjoint copies of $Z'$.
    This replacement does not change the normalized Betti numbers, and the Lemma follows.
\end{proof}

It is sometimes useful to keep in mind that we can compute homology growth of a connected finite complex either using all covers, or just the connected ones.
We record this observation here as a lemma.
\begin{lemma}\label{l:con}
    For a connected finite complex $X$, the upper and lower homology growth can be computed using connected covers.
\end{lemma}
\begin{proof}
    We give the proof for upper homology growth.
    Suppose $X$ is a connected complex, let $X'\to X$ be a finite cover and denote by $X_i'$ its components.
    Then the normalized Betti numbers of this cover are a convex combination
    \[
    \nb(X)(X')=\sum\frac{|X_i'\to X|}{|X'\to X|} \nb(X)(X_i')
    \]
    of the normalized Betti numbers of the components.
    Since the coefficients sum to one, we conclude that $\nb(X)(X')\leq \nb(X)(X'_i)$ for some $i$.
    So $\nb(X)^{\sup}$ can be computed over connected covers of $X$.
    Since $\nb(X)^{\sup}$ is decreasing, its $\inf$ can be computed over any cofinal subset.
    In particular, it can be computed over covers that have identical components.
    Clearly, the answer of this computation is the same as that for any of the components.
    Therefore $\ub(X)$ of a connected complex can be computed using only connected covers.
\end{proof}
\begin{remark}
    The regular covers of $X$ form a cofinal subset of $C_{X}$, so Lemma~\ref{l:posetineq}\eqref{i:cofinal} provides bounds for homology growth in terms of the limits over regular covers.
    We don't know whether regular covers give an exact computation.
\end{remark}

\subsection*{Finiteness}

Since connected covers correspond to subgroups of the fundamental group, we can relax finiteness assumptions on $X$.
Recall that a connected complex $X$ is of type $FP_{n}(\F )$ if the chain complex of the universal cover $C_{*}(\widetilde X; \F)$ is $\F\pi_{1}X$-chain homotopy equivalent to a complex $P_{*}$ of free $\F\pi_{1}X$-modules which have finite rank in degrees $\leq n$.
The following lemma shows that homology growth is well defined and finite in degrees $\leq n$ for such complexes.
\begin{lemma}\label{l:bound}
    Let $X$ be a connected complex so that the chain complex of the universal cover $C_{*}(\widetilde X; \F)$ is $\F\pi_{1}X$-chain homotopy equivalent to a complex $P_{*}$ of $\F\pi_{1}X$ modules, where $P_{k}$ is a free module of finite rank.
    Then the normalized Betti function $\nb_{k}(X; \F)$ is bounded:
    \[
    \nb_{k}(X; \F) \leq \rk_{\F\pi_{1}X} P_{k}.
    \]
\end{lemma}
\begin{proof}
    By the proof of the previous Lemma it is enough to check the inequality for connected covers.
    Denote $G:=\pi_{1}X$.
    Let $X' \to X$ be a finite connected cover, and let $G' < G$ be the corresponding subgroup.
    Then we have
    \[
    \nb_k(X;\F)(X') = \frac{ \dim_{\F} H_k(X; \F [G/G'] )}{[G:G']} = \frac{ \dim_{\F} H_k(P_{*}\otimes_{\F G} \F [G/G'] )}{[G:G']} \leq \rk_{\F G}(P_k).
    \]
\end{proof}

\subsection*{Multiplicativity}

If $X' \to X$ is a finite cover, then $C_{X'}$ is naturally identified with the tail $(C_{X})_{\geq X'}$ of $C_{X}$, and on this tail we have $\nb(X) \abs{X' \to X} = \nb(X')$.
Therefore, since by Lemma~\ref{l:posetineq}(4) the limits can be computed over tails, the homology growth is multiplicative in covers of $X$:
\begin{align*}
    \ub(X')&= \ub(X) \abs{ X' \to X},\\
    \lb(X')&= \lb(X) \abs{ X' \to X}.
\end{align*}

\subsection*{Homology growth as a fibering obstruction}

Multiplicativity implies a variant of L{\"u}ck's mapping torus theorem \cite{l94a} for homology growth.
\begin{theorem}[$\F $-homology mapping torus theorem for $\ub$]\label{mappingtorustheorem}
    Let $X$ be a complex of type $FP_n(\F )$, $f:X\to X$ a self-homotopy equivalence and $T_f$ its mapping torus.
    Then for $k\leq n$
    \[
    \ub_k(T_f;\F )=0.
    \]
\end{theorem}
\begin{proof}
    We can assume that $f$ is a cellular map.
    Let $K=\pi_1(X)$ and let $G=\pi_1(T_f)=K\rtimes\Z$.
    The assumption means that the chain complex $C:=C_*(\widetilde X;\F )$ is $\F K$-chain homotopy equivalent to a complex of free $\F K$-modules $P$ which have finite rank in degrees $\leq n$.
    Let $h:C\to P$ and $l:P\to C$ be the chain homotopy equivalence and its inverse.
    Consider the map $g=f^{m}$ for some positive integer $m$.
    The map $g$ induces a chain map which we will also call $g:C\to C$.

    Let $\hat g=h f l: P\to P$.
    The algebraic mapping telescope $\widetilde T_{\hat g}$ of $\hat g$ is $\F G$-chain homotopy equivalent to $C_*(\widetilde T_g;\F )$.
    Thus, by Lemma~\ref{l:bound}
    \[
    \nb_k(T_{g};\F) \leq \rk_{\F G} (\widetilde T_{\hat g})_{k} = \rk_{\F K}(P_k\oplus P_{k-1}).
    \]

    The mapping torus $T_{g}$ is homotopy equivalent to a degree $m$ cover $T'\to T_f$, so by multiplicativity
    \[
    m\ub_{k}(T_{f}; \F)= \ub_{k}(T_{g}; \F) \leq \rk_{\F K}(P_k\oplus P_{k-1}).
    \]
    Since $m$ can be picked arbitrarily large, we are done.
\end{proof}

\section{\texorpdfstring{$\beta^{\inf}$}{β\^inf} via skew fields}\label{s:skewfields}
In this section we will give a skew field description of $\beta^{\inf}$ for complexes with residually torsion-free nilpotent fundamental group.
The proof goes by first approximating the residually torsion-free nilpotent group by torsion-free nilpotent groups, and then approximating those by finite groups.
Since groups are central to this argument, we will use equivariant notation to highlight the role of the groups involved, rather than relegating it to a subscript in a coefficient module.

\subsection*{Skew field Betti numbers}

Let $G$ be a group, let $Y$ be a free cocompact $G$-CW complex, and suppose we have a homomorphism $\phi: \zz G \to D$ to a skew field.
The homomorphism makes $D$ into a $\zz G$-bimodule, so we can take equivariant homology of $Y$ with coefficients in $D$
\[
H_{*}^{G}(Y ; D)=H_{*}( D \otimes_{\zz G} C_*(Y) ),
\]
and define the \emph{equivariant Betti numbers with coefficients in $D$ of $Y$} by taking its dimension over $D$:
\[
b^G_{*}(Y; D)= \dim_{D} H_{*}^{G}( Y; D ).
\]
More explicitly,
\[
b^G_{i}(Y; D)= \abs{Y^{(i)}} - \big( \rk_{D} \phi(\d_{i}) + \rk_{D} \phi( \d_{i+1}) \big),
\]
where $\phi(\d_{i})$ denotes the image of the matrix of the differential in $D$ and $ \abs{Y^{(i)}} $ is the number of $G$-orbits of $i$-cells in $Y$.
\begin{remark}
    When $Y$ is the universal cover of a connected finite complex $Y/G$ with fundamental group $G$, then this definition coincides with the usual (unequivariant) homology of $Y/G$ with local coefficients in the $\zz G$-module $D$,
    \[
    H_*^G(Y;D)=H_*(Y/G;D).
    \]
    On the level of skew field Betti numbers, $b_*^G(Y;D)=b_*(Y/G;D)$.
\end{remark}
\subsection*{Local homomorphisms}

A nontrivial homomorphism between skew fields is necessarily injective, however there is a more general notion of morphisms between skew fields.
It leads to an inequality between Betti numbers.

A \emph{local homomorphism} (or \emph{subhomomorphism}) between two skew fields $D$ and $D'$ is a homomorphism from a subring $L$ of $D$ to $D'$, $f: L \to D'$ whose kernel is precisely the set of non-units of $L$.
It follows that $L$ is a local ring, $J:=\ker f$ is its unique maximal ideal, and $J\backslash L$ is a sub-skew field of $D'$.
If $M$ is a finitely generated $L$-module, then by Nakayama's lemma, cf. \cite{fd93}*{Corollary 2.13}, any lift of a basis of $J\backslash L \otimes_{L} M= JM\backslash M $ to $M$ is a generating set for $M$, therefore $\dim_{D'} D' \otimes_{L} M \geq \dim_{D} D \otimes_{L} M$.
In terms of ranks we have an opposite inequality, for any $L$-matrix $A$, $\rk_{D'} f(A) \leq \rk_{D} A$.

Thus we have the following: suppose $\zz G \to D'$ extends to a local homomorphism from $D$ to $D'$, then
\[
b^G_*(Y;D) \leq b^G_*(Y;D').
\]

If $D$ is a sub-skew field of $D'$, then we have obvious local homomorphisms from $D$ to $D'$ and vice versa extending the map $\zz G \to D$, which gives us
\begin{Change of coefficients}
    If $D$ is contained in another skew field $D'$, then
    \[
    b^{G}_{*}(Y;D)=b^{G}_{*}(Y;D'),
    \]
    where the latter is computed using the composition $\zz G \to D \into D'$.
    (This is also a consequence of $C_*(Y)\otimes_{\F G}D'=(C_*(Y)\otimes_{\F G}D)\otimes_DD'$.)
\end{Change of coefficients}

In particular, if $\phi:\zz G \to D$, we can always replace $D$ with the skew field generated by the $\phi(\zz G)$, i.e. the \emph{division closure} of $\phi(\zz G)$, without changing Betti numbers.

\subsection*{Epic $\F G$ fields}

There are two classical constructions (for certain amenable groups, and for bi-orderable groups) of $\zz G$ skew fields.
Both constructions depend on the choice of a base field $\F$ (we will be mostly concerned with $\F=\Q$ or $\F=\Fp$) and produce canonical epic embeddings $\F G \into D_{\F G}$.
(A homomorphism of $\F G$ into a skew field $D$ is \emph{epic} if the image of $\F G$ generates $D$.) Moreover, both constructions behave nicely with respect to subgroups.
If $H<G$ then the division closure of $\F H$ in $D_{\F G}$ coincides with $D_{\F H}$.
If $H$ happens to be finite index in $G$, then in both cases $D_{\F G} \cong \oplus_{[G:H]} D_{\F H}$, which implies a multiplicativity formula for the corresponding Betti numbers.
We shall use the same notation for both constructions, and let the context distinguish them.
This does not lead to confusion, as the constructions agree when both are defined.

Of course the existence of such an embedding requires $\F G$ to have no zero divisors.
Conjecturally, $ \F G$ has no zero divisors for any torsion-free group.
It is known for many classes of groups, in particular for left-orderable groups and for torsion-free elementary amenable groups \cite{klm88}*{Theorem 1.4}.

We now discuss both constructions.

\subsection*{Amenable groups}

Suppose $R$ is a ring without zero divisors, and $S$ is a multiplicatively closed subset of nonzero elements.
The pair $(R,S)$ satisfies the (right) Ore condition if for each $r \in R$ and $s \in S$ there are $r' \in R$ and $s' \in S$ with
\[
rs' = sr'.
\]
If the pair $(R,S)$ satisfies the Ore condition, then one can form a ring called the Ore localization $RS^{-1}$.
The elements of $RS^{-1}$ are equivalence classes of fractions $r/s$, $r \in R$, $s \in S$; the Ore condition allows one to add and multiply these expressions.
There is a natural injection $R \to RS^{-1}$, given by $r \mapsto r/1$.
If $S$ is the set of all nonzero elements of $R$, then $RS^{-1}$ is a skew field, and we get an epic embedding $R \into RS^{-1}$.
This embedding is a unique epic embedding, since any other embedding factors through it.

For a group ring, $R = \F G$ without zero divisors, the pair $(\F G ; \F G- \{0\} )$ satisfies the Ore condition if and only if $G$ is amenable \cite{b19}*{Theorem A.1}.
In this case we shall denote the localization by $D_{\F G}$.
So, to summarize, for amenable $G$ with $ \F G$ having no zero divisors we have a unique epic embedding $\F G \into D_{\F G}$.
If $G$ is amenable and $\F G$ has no zero divisors, then this also holds for all subgroups of $G$ and their group rings.
Furthermore, if $H < G$ then we can identify $D_{\F H}$ with the division closure of $\F H$ inside of $D_{\F G}$.
If $G$ is in addition residually finite, then we also have a version of L\"uck's Approximation theorem, which is the main result of Linnell, L\"uck, and Sauer \cite{lls11}*{Theorem 0.2}.
\begin{theorem}[\cite{lls11}*{Theorem 0.2}]\label{t:approx}
    Suppose $G$ is amenable and $\F G$ has no zero divisors.
    Let $G_{i} \lhd G$ be a residual sequence of finite index normal subgroups of $G$.
    Then for any cocompact free $G$-CW complex $Y$
    \[
    b^G_{*}(Y; D_{\F G}) = \lim_{i \to \infty} \frac{b_{*}(Y/G_{i}; \F)}{[G:G_{i}] }.
    \]
\end{theorem}

\subsection*{Bi-orderable groups}

As another example, suppose the group $G$ is bi--orderable, and fix a particular bi-invariant total order on $G$.
The Malcev--Neumann series are infinite linear combinations of elements in $G$ with $\F$ coefficients, whose support is well ordered in the induced order.
They form a skew field, into which the group ring $\F G$ naturally embeds.
Let $D_{\F G}$ denote the division closure of $\F G$ in this skew field.

One can easily see from the construction of inverses that if $a$ is a non-zero Malcev--Neumann series then the support of $a^{-1}$ is contained in the subgroup generated by the support of $a$.

Therefore, for a subgroup $H$ the set of elements of $D_{\F G}$ supported on $H$ is a sub-skew field, and it follows that the division closure of $\F H$ in $D_{\F G}$ is naturally identified with $D_{\F H}$, coming from the induced order on $H$.

Since the action of $G$ on cosets preserves the induced order, a similar picture holds for the set of elements of $D_{\F G}$ supported on a single coset, it has a natural structure of a vector space over $D_{\F H}$.
This gives an injective homomorphism of $D_{\F H}$-vector spaces: $ \bigoplus_{G/H} D_{\F H} \to D_{\F G} $.

The injectivity of this homomorphism is a (very strong) form of the so called Hughes-free condition, so the embedding $\F G \into D_{\F G}$ is Hughes-free.
Hughes \cite{h70} proved that for a given $\F$, Hughes-free epic embeddings are unique up to an isomorphism over $\F G$, thus $D_{\F G}$ does not depend on the choice of the order, and we obtain well-defined Betti numbers $b^G_*(Y;D_{\F G})$.

We also have an approximation theorem of a different flavor.
\begin{theorem}\label{t:ordapprox}

    Let $K_i$ be a nested residual sequence of normal subgroups in a group $G$.
    Suppose we have bi-invariant orderings on $G$ and on the quotients $G_{i}=G/K_{i}$ such that each quotient map $p_{i}:G \to G_{i}$ is order-preserving.
    Then for any cocompact free $G$-complex $Y$ there exist $i_{0}$ such that for any $i \geq i_{0}$,
    \[
    b^{G}_{*}(Y; D_{\F G})= b^{G_{i}}_{*}(Y/K_{i}; D_{\F G_{i}}).
    \]
\end{theorem}
\begin{proof}
    Following \cites{el87, l00} define a sequence of subrings $\{S_i\}$ of $D_{\F G}$ with $S_i \subset S_{i+1}$ consisting of elements whose support has finite intersection with $K_{i}$-cosets.
    A slight generalization of the Malcev--Neumann argument, cf. \cite{l00}*{Proposition 7.1}, shows that $\bigcup S_{i}$ is a skew field.
    Since each $S_{i}$ contains $\F G$, we have $\bigcup S_{i}=D_{\F G}$.
    Note that the maps $p_{i}$ obviously extend to maps $p_{i}: S_{i} \to D_{\F G_{i}}$.

    Given a matrix $A$ over $\F G$, we can diagonalize $A$ over $D_{\F G}$ by performing column and row operations: there exist $D_{\F G}$-matrices $C$ and $C'$, such that $CAC'=I_{\rk_{D_{\F G}} A}$.
    The entries of $C$ and $C'$ are a finite collection of elements of $D_{\F G}$, so they are all contained in $S_{i_{0}}$ for some $i_0$.
    Then for $i \geq i_{0}$, applying $p_{i}$ to the above diagonalization shows that $\rk_{D_{\F G}} A = \rk_{D_{\F G_{i}}} p_{i}(A)$.
    Choosing $i_{0}$ so that the above works for all differentials in $Y$ finishes the proof.
\end{proof}
\begin{remark}
    The skew field $\bigcup S_{i}$ has been recently used by Sikorav in \cite{s23} to give a new fibering criterion for closed aspherical $3$-manifolds.
\end{remark}

\subsection*{Finite index subgroups}

\label{s:mult}
It turns out that in both cases the equivariant Betti numbers have an additional nice property satisfied by the usual $L^2$-Betti numbers, namely multiplicativity for finite index subgroups.
\begin{lemma}\label{l:mult}
    Suppose $G$ is bi-orderable or amenable with $\F G$ having no zero-divisors, and $H < G$ is a finite index subgroup.
    Then for any $G$-complex $Y$
    \[
    b^{H}_{*}(Y; D_{\F H})=[G:H] b^{G}_{*}(Y;D_{\F G}).
    \]
\end{lemma}
\begin{proof}
    For bi-orderable groups, if $H<G$ has finite index, then the full Malcev--Neumann skew field of $G$ is a vector space of dimension $[G: H]$ over the full Malcev--Neumann skew field of $H$, and these skew fields give the same Betti numbers as $D_{\F H}$ and $D_{\F G}$.
    In particular, it follows that $D_{\F G} \cong D_{\F H}^{[G:H]}$ as $D_{\F H}$-vector spaces.

    This also holds in the amenable case, see e.g.
    \cite{lls11}*{Equation 5.2} for a more general statement.
    For convenience, we give the argument here assuming that there are no zero-divisors.
    It is enough to consider normal subgroups.
    The main point is that the pair $(\F G, S = \F H - \{0\})$ satisfies the Ore condition \cite{p77}*{ Lemma 13.3.5 (ii)}, so we can form the localization $(\F G) S^{-1}$.
    This is a $D_{\F H}$-vector space of dimension $[G: H]$ which naturally injects into $D_{\F G}$.
    We claim that this is onto; it suffices to show that each nonzero $t$ in $\F G$ is invertible in $(\F G) S^{-1}$.
    Since $t$ is not a zero divisor in $D_{\F G}$, it is not a zero divisor in $(\F G) S^{-1}$, hence the left multiplication by $t$ induces an injective linear self-map of $(\F G) S^{-1}$.
    Therefore, this multiplication is an isomorphism, and the preimage of 1 is the inverse of $t$.
\end{proof}

\subsection*{Finite generation} Note that we did not assume that the group $G$ is finitely generated.
We now show that in both constructions we can always reduce the computation of $b_k^G(Y;D_{\F G})$ to the case of finitely generated $G$.

First, we need the following observation.
\begin{Induced representation}
    If $H < G$ is a subgroup and $Y_{0}$ is an $H$-complex, then for $Y=G \times_{H} Y_{0}$ we have
    \[
    b^{G}_{*}(Y;D) = b^{H}_*(Y_{0};D),
    \]
    where the latter is computed using the composition $\zz H \into \zz G \to D$, since the chain complexes used to compute the two homologies are identical:
    \[
    C_*(Y_{0})\otimes_{\F H}D=C_*(Y_{0})\otimes_{\F H}\F G\otimes_{\F G}D=C_*(Y)\otimes_{\F G}D.
    \]
\end{Induced representation}
\begin{lemma}\label{l:connected}
    Suppose $G$ is a bi-orderable or amenable group with $\F G$ having no zero-divisors, and $Y$ is a cocompact, free $G$-complex.
    For each connected component of $Y/G$ choose its lift $Y_{i}$ to $Y$, and let $G_{i}$ denote the stabilizer of $Y_{i}$ in $G$.
    Then each $G_{i}$ is finitely generated and
    \[
    b_k^G(Y;D_{\F G})=\sum_{i=1}^n b_k^{G_i}(Y_i;D_{\F G_i}).
    \]
\end{lemma}
\begin{proof}
    Since each $Y_i$ is a connected cocompact free $G_{i}$-complex, each $G_{i}$ is finitely generated.
    $Y$ is a disjoint union of $G$-orbits of $Y_{i}$, i.e. $Y=\coprod_{i=1}^n G\times_{G_i} Y_i$.
    In both constructions the division closure of each $\F G_i$ in $D_{\F G}$ is $D_{\F G_i}$.
    Therefore
    \[
    b_k^G(Y;D_{\F G})=\sum_{i=1}^nb_k^{G_i}(Y_i;D_{\F G})=\sum_{i=1}^n b_k^{G_i}(Y_i;D_{\F G_i}).
    \]
\end{proof}

\subsection*{Torsion-free nilpotent groups}

Let $N$ be a torsion-free nilpotent group.
Then $N$ is both bi-orderable and amenable, and both constructions produce the same skew field $D_{\F N}$.
It will be most useful here to think of $D_{\F N}$ as an Ore localization.
\begin{lemma}\label{l:univ}
    Let $N$ be a torsion-free nilpotent group.
    Then for any free, cocompact $N$-CW complex $Y$
    \[
    b^{N}_{*}(Y; D_{\F N}) \leq b_{*}(Y/N; \F).
    \]
\end{lemma}
\begin{proof}
    By Lemma~\ref{l:connected} we can assume that $N$ is finitely generated.
    Let $C = \langle t \rangle$ be a normal, infinite cyclic subgroup of $N$ with $H=N/C$ torsion-free nilpotent.
    The existence of such a subgroup follows from the fact that the center $Z(N)$ of $N$ is infinite and $N/Z(N)$ is torsion-free nilpotent \cite{p77}*{Lemma 11.1.3, p.~470}.

    The quotient map $N \to H$ induces a map $p: \F N \to \F H$ and $\ker p$ is the two-sided principal ideal generated by $(1-t)$.
    We claim that $p$ extends to a local homomorphism $D_{\F N} \to D_{\F H}$ where the domain consists of elements which have a representation with denominator not in $\ker p$.
    The only nontrivial part of this claim is that the domain is a subring of $D_{\F N}$, or equivalently, that $(\F N,S:=\F N-\ker p)$ satisfies the Ore condition.

    To see this, take $r \in \F N$ and $s \in S$.
    Since $(\F N; \F N - \{0\} )$ satisfies the Ore condition, there are $r', s' \in \F N$ with $rs' = sr'$, and we need to show that $s'$ can be chosen in $S$.

    The key point is that there is a bound on the powers of $(1-t)$ that divide $s'$.
    Indeed, take a coset of $C$ which intersects the support of $s'$ nontrivially; for a suitable choice of $g \in N$ the restriction of $s'$ to this coset has the form $gP(t)$ where $P(t)$ is a polynomial in $t$.
    The right multiplication by $(1-t)$ preserves the coset decomposition, hence the power of $(1-t)$ dividing $s'$ on the right is bounded above by the degree of $P(t)$.

    Now, if $s' \in \ker p$, then $s'r=sr' \in \ker p$, and hence $r' \in \ker p$ as $\F H$ has no zero divisors.
    Therefore, both $s'$ and $r'$ are divisible on the right by $(1-t)$, and we can cancel to get a new $s'$ and $r'$.
    So, we can keep cancelling powers of $(1-t)$ until $s' \in S$.

    Applying the same procedure to the quotient group $H$ in place of $N$ and composing local homomorphisms eventually produces a trivial quotient and therefore a local homomorphism $D_{\F N} \to \F$ extending the augmentation map.
\end{proof}

\subsection*{Residually torsion-free nilpotent groups}

From now on, we suppose that the group $G$ is residually torsion-free nilpotent.
Then $G$ is bi-orderable, and in fact as explained in \cite{el87}*{Corollary to Lemma 4.1}, any torsion-free nilpotent approximation can be made into an bi-orderable one by a suitable choice of orders.
Then the equivariant Betti numbers of a $G$-complex $Y$ are approximated by the equivariant Betti numbers of the torsion-free nilpotent quotients, which in turn are approximated by normalized usual Betti numbers of finite quotients.
This leads to equivalence between the skew field definition and the infimum definition for simply connected $Y$.
\begin{theorem}\label{t:equiv}
    Let $G$ be a residually torsion-free nilpotent group and let $Y$ be a free cocompact $G$-complex.
    Then
    \[
    b^{G}_{*}(Y; D_{\F G}) = \inf_{\substack{H <G \\
    [G:H] < \infty}} \frac{b_{*}(Y/H; \F)}{[G:H]}.
    \]
\end{theorem}
\begin{proof}
    Let $H< G$ be a finite index subgroup of $G$.
    Choose a torsion free nilpotent approximation $K_{i} \lhd H \onto N_{i}$.
    Then, using Theorem~\ref{t:ordapprox} and Lemma~\ref{l:univ}, for sufficiently large $i$ we have,
    \[
    b^{H}_{*}(Y; D_{\F H}) = b^{N_{i}}_{*}(Y/K_{i}; D_{\F N_{i}}) \leq b_{*}(Y/H ; \F).
    \]
    The multiplicativity now implies $\leq$ inequality in the desired formula.

    For the opposite inequality we first find, similar to the above, a torsion free nilpotent quotient $K \lhd G \onto N$ with $b^{G}_{*}(Y; D_{\F G}) = b^{N}_{*}(Y/K; D_{\F N})$ and then apply Theorem~\ref{t:approx} to the $N$-complex $Y/K$ to find a further finite quotient so that $b^{N}_{*}(Y/K; D_{\F N}) $ is approximated within any given $\ge$ by the normalized usual Betti numbers.
\end{proof}

\subsection*{Simply connected components}

When the components of the $G$-complex $Y$ are simply connected, then we can use Theorem~\ref{t:equiv} to relate skew field Betti numbers to homology growth of the quotient $Y/G$ (note that if $Y$ is not simply connected then the right-hand term in Theorem~\ref{t:equiv} is generally not equal to $\beta^{\inf}_{\ast}(Y/G)$).
\begin{corollary}\label{c:sccomp}

    Suppose $G$ is a residually torsion-free nilpotent group and $Y$ is a cocompact, free $G$-complex with simply connected components.
    Then
    \[
    b_k^G(Y;D_{\F G})=\beta^{\inf}_k(Y/G;\F).
    \]
\end{corollary}
\begin{proof}
    By Lemma~\ref{l:connected}
    \[
    b_k^G(Y;D_{\F G})=\sum_{i=1}^n b_k^{G_i}(Y_i;D_{\F G_i}).
    \]
    where each $Y_i$ is a connected component of $Y$ and $G_{i}$ is its stabilizer.
    By assumption each $Y_i$ is simply connected, so finite connected covers of $Y_{i}/G_{i}$ correspond to finite index subgroups of $G_{i}$.

    Therefore, Theorem \ref{t:equiv} and Lemma~\ref{l:con} imply that right hand side equals $\sum_i\beta^{\inf}_k(Y_i/G_i;\F)$.
    Finally, additivity of $\beta^{\inf}$ in disjoint unions identifies this with $\beta^{\inf}_k(Y/G;\F)$.
\end{proof}
\begin{remark}
    Droms \cite{d83a} and Duchamp--Krob \cite{dk92} independently showed that RAAG's are residually torsion-free nilpotent.
    Since this property passes to subgroups, the fundamental group of any compact special cube complex in Haglund and Wise's sense is residually torsion-free nilpotent.
\end{remark}
\begin{remarks}

    Many of the results in this section (and stronger versions) were previously known.
    The fact that for a torsion-free nilpotent group $N$ there is a local homomorphism $D_{\F N} \to \F$ extending the augmentation map follows from \cite{s71}*{Theorem 2.2}, which implies that the complement of the augmentation ideal in $\F N$ satisfies the Ore condition.
    The local ring in this case then consists of fractions with representatives $f/g$ where $g$ is not in the augmentation ideal.

    More generally, a skew field $D$ containing and generated by $\F G$ is called \emph{universal} if any homomorphism $\alpha: \F G \to D'$ can be extended to a local homomorphism $\alpha: D \to D'$.
    Of course, if $D$ is universal, there is a local homomorphism $D \to \F$ extending the augmentation map, so we obtain the same statement as in Lemma \ref{l:univ}.
    For instance, Passman \cite{p82} showed that if $G$ is poly-$\zz$, then $D_{\F G}$ is a universal division ring for $\F G$.

    Jaikin-Zapirain \cite{j21}*{Corollary 1.3} proved that a residually (amenable locally indicable) group has a Hughes-free epic embedding, and this embedding is universal.
    The core of his construction is a very general form of the approximation theorem by locally indicable quotients.
    He also observed \cite{j21}*{Proposition 2.2} that a recent result of Gr\"ater \cite{g20}*{Corollary 8.3} that Hughes-free embeddings are strongly Hughes-free implies multiplicativity of the equivariant Betti numbers, see also \cite{f21}*{Lemma 6.3} for more details.
    Since the Linnell--L\"uck--Sauer theorem applies to amenable locally indicable groups, it follows that Theorem~\ref{t:equiv} holds for residually (amenable locally indicable) groups.
    For these groups, Fisher, Hughes and Leary in a recent paper independently proved one inequality in Theorem~\ref{t:equiv} in \cite{fhl23}*{Theorem D}, and applied this to show non-vanishing homology growth of non-virtually fibered groups in \cite{fhl23}*{Theorem 5.1}.

    For $\F = \Q$, there is another canonical construction due to Linnell of $D_{\Q G}$ that (conjecturally) works for all torsion-free groups $G$.
    The von Neumann algebra $\cn(G)$ is known to satisfy the Ore condition with respect to the set of non-zero divisors, and $D_{\Q G}$ is the division closure of $\zz G$ inside of $\Ore(\cn(G))$.
    The ring $\Ore(\cn(G))$ can be identified with the ring $\cu(G)$ of affiliated operators on $\ell^2(G)$.
    Since $\cn(G)$ has zero-divisors, $\Ore(\cn(G))$ is not a skew field, so it is not obvious that $D_{\Q G}$ is one.
    On the other hand, Linnell showed that $D_{\Q G}$ being a skew field is equivalent to Atiyah's conjecture on integrality of $L^2$-Betti numbers for torsion-free groups, and this is known for many classes of groups.
    If $H < G$ and the Atiyah conjecture holds for $G$, then it also holds for $H$, $D_{\Q H}$ naturally embeds as a sub-skew field of $D_{\Q G}$, and $D_{\Q G}$ is strongly Hughes-free.
    For general fields $\F$, Jaikin-Zapirain and Linton have conjectured that for any torsion-free group $G$ there is an epic, strongly Hughes-free, skew field $D_{\F G}$ containing $\F G$ which is unique up to $\F G$-isomorphism \cite{jl23}*{Conjecture 1, p.7}.
\end{remarks}

\section{Applications of the skew field theory}\label{s:applications}
Let us collect some consequences of the skew field theory from the last section.

\subsection*{Lower homology growth as skew field Betti number}
\begin{corollary}\label{nosup}
    If $X$ is a finite complex with residually torsion-free nilpotent fundamental group $G$, then
    \[
    \lb_*(X;\F )=\beta_*^{\inf}(X;\F )=b^G_*(\widetilde X;D_{\F G})\in\Z.
    \]
\end{corollary}
\begin{proof}
    First, note that $b^G_*(\widetilde X;D_{\F G})$ is a dimension of a vector space over a skew field, so it is an integer.
    Second, this dimension is equal to $\beta_*^{\inf}(X;\F )$ by Theorem \ref{t:equiv}.
    Third, since $b^G_*(\widetilde X;D_{\F G})$ is multiplicative by Lemma~\ref{l:mult}, the normalized $\beta_*^{\inf}$ of finite covers of $X$ are all equal to each other.
    Therefore $\lb_*=\beta_*^{\inf}$.
    This finishes the proof.
\end{proof}
So, since the lower homology growth $\lb$ is multiplicative, it can be thought of as a multiplicative extension of the skew field Betti number from residually torsion-free nilpotent fundamental groups to more general settings where there is no nice skew field around.
\subsection*{Relation between $\Q$ and $\Fp$}

For a finite complex $X$ let
\[
\gt_{k}(X)(X'):=\frac{\logtor H_k(X')}{\abs{X' \to X}}
\]
denote the normalized log of the cardinality of the torsion of the integral homology as a function on the poset of covers.
By a lemma of Gabber (\cite{abfg21}*{Proposition 9.1}) $\gt_{k}(X)$ is a bounded function.

The universal coefficient theorem implies that
\begin{equation}\label{e:uct}
    0 \leq \left( \nb_{k}(X ;\Fp)-\nb_{k}(X ; \Q) \right) \log p \leq \gt_{k}(X) + \gt_{k-1}(X).
\end{equation}

The integrality of the skew field Betti number leads to the following corollary.
\begin{corollary}\label{specialp}
    Suppose $X$ is a finite complex with virtually residually torsion-free nilpotent fundamental group $G$.
    Then for sufficiently large primes $p$ we have
    \[
    \lb_k(X;\Fp)=\lb_k(X;\Q)=b_k^{(2)}(X).
    \]
\end{corollary}
\begin{proof}
    This follows immediately from inequalities \eqref{e:uct} and the fact that $\lb_k \in \frac{1}{[G:G']}\zz$ for some finite index residually torsion-free nilpotent subgroup $G'$.
\end{proof}

\subsection*{Relation between $\lb$ and $\ub$}

We now explain how to use a result of Fisher \cite{f21} to reconcile vanishing of upper and lower $\F $-homology growth for finite aspherical complexes whose fundamental groups embed in right-angled Artin groups.
More precisely, Fisher needs the groups to be residually finite rationally solvable (RFRS), a condition used in Agol's fibering criterion for $3$-manifolds, see \cite{a08}*{Definition 2.1}.

The main result in Fisher's paper \cite{f21} is:
\begin{theorem}[\cite{f21}*{Theorem 6.6}] Suppose $X$ is a finite aspherical complex whose fundamental group $G$ is RFRS. Then there is a finite cover $X'\to X$ and a map $X'\to S^1$ with homotopy fibre\footnote{In this case, this is just the infinite cyclic cover of $X'$ induced by the map to $S^1$.}
    of type $FP_n(\F )$ if and only if $b^G_{k}(\widetilde X; D_{\F G})=0$ for all $k\leq n$.
\end{theorem}
In other words, the numbers $b^G_k(\widetilde X;D_{\F G})$ are the \emph{only} $\F $-homological virtual fibering obstructions for a finite aspherical complex $X$ in this setting.
By Corollary 2.3 of \cite{a08}, RAAGs are RFRS and it is not hard to see that the RFRS property passes to subgroups.
This along with Theorem \ref{mappingtorustheorem} implies:
\begin{prevthm}
    {\ref{only if}} Suppose $X$ is a finite aspherical complex whose fundamental group $G$ embeds in a right-angled Artin group.
    Then $\lb_k(X;\F )=0$ for all $k\leq n$ if and only if $\ub_k(X;\F )=0$ for all $k\leq n$.
\end{prevthm}
\begin{proof}
    By Corollary~\ref{nosup} we have $b^G_k(\widetilde X;D_{\F G})=\lb _k(X;\F )\leq\ub_k(X;\F )$.
    Suppose that $b^G_k(\widetilde X;D_{\F G})=0$ for all $k\leq n$.
    Then, by Fisher's theorem there is a finite cover $X'\to X$ and a further regular infinite cyclic cover $\hat X'\to X'$ such that $\hat X'$ is of type $FP_n(\F )$.
    In other words, $X'$ is homotopy equivalent to the mapping torus $T_g$ of the covering translation $g:\hat X' \to \hat X'$.
    Therefore for $k\leq n$ we have, by Theorem \ref{mappingtorustheorem}, $0=\ub_k(T_g;\F )=\ub_k(X';\F )$, and by multiplicativity of $\ub$ we conclude that $\ub_k(X;\F )=0$.
\end{proof}
Poincar\'e duality implies:
\begin{corollary}\label{c:Babove}
    If $M$ is a closed aspherical manifold whose fundamental group embeds in a right-angled Artin group, then $\ub_{>k}(M;\F )=0$ if and only if $\lb_{>k}(M;\F )=0$.
\end{corollary}
\begin{remark}
    It follows from Jaikin-Zapirain's work in \cite{j21} that Theorem \ref{p:infsup} holds more generally for finite aspherical complexes $X$ with RFRS fundamental group, as RFRS groups are residually (amenable locally indicable).
\end{remark}

\section{Graph products}\label{s:ht}
\label{s:graphproducts}

In this section, we estimate the homology growth of graph products of finite groups.
In low dimensions, this will give us many examples of hyperbolic groups where we have good control over the homology growth; in particular we can construct hyperbolic groups where the homological growth depends on the coefficient field.

\subsection*{Graph products of finite groups}

Let $L$ be a flag complex with vertex set $S$, and suppose $\{G_s\}_{s \in S}$ is a collection of nontrivial groups indexed by $S$.
Let $G_L$ be the corresponding graph product.
Given a simplex $\gs$ in $L$, we let $G_\gs = \prod_{s \in \gs} G_s$ denote the corresponding special subgroup of $G_L$.
Any graph product $G_L$ acts naturally on a right-angled building of type $(W_L,S)$, which we now describe.

Let $\Cone{L}$ be the geometric realization of the poset of simplices of $L$.
Then $\Cone{L}$ is isomorphic to the cone on the barycentric subdivision of $L$, with the empty simplex corresponding to the cone point.
Let $\d \Cone{L}$ be the geometric realization of the poset of nonempty simplices of $L$, which corresponds to simplices in $\Cone{L}$ not containing the cone point.
Recall that a \emph{mirrored complex} consists merely of a complex $X$ and a collection of subcomplexes $\{X_s\}_{s \in S}$ for some index set $S$.
There is a canonical mirror structure on $\Cone{L}$ with mirrors $\{K_s\}_{s \in S}$; the $s$-mirror $K_s$ is the geometric realization of the subposet of simplices containing the vertex $s$.
This is isomorphic to the star of $s$ in the barycentric subdivision of $L$.
\begin{figure}
    \begin{tikzpicture}[>=Stealth]
        \node[name=s, shape=regular polygon, minimum size=2cm, draw ] {}; \foreach \anchor/\placement/\a in {corner 1/above/b, corner 2/above/a, corner 3/left/e\vphantom{d}, corner 4/right/d, corner 5/above/c }
        \draw[shift=(s.\anchor),fill] circle(1pt) node[\placement] {${}_{\a}$};
        \node[below=1.5cm] at (s) {$L$};
        \begin{scope}[xshift=4cm]
            \node[name=s, shape=regular polygon, minimum size=2cm, draw, fill=gray!30, very thick, outer sep=0 ] {}; \foreach \anchor/\placement/\a in {corner 1/above/b, corner 2/above/a, corner 3/left/e\vphantom{d}, corner 4/right/d, corner 5/above/c, side 1/above/ab, side 2/left/ae, side 3/below/de, side 4/right/cd, side 5/above/bc, center/above/} {
            \draw[shift=(s.\anchor),fill] circle(1pt) node[\placement] {${}_{\a}$};
            \draw (s.center)--(s.\anchor); }
            \node[outer sep=0pt, inner sep=1pt, minimum size=0, name =d] at (s.144){};
            \node[above left=.5cm] at (d) {$\d K_{L}$} edge[ ->, bend left, >=Stealth] (d);

            \node[below=1.5cm] at (s) {$K_{L}$};
        \end{scope}
        \begin{scope}[xshift=8cm]
            \node[name=s, shape=regular polygon, minimum size=2cm, draw, rotate=36 ] {}; \foreach \anchor/\placement/\a in {corner 1/above/ab, corner 2/left/ae, corner 3/below/de, corner 4/right/cd, corner 5/above/bc, side 1/left/a, side 2/left/e, side 3/right/d, side 4/right/c, side 5/above/b, center/above/} {
            \draw[shift=(s.\anchor),fill] circle(1pt) node[\placement] {${}_{\a}$}; } \foreach \anchor/\placement/\a in { side 1/left/a, side 2/left/e, side 3/right/d, side 4/right/c, side 5/above/b, center/above/}
            \draw (s.center)--(s.\anchor);
            \draw [very thick] (s.corner 1) -- (s.corner 2) ;
            \node[inner sep=1pt, minimum size=0, name =d] at (s.105){};
            \node[left=.5cm] at (d) {$K_{a}$} edge[ ->, bend left] (d);

            \draw [very thick] (s.side 2) -- (s.center) -- (s.side 5) ;
            \node[above=.1cm, inner sep=1pt, minimum size=0, name =v] at (s.center){};
            \node[right=1.3cm] at (v) {$K_{\Lk(a)}$} edge[ ->, bend right] (v);

            \node[below=1.5cm] at (s) {Cubical structure on $K_{L}$};
        \end{scope}
    \end{tikzpicture}
    \caption{The Davis chamber $K_L$ and its cubical structure.
    The cone vertex in $K_L$ corresponds to the empty simplex in $L$, and $\partial K_L$ consists of simplices/cubes not containing the cone vertex.}
\end{figure}
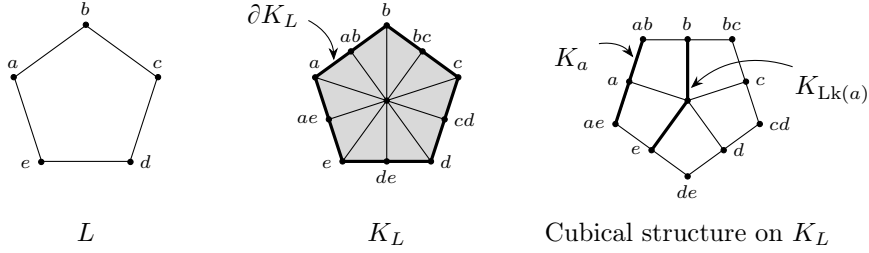

Let $x$ be a point in $\Cone{L}$, and let $\tau$ be a simplex containing $x$.
Then $\tau$ corresponds to a chain of simplices in $L$; let $\gs(x)$ be the smallest element in this chain.
We define
\[
\cu(G_L, \Cone{L}) = G_L \times \Cone{L}/ \sim
\]
where $(g,x) \sim (g', x')$ if and only if $x = x'$ and $gG_{\gs(x)} = g'G_{\gs(x)}$.

If $W \cong G_L$ is a right-angled Coxeter group, then $\cu(G_L, \Cone{L})$ is the Davis complex \cite{d08}*{Chapter 7}, which we denote by $\gS_L$.
If $G_L$ is any graph product, then $\cu(G_L, \Cone{L})$ is a right-angled building with apartments isomorphic to $\gS_L$.
In general, $G_L$ acts on $\cu(G_L, \Cone{L})$ with strict fundamental domain $\Cone{L}$.
The stabilizers of simplices are conjugates of $G_\gs$ for $\gs \subset L$.

From now on, we assume that the groups $G_s$ are all finite.
Then the right-angled building admits the structure of a locally finite CAT(0) cube complex, which we now describe.
Firstly, $\Cone{L}$ can be naturally identified with a CAT(0) cubical subcomplex of $[0,1]^{S}$; $\Cone{L}$ is precisely the union of subcubes of $[0,1]^{S}$ containing $(0,0, \dots, 0)$ corresponding to the collections of vertices which span simplices in $L$.
Note that the link of the vertex $(0,0,\dots,0)$ is isomorphic to $L$, and the vertices of each cube can be identified with vertices of the barycentric subdivision of $L$.

For a vertex $s$ the mirror $K_{s}$ is the intersection of $\Cone{L}$ with the hyperplane $x_{s}=1$.
It is naturally isomorphic to the cubical complex $\Cone{\Lk(s)}$.
There is another parallel embedding of $\Cone{\Lk(s)}$ into $\Cone{L}$ as a subchamber, given by the intersection with the coordinate hyperplane $x_{s}=0$.
We will use both embeddings in the paper.

This cubical structure extends to $\cu(G_L, \Cone{L})$.
The link of any vertex in $\cu(G_L, \Cone{L})$ is isomorphic to the multiple join $\Lk(\gs) * F_1 * \dots * F_{\dim \gs+1}$ of a link of some simplex $\gs$ in $L$ with finite discrete sets $F_i$ (it suffices to consider vertices in $K_L$, in which case the vertex is the barycentre of $\gs$ and the sets $F_i$ come from translates of $K_L$ by the local group $G_\gs = \prod_{s \in \gs} G_s$).
These are all flag complexes since $L$ is flag.
By \cite{d08}*{Theorem 18.3.1}, $\cu(G_L,\Cone{L})$ is simply connected, hence $\cu(G_L, \Cone{L})$ is CAT(0) by Gromov's criterion.

\subsection*{Estimating homology growth of graph products}
\begin{Theorem}\label{t:modpl2}
    Let $G_{L}$ be a graph product of finite groups.
    Suppose $\gG < G_{L}$ is a torsion free finite index subgroup and let $n$ denote its index.
    Then
    \[
    \abs{b_{i}(\gG; \F)/n -\tilde b_{i-1}(L; \F)} \leq 2\abs{\d \Cone{L}}/\min \abs{G_{s}}
    \]
    where $\abs{\d \Cone{L}}$ is the number of cubes in $\d \Cone{L}$ and $\tilde b_*$ are the reduced Betti numbers.
    Furthermore, if $i = \dim L + 1$, then
    \[
    b_{i}(\gG; \F)/n \leq b_{i-1}(L; \F)
    \]
\end{Theorem}
\begin{proof}

    Since all $G_{L}$ stabilizers are finite, $\gG$ acts freely on $\cu$, so we want to estimate the homology of $X=\cu/\gG$.
    Let $p: X \to \Cone{L}$ denote the projection and let $Y=p^{-1}(\d \Cone{L})$.
    $X$ is tiled by copies of $\Cone{L}$ intersecting along $Y$.

    For a cube $\gs$ in $\Cone{L}$ the number of preimages $\abs{p^{-1}(\gs)}$ is $n/\abs{G_{\min \gs}}$, where $\min \gs$ is the smallest element in the chain corresponding to $\gs$, and $G_{\min \gs}$ is the corresponding special subgroup.
    In particular, the cubes in $\Cone{L} - \d \Cone{L}$ have $n$ preimages, and the cubes in $\d \Cone{L}$ have at most $n/\min \abs{G_{s}}$ preimages.
    Thus in the long exact sequence
    \[
    \dots \to H_{i}(Y;\F) \to H_{i}(X;\F) \to H_{i}(X, Y;\F) \to H_{i-1}(Y;\F) \to \dots
    \]
    the dimensions of the first and the last terms are bounded by $\abs{Y}\leq n \abs{\d \Cone{L}}/\min \abs{G_{s}}$ and the relative term is isomorphic to $\oplus_{n} H_{i}(\Cone{L}, \d \Cone{L}; \F)$, so its dimension is $n \tilde b_{i-1}(L; \F)$.
    Therefore,
    \[
    \abs{b_{i}(X; \F)/n -\tilde b_{i-1}(L; \F)} \leq 2\abs{\d \Cone{L}}/\min \abs{G_{s}}.
    \]
    If $i = \dim L + 1$, then $H_{i}(Y;\F) = 0$ so we get $b_{i}(\gG; \F)/n \leq b_{i-1}(L; \F)$.
\end{proof}
\begin{corollary}\label{c:modpl2}
    If $G_{L}$ is a graph product of finite groups, then
    \[
    \abs{\lb_{i}(G_L; \F) -\tilde b_{i-1}(L; \F)} \leq 2\abs{\d \Cone{L}}/\min \abs{G_{s}}.
    \]
\end{corollary}
\begin{remarks}
    A similar estimate in terms of thickness holds for torsion free uniform lattices in chamber transitive buildings.
    For sufficiently large thickness $L^{2}$-Betti numbers of such buildings were computed in \cite{ddjo07}*{Theorems 10.4(ii) and 13.8}.
    The exact formula there has extra terms involving homology of certain subcomplexes of $\d \Cone{L}$, however their contribution is of order 1/thickness.

    For right-angled Artin groups $A_L$, \cite{aos21} gives the exact computation $\lb_i(A_L;\F)=\tilde b_{i-1}(L;\F)$, so $|\lb_i(G_L;\F)-\lb_i(A_L;\F)|\leq 2|\d \Cone{L}|/\min|G_s|$.
    We conclude that the homology growth of a graph product of finite groups converges to the homology growth of the corresponding right-angled Artin group as the minimum orders of the groups go to infinity.
\end{remarks}

\subsection*{Flag no-square triangulations}

A flag simplicial complex $L$ is \emph{no-square} if each simplicial cycle of length four has a diagonal.
If $L$ is equipped with a piecewise spherical metric with all edge lengths = $\pi/2$ (i.e. $L$ is \emph{all right}), then the flag no-square condition is equivalent to $L$ not having any closed geodesics of length $\le 2\pi$.
Suppose $L$ is a flag no-square simplicial complex, and $G_L$ is a graph product of finite groups based on $L$.
In this case, the right-angled building $\cu(G_L, \Cone{L})$ admits a $G_L$-equivariant CAT$(-1)$ metric, and in particular, $G_L$ is a hyperbolic group \cite{m96}, see also \cite{d08}*{Corollary 18.3.10}.
Note that any graph product of nontrivial groups based on a $4$-cycle contains free abelian subgroups of rank $2$, hence if $L$ is not no-square then $G_L$ is not hyperbolic.
The construction of the flag no-square triangulations we require is due to Przytycki and \'Swi\k{a}tkowski.
\begin{Theorem}[\cite{ps09}*{Corollary 2.14}]\label{t:nosquare}
    Let $L$ be a simplicial complex of dimension $\le 3$.
    Then there is a subdivision $L'$ of $L$ which is flag no-square.
\end{Theorem}
\begin{remark}
    Theorem \ref{t:nosquare} only holds in dimensions $\le 3$, for instance there is no flag no-square triangulation of any $4$-dimensional homology sphere \cite{d08}*{Proposition I.6.6}.
    In this generality, this is due to Moussong.
    Vinberg had earlier shown the non-existence of compact right-angled convex polytopes in hyperbolic $n$-space for $n > 4$ (the dual of the boundary would be a flag no-square triangulation of $S^n$ for $n > 3$).
\end{remark}

\subsection*{Dependence on the coefficient field}

Note that combining Theorem \ref{t:modpl2} with Theorem \ref{t:nosquare} immediately gives some low-dimensional hyperbolic groups with $\lb_{i}(G; \F_p) > \lb_{i}(G; \Q)$ (for example for $p=2$, we can take $G_L$ to be a graph product of large finite groups over a flag no-square triangulation of $\rr P^2$).
In fact, a construction of flag no-square triangulations due to Osajda \cite{o13} gives examples in all dimensions:
\begin{Theorem}
    For any prime $p$ and $i \geq 2$ there is a hyperbolic right-angled Coxeter group $W$ with
    \[
    \lb_i(W; \F_p) > \lb_i(W; \Q).
    \]
\end{Theorem}
\begin{proof}
    Suppose that we have a flag no-square $L$ with $b_i(L; \F_p) > b_i(L; \Q)$.
    If $G_L$ is a graph product of $(\zz/2)^N$ over $L$, then for $N \mge 0$ we have $\lb_{i+1}(G_L; \F_p) > \lb_{i+1}(G_L; \Q)$ by Theorem \ref{t:modpl2}.
    The group $G_L$ is also a right-angled Coxeter group with flag no-square nerve $L'$ (obtained from $L$ by replacing each vertex with an $(N-1)$-simplex).
    The commutator subgroup of $G_L$ is the fundamental group of a locally CAT$(-1)$ cube complex $X_{L'}$ where the links are all isomorphic to $L'$.
    By Theorem \ref{t:modpl2} we have that $b_{i+1}(X_{L'}; \F_p) > b_{i+1}(X_{L'}; \Q)$.

    Now, the ``simplicial thickening'' of a cube complex $C$ is a simplicial complex $\Th(C)$ with the same vertex set, where vertices span a simplex if and only if they are contained in the same cube.
    It is easy to see that $\Th(C)$ is homotopy equivalent to $C$.
    Osajda showed that if the link of each vertex in $C$ is flag no-square, then the link of each vertex in $\Th(C)$ is flag no-square \cite{o13}*{Lemma 3.2}.
    Therefore, the thickening $\Th(X_{L'})$ has flag no-square links.
    By passing to a further finite cover, we can assume that the injectivity radius is large, which implies it is flag no-square.
    Therefore, we can use $\Th(X_{L'})$ as our next nerve to get dependence on field coefficients in one higher dimension.
\end{proof}
\begin{remark}
    Theorem \ref{t:nosquare} only holds in dimensions $\le 3$, for instance there is no flag no-square triangulation of any $4$-dimensional homology sphere \cite{d08}*{Proposition I.6.6}.
    In this generality, this is due to Moussong.
    Vinberg had earlier shown the non-existence of compact right-angled convex polytopes in hyperbolic $n$-space for $n > 4$ (the dual of the boundary would be a flag no-square triangulation of $S^n$ for $n > 3$).

    This is a barrier to producing higher dimensional hyperbolic examples without the $\Fp$-Singer property.
    On the other hand, in \cite{js03}, Januszkiewicz and \'Swi\k{a}tkowski constructed flag no-square triangulations of $n$-dimensional pseudomanifolds for any degree $n$.
    A corollary is the existence of word-hyperbolic right-angled Coxeter groups of arbitrarily high cohomological dimension.

    The top-dimensional homology of the examples in \cite{js03} does not depend on the coefficient field.
    It would be interesting if one could construct $d$-dimensional, flag no-square $L$ with $H_d(L; \F_p) \ne 0$ and $H_d(L, \Q) = 0$.
    Osajda \cite{o13} describes a simple construction of flag no-square $L$ with $H_k(L) \ne 0$ for any given $k$, however this homology does not occur in the top dimension of $L$.
\end{remark}

\subsection*{Virtual duality groups}

We also recall the criterion for a graph product of finite groups to be a virtual duality group.
This will be used in the proof of Theorem \ref{rational}.
\begin{theorem}[\cite{dm02}*{Corollary 6.4}]\label{dualitycriterion}
    A graph product $G_L$ of finite groups over a flag complex $L$ is a virtual duality group of dimension $n$ if and only if for every simplex $\gs$ of $L$ (including the empty simplex), $\overline H_{*}(L-\gs;\Z)$ is torsion-free and concentrated in dimension $n-1$.
\end{theorem}
In particular, for any flag triangulation of $S^3$, a graph product of finite groups $G_{S^3}$ is a virtual duality group of dimension $4$.

\section{Thickenings}\label{s:thickenings}
In this section we construct our seed manifold $N$ by building a manifold model for the classifying space of a sufficiently deep torsion free finite index subgroup in an appropriately chosen graph product.
We begin with a general statement, which tells us when we can thicken these classifying spaces to manifolds of less than twice their dimension.
\begin{Theorem}\label{t:manifold2}
    Let $G_L$ be a graph product of finite groups based on a $(d-1)$-dimensional flag complex $L$.
    If $d>3$ and $H_{d-1}(L; \F_{2}) = 0$ then $G_L$ has a finite index subgroup $\gG_L$ which is the fundamental group of a compact aspherical $(2d-1)$-manifold with boundary.
\end{Theorem}

Before proving this theorem, we need to recall a number of previous results and constructions.
We will construct aspherical manifolds using the reflection group method.
To that end, the first step is to embed a finite index subgroup of the graph product $G_L$ into a right-angled Coxeter group based on a flag complex $OL$, the octahedralization of $L$ as in the Introduction.
This leads to the first key ingredient in our construction, which is commensurability between different graph products over the same flag complex.
Recall that two groups $G$ and $H$ are commensurable if they have isomorphic finite index subgroups.
They are \emph{strongly commensurable} if these finite index subgroups have the same index in $G$ and $H$.
Januszkiewicz and \'Swi\k{a}tkowski proved the following:
\begin{Theorem}[\cite{js01}*{Theorem 1}]\label{t:commensurable}
    Suppose that $G_L$ and $G_L'$ are two graph products of groups over the same flag complex $L$.
    Suppose that for all $v \in S$, the group $G_v$ is strongly commensurable to $G_v'$.
    Then $G_L$ is strongly commensurable to $G_L'$.
\end{Theorem}
We shall also need the following lemma, which follows immediately from the normal form for graph products.
\begin{lemma}\label{l:subgroup}
    Suppose that $G_L$ is a graph product of groups over a flag complex $L$, with vertex set $S$.
    For each $s \in S$, choose a subgroup $H_s$ of $G_s$.
    Then the graph product $H_L$ corresponding to the $H_s$ embeds as a subgroup of $G_L$.
\end{lemma}

These imply the following corollary.
\begin{corollary}\label{c:emb}
    A graph product $G_L$ of finite groups over a flag complex $L$ has a torsion free finite index subgroup $\gG_L$ which embeds into $W_{OL}$.
\end{corollary}
\begin{proof}
    By Lemma \ref{l:subgroup} $G_L$ embeds into the graph product over $L$ of direct products $G_v \times D_\infty$, which by Theorem \ref{t:commensurable} is strongly commensurable to the graph product over $L$ of $D_{\infty}$ of infinite dihedral groups, since both $G_v \times D_\infty$ and $D_\infty$ contain $\zz$ as a subgroup of index $2\abs{G_v}$.
    The graph product of $D_{\infty}$ has a natural structure of a right-angled Coxeter group $W_{OL}$, hence its commutator subgroup has finite index and torsion-free.
    Therefore intersecting a common finite index subgroup of the latter two graph products with this commutator subgroup and then further intersecting with $G_L$ produces the desired $\gG_{L}$.
\end{proof}

In order to construct a contractible $n$-manifold on which the reflection group $W_{OL}$ (and hence $\Gamma_L$) acts, we embed this group into the reflection group of a flag triangulation of an $(n-1)$-sphere $W_{S^{n-1}}$.
This brings us to the second key ingredient, which is a van Kampen style embedding theory for octahedralizations.
To construct ``low-dimensional'' manifold models for $B\gG_L$, we use the ``if'' direction of the following embedding theorem for octahedralizations that we proved with Davis in \cite{ados16}:
\begin{Theorem}[\cite{ados16}]\label{t:manifold}
    Let $L$ be a $(d-1)$-dimensional flag complex.
    If $d \ne 3$, $OL$ embeds as a full subcomplex of a flag triangulation of $S^{2d-2}$ if and only if $H_{d-1}(L; \F_{2}) = 0$.
\end{Theorem}

Note that this theorem implies an embedding in one less dimension than general position guarantees.
Since the group $\Gamma_L$ will have infinite index in the Coxeter groups used to build the contractible manifold, the construction naturally leads to non-compact aspherical manifolds.
To perform the reflection trick in the next section, we need to compactify such manifolds by adding a boundary.
A complete obstruction for doing so was developed by Siebenmann in his thesis \cite{s65a}.
Fortunately, when $d>3$, one arrives at a manifold whose end has the same fundamental group as its interior and the following theorem shows that the obstruction vanishes.
\begin{Theorem}[$\pi$--$\pi$ compactification theorem]\label{t:pi-pi}
    Suppose $N^{n}$ is a one-ended PL manifold, $n>5$.
    Also suppose $ \pi_1$ of the end $\ge$ is stable and the natural map $\pi_{1}(\ge) \to \pi_1(N)$ is an isomorphism.
    Then $N^{n}$ is PL homeomorphic to the interior of a compact PL manifold with boundary if and only if it has finite homotopy type.
\end{Theorem}

The $\pi$--$\pi$ theorem follows directly from the results in \cite{s65a} but is not explicitly stated there, so we give the derivation below.
\begin{proof}[Proof of Theorem \ref{t:pi-pi}] The only if direction is clear, so we suppose that $N$ has finite homotopy type.
    Then our isomorphism assumption implies that $\pi_{1}(\ge)$ is finitely presented, and by \cite{s65a}*{Theorem 3.10} the end has arbitrarily small 1-neighborhoods.
    This means that we have an arbitrarily small connected submanifold $(V,\d V)$ with connected boundary and compact complement of the interior, and the inclusions induce isomorphisms

    \[
    \pi_{1}(\d V) \cong \pi_{1}(V) \cong \pi_{1}(\ge).
    \]
    Since $\pi_{1}(\ge) \cong \pi_{1}(N)$ by our assumption, the van Kampen theorem applied to the decomposition $N=(N-\mathring{V})\cup_{\d V} V$ shows that also $\pi_{1}(N - \mathring{V}) \cong \pi_{1}(\d V)$.
    Then \cite{s65a}*{Complement 6.6(b)} implies that $V$ is finitely dominated, and \cite{s65a}*{The Sum Theorem 6.5} implies that the finiteness obstruction for $V$ vanishes, and the claim follows from the Main Theorem of \cite{s65a}.
\end{proof}
\begin{remark}
    We want to emphasize that in the $\pi$--$\pi$ situation one does not need to assume that the end is finitely dominated, but gets it for free when the interior is finitely dominated.
\end{remark}

We are now ready for the proof of the thickening theorem.
\begin{proof}[Proof of Theorem \ref{t:manifold2}] By Corollary \ref{c:emb}, we have a torsion free finite index subgroup $\gG_{L}$ of $G_L$ which is also a subgroup of $W_{OL}$.
    Theorem \ref{t:manifold} gives us an embedding of $OL$ as a full subcomplex of a flag triangulated sphere $S^{2d-2}$.
    Since $OL$ is a full subcomplex, $W_{OL}$ is a subgroup of $W_{S^{2d-2}}$ and on the level of Davis complexes $\gS_{OL} \subset \gS_{S^{2d-2}} $ is a $W_{OL}$-stable subspace.

    The quotient $N:=\gS_{S^{2d-2}}/\Gamma_L$ is an aspherical $(2d-1)$-manifold with fundamental group $\Gamma_L$.
    This manifold has finite type since it is homotopy equivalent to the finite complex $\cu(G_L,\Cone{L})/\Gamma_L$, but the manifold is not itself compact.
    To fix this, we will use the fact that $N$ has another $d$-dimensional classifying space for $\Gamma_L$ embedded inside of it, namely the $d$-complex $Y:=\gS_{OL}/\Gamma_L$.

    Pick an exhaustion $C_i$ of $Y$ by finite subcomplexes.
    For each $C_i$ pick a closed PL regular neighborhood $U_i$ in $N$ such that $U_i$ is contained in the interior of $U_{i+1}$.
    Then $N'=\bigcup U_i$ is an open $(2d-1)$-manifold containing $Y$ and the embedding $Y\into N'$ is a homotopy equivalence.
    Since $d>3$, $C_i$ has codimension $\geq 3$ in $N'$, so the map $(N'-C_i)\to N'$ is an isomorphism on $\pi_1$, and the same is true for the map from the complement of the regular neighborhood $(N'-U_i)\to N'$.
    Therefore $\pi_1$ of the end $\ge$ of $N'$ is stable and the natural map $\pi_1(\ge) \to \pi_1(N')$ is an isomorphism.
    Moreover the dimension $2d-1$ is greater than five, and $N'$ is homotopy equivalent to $Y$ and hence to $N$, so it has finite type.
    Therefore, it follows from Theorem \ref{t:pi-pi} that $N'$ is PL homeomorphic to the interior of a compact PL manifold with boundary.
    This finishes the proof.
\end{proof}
The last paragraph of the proof only uses the fact that $Y$ has finite type and has codimension three inside a manifold of dimension greater than five.
So, it gives the following result which may be of independent interest.
\begin{proposition}
    Fix $n>5$.
    Let $Y^d$ be a complex of dimension $d\leq n-3$.
    If $Y$ has finite type and PL embeds in an $n$-manifold, then $Y$ is homotopy equivalent to a compact PL $n$-manifold with boundary.
\end{proposition}

\subsection*{Construction of the $7$-dimensional seed manifold}

Finally, combining Corollary \ref{c:modpl2}, Theorem \ref{t:nosquare}, and Theorem \ref{t:manifold2} gives the following:
\begin{Theorem}\label{t:hypmanifold}
    For any odd prime $p$, there is a compact, aspherical $7$-manifold with boundary $(N, \d N)$ such that $\pi_1(N)$ is special hyperbolic and $\lb_4(\pi_1(N); \F_p) \ne 0$.
\end{Theorem}
\begin{proof}
    Let $L$ be a flag no-square triangulation of the complex $S^2 \cup_p D^3$, where $D^3$ is glued onto $S^2$ by a degree $p$ map.
    This triangulation exists by Theorem \ref{t:nosquare}.
    Let $G_L$ be a graph product of large finite groups over this $L$.
    Then $G_L$ is virtually special hyperbolic and for any finite index subgroup $\gG_L$ we have $\lb_4(\Gamma_L;\F_p)=\lb_4(G_L; \F_p)|G_L/\Gamma_L| \ne 0$ by Corollary \ref{c:modpl2}.
    We can pick $\gG_L$ to be special and, by Theorem \ref{t:manifold2}, to be $\pi_1$ of a compact aspherical $7$-manifold $N^{7}$.
\end{proof}

\section{A hyperbolic reflection group trick}\label{s:hrt}

In this section, we describe how to build a closed, aspherical manifold $\M$ out of copies of a compact aspherical manifold with boundary $N$.
The idea is to combine the Davis reflection group trick with the strict hyperbolization procedure of Charney--Davis \cite{cd95c}; it has the advantage of preserving hyperbolicity (and as we shall see, a host of other properties including virtual specialness).
The output is very similar to the relative strict hyperbolization procedure of Belegradek \cites{b06,b07}, which combines the relative hyperbolization of Davis--Januszkiewicz--Weinberger \cite{djw01} with strict hyperbolization.
Our hyperbolic reflection group trick has one major difference with Belegradek's procedure: our output has many disjoint copies of the input $N$, whereas his output has only one.
This makes our construction easier to work with, but doesn't give all of the applications in \cites{b06, b07}.

\subsection*{Strict hyperbolization}

Charney and Davis define a hyperbolization procedure that converts any piecewise Euclidean, locally CAT(0), cube complex $K$ of dimension $\le n$ into a piecewise hyperbolic, locally CAT$(-1)$, polyhedron $hK$ (the output $hK$ depends on $n$).
Roughly speaking, this replaces every $n$-cube with a hyperbolic $n$-manifold with boundary so that $k$-cubes are replaced by totally geodesic $k$-dimensional submanifolds.
Charney and Davis construct the hyperbolized $n$-cube $\CD^n$ by cutting a closed arithmetic hyperbolic $n$-manifold $A^n$ along a suitable collection of codimension-one totally geodesic submanifolds.
Let $B_{n}$ denote the finite Coxeter group of type $B_{n}$---the symmetry group of the cube $[0,1]^n$.
More precisely, Charney and Davis prove:
\begin{theorem}[\cite{cd95c}*{Theorem 6.1}]\label{t:arithmetic}
    For each $n > 0$, there is a closed connected hyperbolic $n$-manifold $A^n$, a collection $\W_1, \dots ,\W_n$ of codimension-one closed connected submanifolds, and an isometric action of $B_n$ on $A^n$ satisfying:
    \begin{itemize}
        \item $\W_i$ is a component of the fixed set of a standard generator of $B_n$.
        \item Each $\W_i$ is totally geodesic.
        \item The $\W_i$'s intersect orthogonally.
        \item $A^n$ and $\W_i$ are orientable.
        \item $\W_1 \cap \dots \cap \W_n$ is a single point $\{y\}$.
        \item $B_n$ fixes $y$ and the representation on $T_y A^n$ is equivalent to the standard representation.
    \end{itemize}
\end{theorem}

Cutting $A^n$ along the $\{W_i\}$ leaves a compact, connected, orientable hyperbolic $n$-manifold with corners $\CD^n$ with isometric $B_{n}$ action.
Each $W_{i}$ is itself cut by other $W_{j}$'s, the boundary of $\CD^n$ contains two disjoint copies of each cut-open $W_{i}$ and is covered by these copies.
A codimension $k$ face of $\CD^n$ is a nonempty $k$-fold intersection of these cut-open copies.
\begin{theorem}[\cite{cd95c}*{Corollary 6.2}]\label{t:CD}
    The manifold $\CD^{n}$ has the following properties:
    \begin{itemize}
        \item There is a degree one map $f: \CD^n \to [0,1]^n$ which induces a $B_n$-equivariant isomorphism between the face posets of $CD^{n}$ and $[0,1]^n$.
        \item Each codimension-one face of $\CD^n$ is connected and totally geodesic, and the faces intersect orthogonally.
        \item Each $0$-dimensional face is a single point.
    \end{itemize}
\end{theorem}

Note that although $W_{i}$ are totally geodesic and connected, they are not necessarily convex because the geodesics between points in $A^{n}$ are not unique.
So lower-dimensional faces of $\CD^n$ are totally geodesic submanifolds, which however are generally not connected.

Given a cube complex $K$ of dimension $\leq n$, thought of as a collection of cubes glued via isometries between faces, Charney and Davis define the \emph{$n$-dimensional hyperbolization} $hK$ of $K$ by gluing via isometries copies of faces of $CD^{n}$ in the same combinatorial pattern.
In general this relies on the $B_{n}$ isometric action and the equivariance of the face posets isomorphism.
If $K$ is \emph{foldable}, i.e. $K$ admits a cellular map $p: K \to [0,1]^n$ such that the restriction to any cell is a combinatorial isomorphism, then $hK$ is the same piecewise hyperbolic polyhedron as the fiber product
\[
hK = \{(k,x) \in K \times \CD^n \mid p(k) = f(x)\}
\]

Of course, there is a choice of $\CD^n$, but for any such choice the hyperbolization procedure satisfies the following:
\begin{theorem}[Charney--Davis \cite{cd95c}*{Corollary 7.1.}]\label{t:hK}
    Let $K$ be a locally $CAT(0)$ cube complex of dimension $\le n$.
    Then there is a piecewise hyperbolic, locally CAT$(-1)$ space $hK$, and a map $q: hK \to K$ such that:
    \begin{enumerate}
        \item For each $k$-cube $C^k$ in $K$, $q^{-1}(C^k)$ is isometric to a $k$-dimensional face of $\CD^n$.
        If $J$ is a subcomplex of $K$, then $q^{-1}(J)$ is isometric to $hJ$.
        \item The link of $q^{-1}(C^k)$ is isometric to the link of $C^k$ in $K$.
        \item $h$(a point) = $\text{a point}$.
        \item If $J$ is a totally geodesic subcomplex of $K$, then $hJ$ is totally geodesic in $hK$.
    \end{enumerate}
\end{theorem}

Recall that if $K$ is a locally CAT(0) cube complex and $J$ a subcomplex, then there is an easy to verify link condition to check that $J$ is totally geodesic.
In particular, $J$ is totally geodesic if for each vertex $v \in K$, $\Lk_J(v)$ is a full subcomplex of $\Lk_K(v)$; i.e. if the vertices of a simplex are contained in $\Lk_J(v)$, then that simplex is contained in $\Lk_J(v)$.

\subsection*{A hyperbolic reflection group trick}

Suppose $(N, \d N)$ is a manifold with boundary, and $\d N$ is triangulated as a flag simplicial complex (this can always be ensured by taking a barycentric subdivision).
For the moment, we forget about the manifold $N$, and concentrate on the flag triangulation of the boundary, which we will denote by $\d$.
It determines a right-angled Coxeter group $W_{\d}$, and hence a locally CAT(0) cube complex $P_{\d}$ whose fundamental group is the commutator subgroup of $W_{\d}$.
The complex $P_{\d}$ can be also described as the basic construction $\cu( (\zz/2)^{\abs{\d^{0}}}, \Cone{\d})$, so it folds onto the Davis chamber $\Cone{\d}$.

If we remove a neighborhood of each cone vertex and replace with a copy of $N$, this is precisely the output of a right-angled Davis reflection group trick.
However, the fundamental group will generally not be hyperbolic.
\begin{figure}

    \tikzset{ house/.pic={
    \draw (0,-1)--(4,-1)--(4,3) to [out=180, in=-45] (2,4) to [out=-135, in=0] (0,3)--cycle;
    \draw (1,1) to [out=0, in=-45] (1.1,2) to [out=135, in=180] (2,3) to [out=0, in=45] (2.9,2) to [out=-135, in=180] (3,1) ;
    \draw (1.5,2) to [out=-30, in=-150] (2.5,2) ;
    \draw (1.6,1.95) to [out=30, in=150] (2.4,1.95) ; }}

    \tikzset{ CD/.pic={
    \draw (-1,5) to [out=0, in=-45] (-.9,7) to [out=135, in=180] (0,8) to [out=0, in=45] (.9,7) to [out=-135, in=180] (1,5) ;
    \draw (-.5,7) to [out=-30, in=-150] (.5,7) ;
    \draw (-.4,6.95) to [out=30, in=150] (.4,6.95) ; } }

    \tikzset{ cube/.pic={
    \draw (0,0) -- (3,-1)-- (5,0) -- (5,3) -- (2,4) -- (0,3) -- cycle;

    \draw (3,2)-- (3,-1);
    \draw (3,2)-- (5,3);
    \draw (3,2)-- (0,3);

    \draw [dashed] (2,1) -- (0,0) ;
    \draw [dashed] (2,1) -- (5,0) ;
    \draw [dashed] (2,1) -- (2,4) ; } } \tikzset{ dN/.pic={
    \draw [fill=white] (0,0) ellipse (1cm and .5cm);
    \draw (-.5,0) to [out=30, in=150] (.5,0) ;
    \draw (-.4,.05) to [out=-30, in=-150] (.4,.05) ; } } \tikzset{ N/.pic={
    \draw[fill=white] (-1,0) to [out=90, in=-85] (-1,1) to [out=95, in=180] (0,2) to [out=0, in=85] (1,1) to [out=-95, in=90] (1,0) ;
    \draw pic[transform shape] {dN} ; } }
    \begin{tikzpicture}[xscale=0.75, yscale = .75]
        \draw pic[rotate=90] {N} ;
        \node at (-.7,1.8) {$(N,\d N)$ };
        \begin{scope}[xshift=5cm, yshift=-1.5cm) ]
            \node at (1, 4) {$P_{\d}$ }; \pic[transform shape] {cube}; \pic[transform shape, yscale=0.5, xscale=0.5, rotate=115] at (0.31, 0.32) {dN} ;
        \end{scope}
        \begin{scope}[xshift=-3cm, yshift=-8cm) ]
            \node at (1, -1) {$hP_{\d}$ }; \pic[transform shape] {cube}; \pic[transform shape, yscale=0.5, xscale=0.5, rotate=115] at (0.31, 0.32) {dN} ; \pic[transform shape, scale=.3] at (2.5,1.3) {CD}; \pic[transform shape, scale=.3, rotate=-115] at (2.5,1.8) {CD}; \pic[transform shape, scale=.3, rotate=125] at (2.8,1.8) {CD};
        \end{scope}
        \begin{scope}[xshift=5cm, yshift=-8cm) ]
            \node at (1, -1) {$hP_{\d}^{N}$ }; \pic[transform shape] {cube}; \pic[transform shape, yscale=0.5, xscale=0.5, rotate=115] at (0.24, 0.27) {N} ; \pic[transform shape, yscale=0.5, xscale=0.5, rotate=-65] at (4.8, 2.7) {N} ; \pic[transform shape, yscale=0.5, xscale=0.5, rotate=60] at (.3, 2.8) {N} ; \pic[transform shape, yscale=0.5, xscale=0.5, rotate=-170] at (3.1, -.76) {N} ; \pic[transform shape, scale=.3] at (2.5,1.3) {CD}; \pic[transform shape, scale=.3, rotate=-135] at (2.5,2.3) {CD}; \pic[transform shape, scale=.3, rotate=115] at (2.5,2) {CD};
        \end{scope}

        \draw[->] (1,0) to [out=-30, in=150] (4,0) ;
        \draw[->] (4.5,-2) to [out=-180, in=0] (0,-4) ;
        \draw[->] (2.5, -6) to [out=-30, in=150] (4,-6) ;
    \end{tikzpicture}

    \caption{The construction of $hP_{\d}^N$.
    Starting with a triangulation of $\d$, we first construct a locally CAT(0) cube complex $P_\partial$ and then a locally CAT(-1) space $hP_\partial$ where the vertices all have links isomorphic to $\d N$.
    Then, each of these links is replaced with a copy of $N$ to form $hP_\partial^N$.}\label{f:hP}
\end{figure}
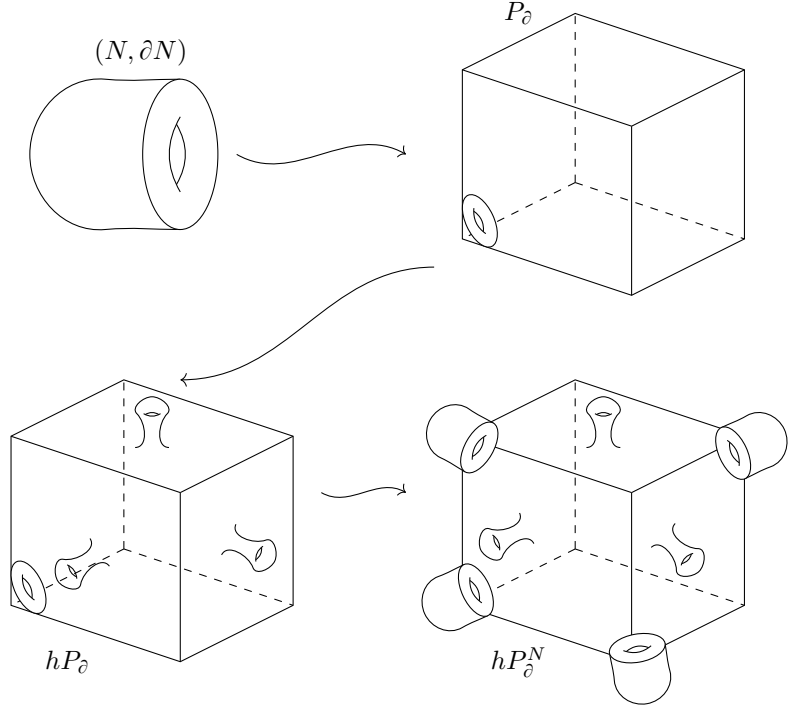

Instead, we first apply the Charney--Davis strict hyperbolization procedure to $P_{\d}$.
This produces a piecewise hyperbolic locally CAT$(-1)$ polyhedron $hP_{\d}$.
Since the strict hyperbolization preserves vertices and links it has the same number of singular vertices as $P_{\d}$ and the link of each one is still isomorphic to $\d$.

We now choose $\ge$ small enough so that in $\CD^{n}$ the $\ge$-ball around a vertex is isometric to the intersection of the $\ge$-ball in $\hh^{n}$ with an octant.
We resolve the singularities by removing $\ge$-neighborhoods of each vertex in $hP_{\d}$ and replacing them with copies of $N$.

This produces a closed manifold $\M$ equipped with a map to $hP_{\d}$.
When we want to emphasize the building blocks involved in the construction of $\M$, we will also denote it by $hP^N_{\d}$.

In some sense, the rest of the paper is concerned with showing that, for an appropriate choice of Charney--Davis piece $CD^n$ and triangulation $\d$, the ambient manifold $\M$ inherits many properties from its seed $N$.

\subsection*{Resolved chambers}

We let $\Lhc{N}{\d}$ denote the hyperbolized Davis chamber $\hc{\d}$ with a $\ge$-neighborhood of the cone point replaced with a copy of $N$, and similarly we denote by $K^{N}_{\d}$ the same procedure applied to the original Davis chamber.
We will refer to these as \emph{resolved} (hyperbolized) Davis chambers.

The symmetry of $P_{\d}$ lifts to $\M$, so $\M$ can be described as a basic construction obtained by reflecting around $\Lhc{N}{\d}$ using $(\zz/2)^{\abs{\d^0}}$.
The mirrors are hyperbolizations of the mirrors $K_{v}$ of the Davis chamber for $v \in \d$, and we will denote them by $h(K_{v})$.
They are isomorphic to $\hc{\Lk(v)}$.

\subsection*{Universal cover of $\M$}

Since $\M$ is a basic construction $\cu( (\zz/2)^{\abs{\d^0}}, \Lhc{N}{\d} )$, its universal cover $\widetilde \M$ is itself a basic construction, where the base space is the universal cover of $\Lhc{N}{\d}$ and the mirror structure is lifted from the mirror structure downstairs.
This is a general fact about basic constructions, and we review the general statement.
Let $(X, \{X_s\}_{s \in S})$ be a mirrored space, $W$ the corresponding RACG, and $\cu(W,X)$ the associated basic construction.
Let $\pi: \widetilde X \to X$ be the universal cover of $X$, and let $\widetilde S$ be the set of components of preimages of $X_s$ over all $s \in S$.
We will denote such a component of $\pi^{-1}(X_s)$ by $\tilde s$.
The map $\pi$ induces a natural map $\widetilde S \to S$ sending each $\tilde s$ to $s$.
We define a RACG $\widetilde W$ with generating set $\widetilde S$ as follows: for $\tilde s \ne \tilde t$, set $m_{\tilde s \tilde t} = 2$ if the lifts $\tilde s$ and $\tilde t$ intersect, and set $m_{\tilde s \tilde t} = \infty$ otherwise; it is clear that $\pi$ determines a group homomorphism $\widetilde W \to W$.
Let $\widetilde \cu = \cu(\widetilde W, \widetilde X)$ be the corresponding basic construction.

By \cite{d08}*{Theorem 9.1.3}, $\widetilde \cu$ is simply connected, and the map $(\pi,\pi): \widetilde \cu \to \cu(W,X)$ is the universal covering map.

The $\pi_1(X)$-action on $\widetilde X$ induces an action of $\pi_1(X)$ on $\widetilde W$ by automorphisms, so we can form the semidirect product $\widetilde W \rtimes \pi_1(X)$.
This semidirect product acts on $\widetilde \cu$ by setting
\[
(\tilde w,g)\cdot(\tilde w',x) = (\tilde w g(\tilde w'), gx).
\]
in fact, the semidirect product is precisely the group of lifts of the $W$-action to $\widetilde \cu$.
Therefore, there is an exact sequence
\[
1 \to \pi_1(\cu(W,X)) \to \widetilde W \rtimes \pi_1(X) \to W \to 1.
\]

Note that for $\M$, the group $W$ is finite, so $\pi_1(\M)$ is finite index in $\widetilde W \rtimes \pLhc{N}{\d}$.

\subsection*{Basic properties of the hyperbolic reflection group trick (Proof of Theorem \ref{t:hypdavistrick}(1)--(4))}

We collect some basic properties of the hyperbolic reflection group trick in the theorem below.
It follows from previous work of Davis--Januszkiewicz--Weinberger \cite{djw01} and Belegradek \cite{b07}, but since our construction is slightly different, we provide a sketch of the proof.
\begin{prevthm}
    {\ref{t:hypdavistrick}(1)--(4)} The manifold $\M$ satisfies the following properties:
    \begin{itemize}
        \item[(1)] $\M$ retracts onto $N$, hence $\pi_1(N)$ injects into $\pi_1(\M)$.
        \item[(2)] If $N$ is aspherical, then $\M$ is aspherical.
        \item[(3)] If $N$ is $\F $-aspherical, then $\M$ is $\F $-aspherical.
        \item[(4)] $\pi_1(\M)$ is relatively hyperbolic relative to the collection of subgroups corresponding to the $2^{\abs{\d^{0}}}$ copies of $N$.
        Therefore, if $\pi_1(N)$ is hyperbolic, then $\pi_1(\M)$ is hyperbolic.
    \end{itemize}
\end{prevthm}
\begin{proof}
    The manifold $\M$ folds onto $\Lhc{N}{\d}$, which maps onto $K^{N}_{\d}$ fixing $N$, which in turn deformation retracts onto $N$.
    This induces a retraction of $\M$ onto $N$, and hence an injection $\pi_1(N) \rightarrow \pi_1(M)$.

    Thus, in the universal cover $\widetilde \M$ of $\M$, $N$ lifts to copies of its universal cover $\widetilde N$.
    If we replace each copy of $\widetilde N$ with a cone on its boundary $\d \widetilde N$, we obtain a branched cover $\overline{hP}_{\d}$ of $hP_{\d}$, and $\pi_1(\M)$ now acts with cone vertex stabilizers equal to conjugates of $\pi_1(N)$.
    Since $\widetilde N$ is simply connected, the map $\widetilde \M\to \overline{hP}_{\d}$ collapsing copies of $\widetilde N$ to cone points is a $\pi_1$-isomorphism, so $\overline{hP}_{\d}$ is simply connected.

    Lifting the metric from $hP_{\d}$ gives $\overline{hP}_{\d}$ the structure of a piecewise hyperbolic space.
    The links of the cone vertices are disjoint copies of $\d \widetilde N$, which covers $\d N$, and hence are flag complexes.
    Therefore, $\overline{hP}_{\d}$ is locally CAT$(-1)$ as away from the cone points $\overline{hP}_{\d}$ is locally isometric to $hP_{\d}$.
    Since $\overline{hP}_{\d}$ it is simply connected it is CAT$(-1)$ and therefore contractible.
    Now, if $N$ is aspherical, then each $\widetilde N$ is contractible and the collapse map is a homotopy equivalence, so $\widetilde \M$ is contractible, and therefore $\M$ is aspherical.

    Similarly, if $N$ is $\F $-aspherical, then each $\widetilde N$ is $\F $-acyclic and the collapse map is an $\F $-homology equivalence, so $\widetilde \M$ is $\F $-acyclic, which means that $\M$ is $\F $-aspherical.

    Next, since the $1$-skeleton of $\overline{hP}_{\d}$ of quasi-isometric to $\overline{hP}_{\d}$, it is a $\gd$-hyperbolic graph.
    To prove relative hyperbolicity we need to show that this graph is \emph{fine} in the sense of Bowditch \cite{b12}.
    This amounts to showing that for each cone, if we delete the cone point from the graph, then in the metric of the deleted graph any bounded subset of the vertices in the link $\overline{hP}_{\d}$ is finite.
    Since the cones are convex in $\overline{hP}_{\d}$ (they are $\ge$-balls), and since the nearest point projection onto a convex subset of a CAT$(-1)$ space is distance decreasing, the distance in the deleted graph is bounded from below by the inner metric on $\d \widetilde N$.
    The claim now follows from the fact that $\pi_{1}(N)$ acts properly and cocompactly on $\d \widetilde N$.

    Finally, if $\pi_1(N)$ is hyperbolic, and $\pi_1(\M)$ is hyperbolic relative to $\pi_1(N)$, it follows \cite{f98}*{p.
    822} that $\pi_1(\M)$ is hyperbolic.
\end{proof}

\section{Virtual specialness of \texorpdfstring{$\pi_1(\M)$}{the fundamental group of M} }\label{s:special}
A group $G$ is \emph{cocompact special} if it is the fundamental group of a compact special cube complex in Haglund and Wise's sense \cite{hw08}.
All our special cube complexes will be compact, so following \cite{gm18} we will drop the word `cocompact'. We will be mostly concerned with virtually special hyperbolic groups.
By Agol's theorem \cite{a13}, these are precisely the hyperbolic groups which admit cocompact proper actions on CAT(0) cube complexes.
For example, the hyperbolic graph products of finite groups considered in Section \ref{s:ht} are virtually special.

In this section, we will show that the fundamental groups of strict hyperbolizations $hP_L$ are virtually special.
We also show that if a seed manifold $N$ has virtually special hyperbolic fundamental group then so does the manifold $\M$ obtained via the hyperbolic reflection group trick.
One utility of this is that a virtually special hyperbolic group $G$ virtually retracts onto any quasi-convex subgroup $H<G$.
This will let us do cutting arguments and estimate the homology growth of $\M$ from that of its constituent pieces in Section \ref{mvsection}.

\subsection*{Virtual specialness of Charney--Davis pieces}

We start by recalling that the arithmetic manifold $A^n$, and hence $\CD^{n}$, has virtually special fundamental group.
The manifold $A^n$ is constructed via arithmetic methods, and involves some number of choices.
We can take it to be the quotient of the hyperbolic space $\hh^{n}$ by a congruence subgroup of the orthogonal group of the quadratic form
\[
-\frac{1+\sqrt{5}}{2} x_0^2 + x_1^2 + x_2^2 + \dots + x_n^2
\]
over the ring $\zz\left [\frac{1+\sqrt{5}}{2} \right ]$.
This is a uniform arithmetic lattice of simplest type (we give a definition and discuss some properties of such lattices in Section~\ref{s:induction}.)

Moreover, as pointed out in \cite{b84}, for $n \leq 7$ the orthogonal group is virtually a Coxeter group.
The manifolds $\W_i$ that we cut $A^n$ along are images under the covering projection of hyperplanes in $\hh^{n}$ given by setting $x_i = 0$ for $1 \le i \le n$.
Charney and Davis showed that if $\gG$ is any torsion-free congruence subgroup of the orthogonal group, then all the conditions of Theorem \ref{t:arithmetic} hold.
In \cite{hw12}, Haglund and Wise showed virtual specialness for such lattices, see also \cite{bhw11} for an alternative approach.
\begin{theorem}[\cite{hw12},\cite{bhw11}]\label{t:arithmeticspecial}
    Uniform arithmetic lattices in $\hh^n$ of simplest type are virtually special.
\end{theorem}

The group $\pi_1(CD^n)$ acts on a right-angled convex polyhedron in $\hh^n$ with infinitely many sides, precisely the intersection of halfspaces bounded by certain translates of the hyperplanes $x_i = 0$.
This implies that $\pi_1(\CD^n)$ is a quasiconvex subgroup of $\pi_1(A^n)$, see also \cite{lr22}*{Lemma 5.8}.
Since quasiconvex subgroups of hyperbolic virtually special groups are virtually special, we have
\begin{corollary}
    The fundamental group of a Charney--Davis piece is virtually special.
\end{corollary}

\subsection*{Virtual specialness of hyperbolized cones}

Our manifold $\M$ is constructed by reflecting $\Lhc{N}{\d}$; so we first show that $\pLhc{N}{\d}$ is virtually special.
We begin by showing that the fundamental groups of strictly hyperbolized cones are virtually special, and then do the resolved case.

Let $L$ be a flag complex, and let $\hc{L}$ be the corresponding hyperbolized cone.
If $A \subset L$ is a full subcomplex, then by the Link Condition, $\Cone{A}$ is a totally geodesic cubical subcomplex of $\Cone{L}$.
Hence, by Theorem~\ref{t:hK}(4), $\hc{A}$ is a totally geodesic subspace of $\hc{L}$, and even though $K_A$ and $K_L$ are contractible, the hyperbolizations $hK_A$ and $hK_L$ have quite complicated topology.

Note that if $v$ is a vertex of $L$, then the above subcomplex $\hc{\Lk(v)}$ is different than the mirror $h(K_{v})$ used to reflect the hyperbolized Davis chamber around in the previous section.
Indeed, that subcomplex was a hyperbolization of the cubical star of the vertex $[v]$ in $\d K_L$, whereas this is the hyperbolization of the subcomplex of cubes containing the cone point and corresponding to simplices in $\Lk(v)$.
\begin{theorem}\label{t:hcone special}
    Let $L$ be a finite $(k-1)$-dimensional flag complex, and suppose that $\hc{L}$ is obtained by applying the $n$-hyperbolization procedure to $\Cone{L}$ with $n \geq k$.
    Then the connected components of $\hc{L}$ have virtually special fundamental group.
\end{theorem}

To prove this, we use the following result of Groves and Manning, which generalizes Wise's notion of a quasiconvex hierarchy for virtually special groups.
\begin{theorem}[Groves--Manning \cite{gm18}*{Theorems A and D}]\label{t:gm}
    Suppose that $G$ is a hyperbolic group acting cocompactly on a CAT$(0)$ cube complex $X$ so that vertex stabilizers are quasi-convex and virtually special.
    Then $G$ is virtually special.
\end{theorem}
\begin{proof}[Proof of Theorem \ref{t:hcone special}] We induct on the number of vertices of $L$.
    If $L$ is a simplex, then $\Cone{L}$ is a single cube.
    Therefore, the fundamental group of each component of $\hc{L}$ is a Charney--Davis piece (or face of such), and so by assumption is virtually special.
    Otherwise, we can find a vertex $s$ such that $\St(s)$ is not equal to $L$, which gives us a decomposition of $L$ into proper full subcomplexes:
    \[
    L = \St(s) \cup_{\Lk(s)} (L - s).
    \]
    This decomposition induces a decomposition of $\Cone{L}$ into totally geodesic cubical subcomplexes
    \[
    \Cone{L} = \Cone{\St(s)} \cup_{\Cone{\Lk(s)}} \Cone{L-S,}
    \]
    and hence a decomposition of $\hc{L}$ into totally geodesic subcomplexes
    \[
    \hc{L} = \hc{\St(s)} \cup_{\hc{\Lk(s)}} \hc{L - s}.
    \]

    Now, take a connected component of $\hc{L}$.
    This is a union of components of $\hc{\St(s)}$ and $\hc{L-s}$ meeting along components of $\hc{\Lk(s)}$.
    This union gives us a graph of groups decomposition with vertex groups the fundamental groups of components of $\hc{\St(s)}$ and components of $\hc{L-s}$, and edge groups the fundamental groups of components of $\hc{\Lk(s)}$.
    By induction on the number of vertices, these fundamental groups are virtually special, and they are all quasiconvex as they correspond to totally geodesic subcomplexes of $\hc{L}$.
    The action on the associated Bass--Serre tree and Theorem \ref{t:gm} imply that each component of $\hc{L}$ has virtually special fundamental group.
\end{proof}

\subsection*{Unnatural embeddings} We now adapt the argument in the previous subsection to the resolved case.
The idea is exactly the same; we inductively cut along hyperbolized walls $\hc{\Lk(v)}$ for $v$ a vertex of $\d$.

In general, there is no canonical way of embedding $\hc{\Lk(v)}$ into $\Lhc{N}{\d}$.
However, given any vertex $v$ of $\d$, there is an ``unnatural'' embedding $\hc{\Lk(v)} \to \Lhc{N}{\d}$ which sends the cone point of $\hc{\Lk(v)}$ to $v$ (and a neighborhood of the cone point to the star of $v$ in $\d$).
\begin{figure}[!t]\label{f:unnatural}
    \tikzset{ rhouse/.pic={
    \draw (0,0)--++(60:6)--++(120:6)--++(-120:6)--cycle;
    \draw (-1,5) to [out=0, in=-45] (-.9,7) to [out=135, in=180] (0,8) to [out=0, in=45] (.9,7) to [out=-135, in=180] (1,5) ;
    \draw (-.5,7) to [out=-30, in=-150] (.5,7) ;
    \draw (-.4,6.95) to [out=30, in=150] (.4,6.95) ;
    \draw[black, fill = white] (0,0) circle (10mm); } }
    \begin{tikzpicture}[scale=.2] \foreach \i in {0,60,...,360}
        \path (0,0) pic[transform shape, rotate=\i] {rhouse};
        \draw[ultra thick] ([shift=(0:1)]0,0) arc (0:120:1);
        \draw[ultra thick] (6,0)--(1,0);
        \draw[ultra thick] (-3,5.19)--(-.5,.865);
        \draw[fill] ([shift=(60:1)]0,0) circle [radius =0.3] ;
        \node at (1.7, 1.3) {$v$};
    \end{tikzpicture}
    \caption{An ``unnaturally'' embedded $\hc{\Lk(v)}$ inside of $\Lhc{N}{\d}$.}
\end{figure}
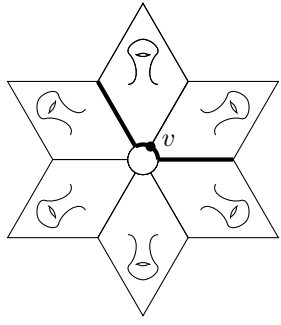

To describe the result of the cutting we need to consider a more general situation.
Each full subcomplex $A$ of $\d$ determines a subspace $\Lhc{N}{A}$ of $\Lhc{N}{\d}$, obtained by replacing an $\ge$-neighborhood of the cone point in $\hc{A}$ with a copy of $N$ glued along $A$.

As before, given any vertex $v$ of $A$, there is again an unnatural embedding of $\hc{\Lk_A(v)}$ into $\Lhc{N}{A}$ (and hence into $\Lhc{N}{\d}$) sending the cone point of $\hc{\Lk_A(v)}$ to $v$.
\begin{lemma}\label{l:qclinks}
    For any full subcomplex $A$ of $\d $ and any vertex $v \in A$, the components of the unnaturally embedded walls $\hc{\Lk_A(v)}$ are $\pi_1$-injective in $\Lhc{N}{\d}$, and the corresponding subgroups are quasiconvex.
\end{lemma}
\begin{proof}
    A priori $\hc{\Lk_{A}(v)}$ may be disconnected, and hence the image of the unnatural embedding into $\Lhc{N}{\d}$ may be disconnected as well.
    However there is a distinguished component of the image containing $v$.
    This component is unnaturally embedded, and the embeddings of the other components are unchanged, i.e. including these into $\Lhc{N}{\d}$ and composing with the collapse map $q: \Lhc{N}{\d} \to \hc{\d}$ agrees with the natural embedding into $\hc{\d}$.
    The intersection of the distinguished component with $N$ is the subcomplex $\St(v)$, which is contractible, and therefore the composition with the collapse map is homotopic to the natural embedding into $\hc{\d}$.
    Therefore, the restriction of $q_\ast: \pi_1(\Lhc{N}{\d}) \to \pi_1(\hc{\d})$ to $\phc{\Lk_A(v)}$ is an isomorphism for each component of $\hc{\Lk_A(v)}$.

    Since the groups $\pi_1(\Lhc{N}{\d})$ and $\pi_1(\hc{\d})$ are hyperbolic and $\hc{\Lk_A(v)}$ is totally geodesic in $\hc{d}$, the quasiconvexity claim follows from the following general lemma:
    \begin{lemma}\label{l:quasiconvex}
        Let $G$ and $H$ be hyperbolic groups, $\phi: G \to H$ a homomorphism, and $K$ a subgroup of $G$.
        If $\phi(K)$ is quasiconvex in $H$ and $\phi|_K$ is an isomorphism, then $K$ is quasiconvex in $G$.
    \end{lemma}
    \begin{proof}
        Note that $K \cong \phi(K)$ which is quasiconvex in $H$ and hence finitely generated.
        Given any subgroup $K$ of $G$, the inclusion map is Lipschitz with respect to the word metrics, so we only have to check lower bounds.
        Let $k$ and $k'$ be elements in $K$.
        Then $d_G(k, k')$ is linearly bounded from below by $d_H(\phi(k), \phi(k'))$.
        Since $\phi(K)$ is quasiconvex, $d_H(\phi(k), \phi(k'))$ is bounded from below by a linear function in $d_{\phi(K)}(\phi(k), \phi(k'))$.
        Since $\phi|_{K}$ is an isomorphism it induces a quasi-isometry between $K$ and $\phi(K)$.
        These combine to give a lower bound for $d_G(k, k')$ by a linear function in $d_{\phi(K)}(\phi(k), \phi(k'))$.
    \end{proof}

    \renewcommand{\qedsymbol}{}
\end{proof}

\subsection*{Virtual specialness of resolved hyperbolic cones} To extend Theorem \ref{t:hcone special} to the resolved case, we need the following well-known lemma:
\begin{lemma}[\cite{b98a}*{Proposition 1.2}]\label{l:splitting}
    Let $G$ be a hyperbolic group which splits as a finite graph of groups.
    If each edge group is quasiconvex, then each vertex group is quasiconvex.
\end{lemma}
\begin{theorem}
    If $\pi_{1}(N)$ is a virtually special hyperbolic group, then $\pLhc{N}{\d}$ is virtually special hyperbolic.
\end{theorem}
\begin{proof}
    We inductively cut along these unnaturally embedded walls in $\Lhc{N}{\d}$.
    At each inductive step, we are cutting $\Lhc{N}{A}$ along an unnaturally embedded $\hc{\Lk_A(v)}$ for some full subcomplex $A$.
    Let $X_A$ denote the component of $\Lhc{N}{A}$ containing $N$.
    The components of $\Lhc{N}{A}$ other than $X_A$ are the same as components of $\hc{A}$, and hence have virtually special fundamental group by Theorem \ref{t:hcone special}.
    Therefore, we focus on the cutting's effect on $X_A$.

    Removing $\hc{\Lk_A(v)}$ decomposes $X_A$ into a graph of spaces where the edge spaces are components of $\hc{\Lk_A(v)}$, and the vertex spaces are either components of $\hc{\St_A(v)}$ or $X_{A-v}$.
    The components of $\hc{\Lk_A(v)}$ and $\hc{\St_A(v)}$ again have virtually special $\pi_1$ by Theorem \ref{t:hcone special}, and $\pi_1(X_{A-v})$ is virtually special by induction on the number of vertices; the base case here is when $A$ is a disjoint union of simplices, so $\pi_1(X_A)$ is a free product of $\pi_1(N)$ and fundamental groups of Charney--Davis pieces, and hence is hyperbolic and virtually special.
    By Lemma \ref{l:qclinks} the edge groups are quasiconvex, and hence by Lemma~\ref{l:splitting} so are the vertex groups.
    Again, by Theorem \ref{t:gm} we are done.
\end{proof}

\subsection*{Hyperbolic reflection group trick preserves virtual specialness (Proof of Theorem \ref{t:hypdavistrick}(5))}

We now extend this to our hyperbolic reflection group trick $\M$.
\begin{prevthm}
    {\ref{t:hypdavistrick} (5)}\label{t:manifoldspecial}
    If $\pi_1(N)$ is virtually special hyperbolic, then $\pi_1(\M)$ is virtually special hyperbolic.
\end{prevthm}
\begin{proof}
    We have that $\M$ is obtained by reflecting around $\Lhc{N}{\d}$ using $(\zz/2)^{\abs{\d N^0}}$.
    We saw in Section \ref{s:hrt} that $\widetilde \M$ is a basic construction obtained by reflecting around the universal cover of $\Lhc{N}{\d}$, where the RACG $\widetilde W$ has generators corresponding to components of preimages of mirrors in $\Lhc{N}{\d}$.
    There is a natural action of $\pLhc{N}{\d}$ on $\widetilde W$ induced by the action on the universal cover of $\Lhc{N}{\d}$, and $\pi_1(\M)$ is a finite index subgroup of $\widetilde W \rtimes \pLhc{N}{\d}$.

    Now, let $\widetilde \gS$ denote the corresponding Davis complex for $\widetilde W$, equipped with its Coxeter cellulation (in particular, the $1$-skeleton of $\widetilde \gS$ is identified with the Cayley graph of $\widetilde W$).
    This gives $\widetilde \gS$ the structure of a finite dimensional, locally infinite, CAT(0) cube complex.

    The group $\pLhc{N}{\d}$ acts on $\widetilde \gS$ via its action permuting the generators of $\widetilde W$, and this extends to an action of $\widetilde W \rtimes \pLhc{N}{\d}$ on $\widetilde \gS$.
    There is one orbit of vertices, and a finite number of orbits of edges since there are a finite number of $\pLhc{N}{\d}$-orbits of mirrors upstairs.
    Since $\widetilde \gS$ is finite dimensional this implies the action is cocompact.
    The stabilizer of a vertex is conjugate to $\pLhc{N}{\d}$.
    We have shown that this is virtually special, and since $\widetilde W \rtimes \pLhc{N}{\d}$ retracts onto $\pLhc{N}{\d}$, it is quasiconvex.
    Therefore, we are done by Theorem \ref{t:gm}.
\end{proof}

\subsection*{Virtual retractions}

The same argument applied to $\hc{L}$ in place of $\Lhc{N}{\d}$ shows that the fundamental group of each $hP_L$ is hyperbolic and virtually special.
We record a corollary of the fact that the $hP_L$ have virtually special fundamental groups and work of Haglund--Wise \cite{hw08}*{Sections 6 and 7} that we will use in the next section in order to prove Mayer--Vietoris inequalities for homology growth of manifolds obtained by the hyperbolic reflection group trick.

We need a few simple observations, which are true for general complexes.
Given a map $f: Z \to X$ we say that $Z$ is a \emph{virtual retract} of $X$ if there exist a finite cover $p:X^{r} \to X$, an injective lift $g:Z \into X^{r}$ of $f$, and a retraction $r: X^{r} \to Z$ such that $rg=id_{Z}$.
\begin{lemma}\label{l:simple1}
    Suppose $f: Z \to X$ and $Z$ consists of finitely many components $Z_i$.
    If each $Z_i$ is a virtual retract of $X$, then $Z$ is a virtual retract of $X$.
\end{lemma}
\begin{proof}
    By assumption, there are finite covers $X_i\to X$, lifts $g_{i}: Z_i \into X_i$ and retractions $r_i:X_i\to Z_i$.
    These assemble to a lift $Z=\coprod Z_i\into \coprod X_i=:X^{r}$ and a retraction $X^{r}\to Z$.
\end{proof}
\begin{lemma}\label{l:vr}
    Suppose $f: Z \rightarrow X$ is a virtual retract.
    Then, in the notation of the above definition:
    \begin{enumerate}
        \item If $q:X' \to X$ is a finite cover of $X$, then the pullback $q^{*}(Z) = f^{*}(X')$ is a virtual retract of $X'$.
        \item If $t: Z' \rightarrow Z$ is a finite cover of $Z$, then $Z'$ is a virtual retract of $X^{r}$.
        \item $f^{*}(C_X)$ is a cofinal subset of $C_{Z}$.
        \item If $f: Z\into X$ is a retract of $X$, then $f^{*}(C_X)=C_{Z}$.
    \end{enumerate}
\end{lemma}
\begin{proof}

    Under the assumptions of (1), the pullback $q^{*}(X^{r})=p^{*}(X')$ is a cover of $X'$ which retracts onto $q^{*}(Z)$ via $q^{*}(r)$.
    This proves (1).

    Under the assumptions of (2), the pullback $t^{*}(X^{r})=r^*(Z')$ is a finite cover of $X^{r}$ which retracts onto $Z'$ via $t^{*}(r)$.
    This proves (2).

    Starting with a finite cover $Z'$ of $Z$, we can think of $r^*(Z')$ as a finite cover of $X$.
    Let $X'$ be a further regular cover of $X$.
    Since $X'$ is regular and covers $X^{r}$, $p^*(X')$ is a disjoint union of copies of $X'$ covering $X^r$.
    Since $g^{*}(r^{*}(Z'))=Z'$, regarding $X'$ as a cover of $X^{r}$, $g^*(X')$ is a cover of $Z$ which covers $Z'$.
    Therefore, $f^*(X') = g^*p^*(X')=\coprod g^{*}(X')$ covers $Z'$ as well.
    This proves (3).

    If $Z$ is a retract of $X$, then $Z' = f^*r^*(Z')$ for any finite cover $Z'$ of $Z$, hence $(4)$.
\end{proof}

The following lemma is a direct consequence of the homotopy extension property.
\begin{lemma}\label{l:simple2}
    A subcomplex $i:Z\hookrightarrow X$ is a retract if and only if the inclusion $i$ has a left homotopy inverse $r:X\to Z$.
\end{lemma}

Our source of virtual retracts is provided by work of Haglund--Wise.
\begin{theorem}[\cite{hw08}*{Theorem 7.3}]\label{t:hwvr}
    Let $G$ be a hyperbolic virtually special group and $H < G$ a quasiconvex subgroup.
    Then there is a finite index subgroup of $G$ which retracts onto $H$.
\end{theorem}
\begin{corollary}
    If $A \subset L$ is a full subcomplex, then $hP_A$ is a virtual retract of $hP_L$.
\end{corollary}
\begin{proof}
    By Lemma \ref{l:simple1}, it is enough to show that each component of $hP_A$ is a virtual retract, and since the components of all the spaces involved are aspherical, Lemma \ref{l:simple2} implies that it is enough to construct a virtual retract on the level of fundamental groups.
    Since each component $hP_A$ is a totally geodesic subspace of $hP_L$, its fundamental group $H$ is a quasiconvex subgroup of the fundamental group $G$ of $hP_L$.
    So, we are done by Theorem~\ref{t:hwvr}.
\end{proof}
\begin{corollary}\label{c:vretract}
    For a full subcomplex $A$ of a link in $\d$, let $hP_A\into hP^N_{\d}$ be the unnatural embedding.
    Then $hP_A$ is a virtual retract of $hP^N_{\d}$.
\end{corollary}
\begin{proof}
    Pull back the virtual retraction produced in the previous corollary by the collapse map $hP^N_{\d}\to hP_{\d}$.
\end{proof}
\begin{remark}

    A recent result \cite{lr22} of Lafont and Ruffoni also shows that the manifolds $hP_{S^r}$ have virtually special fundamental groups.
    In fact, since any flag complex $L$ embeds as a full subcomplex of some flag $S^N$, the components of $hP_L$ are totally geodesic inside of $hP_{S^N}$, and hence their results show these have virtually special fundamental groups.
    More generally, they show that if $X$ is an $n$-dimensional locally CAT(0) cube complex with each cube contained in an $n$-cube, and every $(n-1)$-cube contained in at least two $n$-cubes, then $hX$ is virtually special.
    To do this, they show that the fundamental group of the hyperbolization $hX$ acts on a certain cube complex with stabilizers that are quasi-convex subgroups in the arithmetic hyperbolic $n$-manifolds $A^n$, prove that cube complex is CAT(0), and apply Theorem \ref{t:gm}.
    Our argument does not construct such a complex.
    Instead, we observe that (since the initial complex has the form $P_L$) the fundamental group of $hP_L$ acts on a Davis complex (which is definitely a CAT(0) cube complex) with vertex stabilizers that are retracts (hence quasi-convex) and isomorphic to fundamental groups of hyperbolized cones $\hc{L}$ (which may not embed in arithmetic hyperbolic $n$-manifolds), show these stabilizers are virtually special by the cutting argument in Theorem \ref{t:hcone special}, and then apply Theorem \ref{t:gm}.
\end{remark}

\section{Mayer--Vietoris arguments\label{mvsection}
} Our goal in this section is to control the homology growth of a space $X$ in terms of the homology growth of simpler pieces the space can be cut into.
This will be used to compute the homology growth of the manifold $\M$ that we constructed in Section \ref{s:hrt} in terms of the seed manifold $N$.
\subsection*{Mayer--Vietoris inequalities for homology growth}
\begin{lemma}[Restricted Mayer--Vietoris inequalities]\label{l:newmv}
    Suppose $X=A_1\cup_B A_2$.
    \begin{enumerate}
        \item If $\ub_{k-1}^{X}(B)=0$ then
        \[
        \ub_{k}(X) \leq \ub_{k}^{X}(A_1)+\ub_{k}^{X}(A_2).
        \]
        \item If $\ub_{k}^{X}(B)=0$ then
        \begin{align*}
            \lb_{k}^{X}(A_{i}) &\leq \lb_{k}(X) ,\\
            \ub_{k}^{X}(A_{i}) &\leq \ub_{k}(X).
        \end{align*}
    \end{enumerate}
\end{lemma}
\begin{proof}

    For any cover $X'\to X$ denote the induced covers by $'$.
    The Mayer--Vietoris sequence
    \[
    \dots \to H_k(B')\to H_k(A_1')\oplus H_k(A_2')\to H_k(X')\to H_{k-1}(B') \to \dots
    \]
    implies that on the poset $C_X$ the normalized Betti functions satisfy
    \[
    \nb_{k}(X) \leq \nb_{k}(A_{1}) + \nb_{k}(A_{2}) + \nb_{k-1}(B).
    \]
    Taking the upper limit of this inequality over $C_X$ and noting that $\ub_{k-1}^{X}(B )=0$ proves part~(1) of the lemma.
    The same Mayer--Vietoris sequence implies that on $C_X$
    \[
    \nb_{k}(A_{1}) + \nb_{k}(A_{2}) \leq \nb_{k}(B) + \nb_{k}(X).
    \]
    Taking the upper or lower limit of this inequality over $C_X$ and using $\ub_{k}^{X}(B)=0$ and subadditivity Lemma \ref{l:posetineq}\eqref{i:sub} proves~(2).
\end{proof}
Applying the second part of the above Lemma repeatedly gives the following.
\begin{corollary}\label{l:homologygrowthmv}
    Suppose $X$ has a filtration $A_{0}= X_0\subset X_1\subset\dots\subset X_m=X$ with $X_i=X_{i-1}\cup_{B_i} A_i$.
    If $\ub^{X}_k(B_i)=0$ for all $i$, then
    \[
    \lb_k^{X}(A_0) \leq\lb_k(X),
    \]
    \[
    \ub_k^{X}(A_0) \leq\ub_k(X).
    \]
\end{corollary}

These inequalities are only useful when we can get rid of restriction on the covers, and this is one thing the virtual retractions from Section \ref{s:special} are good for.
Parts (3) and (4) of Lemma \ref{l:vr} together with Lemma \ref{l:posetineq}\eqref{i:cofinal} imply:
\begin{corollary}[Virtual retractions]\label{c:virtual}
    Suppose $A$ is a virtual retract of $X$.
    Then
    \[
    \lb(A) \leq \lb^{X}(A) \leq \ub^{X}(A) \leq \ub(A).
    \]
    If $A$ is a retract of $X$, then the left and right inequalities are equalities.
\end{corollary}
Applying Corollary~\ref{c:virtual} to Lemma~\ref{l:newmv} and Corollary~\ref{l:homologygrowthmv} results in the following two corollaries about homology growth that we will use later in the paper.
\begin{corollary}[Absolute Mayer--Vietoris inequalities]\label{c:mvinequalities}
    Suppose $X=A_1\cup_B A_2$.
    Suppose further that $A_1,A_2$ and $B$ are virtual retracts of $X$.
    Let $d$ be the degree of a cover of $X$ which retracts to $A_{1}$.
    \begin{enumerate}
        \item If $\ub_{k-1}(B)=0$, then
        \[
        \ub_k(X)\leq \ub_k(A_1)+\ub_k(A_2).
        \]
        \item If $\ub_k(B)=0$, then

        \[
        \lb_k(A_i) \leq \lb_k(X),
        \]
        and
        \[
        \ub_k(A_1) \leq d\cdot\ub_k(X).
        \]
    \end{enumerate}
\end{corollary}
\begin{proof}
    (1) and the first inequality in (2) follows immediately from the restricted version and Corollary~\ref{c:virtual} applied to $A_{i}$ and $B$.

    For the second part of (2), pass to a finite cover $X^{r}\to X$ which retracts onto $A_{1}$.
    Then the induced cover $A'_{1} \to A_{1}$ has a section.
    In other words, $A'_{1}$ contains a copy of $A_{1}$, so
    \[
    \ub_{k}^{X^{r}}(A_{1}) \leq \ub_{k}^{X^{r}}(A'_{1}).
    \]
    Since $A_{1}$ is a retract of $X^{r}$, $ \ub_{k}^{X^{r}}(A_{1}) = \ub_{k}(A_{1})$.
    Now apply restricted version to decomposition $X^{r}=A'_1\cup_{B'} A'_2$ and use multiplicativity to get
    \[
    \ub_{k}(A_{1}) \leq \ub_{k}^{X^{r}}(A'_{1}) \leq \ub_{k}(X^{r}) = d\cdot\ub_{k}(X) .
    \]
\end{proof}
\begin{corollary}[Cutting down to a virtual retract]\label{homologygrowthmv}
    Suppose that $X$ has a filtration $A_{0}= X_0\subset X_1\subset\dots\subset X_m=X$ with $X_i=X_{i-1}\cup_{B_i} A_i $, and that $A_0$ and each $B_i$ are virtual retracts of $X$.
    Let $d$ be the degree of a cover of $X$ which retracts to $A_{0}$.
    If $\ub_k(B_i)=0$ for all $i$, then
    \[
    \lb_k(A_0)\leq\lb_k(X),
    \]
    \[
    \ub_k(A_0)\leq d\cdot \ub_k(X).
    \]

    In particular, if $\ub_k(A_0) \ne 0$ then $\ub_k(X) \ne 0$.
\end{corollary}

\subsection*{Mayer--Vietoris inequalities via skew fields}

One of the nice features of the skew field $D_{\F G}$ described in Section \ref{s:skewfields} is that---in situations when it is available---it provides a simple, alternate framework for doing Mayer--Vietoris cutting arguments.
We record this here.
(It is not needed in the proofs below, but can be used as an alternative to the inequalities above in special cases.)

Suppose $X$ has residually torsion-free nilpotent fundamental group $G$ and let $D=D_{\F G}$ be the skew field from Section \ref{s:skewfields}.
If $X=A_1\cup_B A_2$ then (only in this subsection) use $\hat{}$ to denote the induced $G$-covers.
Then we have a Mayer--Vietoris sequence

\[
\dots \to H^G_k(\hat{B};D)\to H^G_k(\hat{A}_1;D)\oplus H^G_k(\hat{A}_2;D)\to H^G_k(\hat{X};D)\to H^G_{k-1}(\hat{B};D) \to \dots
\]
for homology with coefficients in $D$.
If $B\into X$ is $\pi_1$-injective, then the $A_i$ are as well, and then we can identify the $D$-dimensions of the $D$-vector spaces appearing in this sequence with lower homology growth by Corollary \ref{c:sccomp}:
\begin{align*}
    b^G_k(\hat{B};D)&=\lb_k(B;\F),\\
    b^G_k(\hat{A}_i;D)&=\lb_k(A_i;\F).
\end{align*}
Therefore, the Mayer--Vietoris sequence gives the usual inequalities ((1) and (2) below) for lower homology growth.
Using multiplicativity of $\lb$ extends these inequalities to complexes $X$ with virtually residually torsion-free nilpotent fundamental groups.
In summary, we get
\begin{lemma}
    Suppose a finite complex $X$ with virtually residually torsion-free nilpotent fundamental group decomposes as a union $X=A_1\cup_B A_2$ with $\pi_1$-injective intersection.
    \begin{enumerate}
        \item If $\lb_{k-1}(B;\F )=0$ then $\lb_k(X;\F )\leq \lb_k(A_1;\F )+\lb_k(A_2;\F )$,
        \item If \hphantom{${}_{-1}$}$ \lb_{k}(B; \F) =0$ then $\lb_{k}(X; \F) \geq \lb_{k}(A_1; \F)+\lb_k(A_2;\F )$.
    \end{enumerate}
\end{lemma}

Iterating the second part of the lemma gives
\begin{corollary}\label{c:rtfnthmv}
    Suppose a finite complex $X$ with virtually residually torsion-free nilpotent fundamental group has a filtration $A_{0}= X_0\subset X_1\subset\dots\subset X_m=X$ with $X_i=X_{i-1}\cup_{B_i} A_i $, and that each $B_i\into X$ is $\pi_1$-injective.
    If $\lb_k(B_i;\F )=0$ for all $i$, then
    \[
    \lb_k(A_0;\F )\leq\lb_k(X;\F ).
    \]
\end{corollary}

\subsection*{A spectral sequence in even dimensions}

If we have a large collection of subcomplexes, it is more convenient to use a spectral sequence rather than inductively cutting and using Mayer--Vietoris.
We record the spectral sequence we will use below, in our setting we only need to consider an ambient manifold and a collection of codimension-one submanifolds.

Suppose $M$ is a closed manifold, and $\{X_s\}$ is a finite collection of closed codimension-one submanifolds, such that the pair $(M, \bigcup X_{s})$ looks locally like a hyperplane arrangement in $\rr^n$.

Removing a regular neighborhood of $\bigcup X_{s}$ from $M$ produces a manifold with boundary $(Z, \d Z )$, the cut-open $M$.
We call the collection $\{X_s\}$ \emph{tractable} if, in addition, the inclusions of $Z$ and of each intersection $X_{J}$ of the walls into $M$ are virtual retractions.
Note that by Lemma~\ref{l:simple1} this is equivalent to the disjoint union of $Z$ and all $X_{J}$ being a virtual retract.
\begin{lemma}[Even chopping]\label{l:ss}
    Let $(M^{2n}, \{X_{s}\})$ be a tractable collection of codimension one submanifolds and $(Z, \d Z )$ the resulting cut-open manifold.
    Suppose all intersections of walls have the upper $\F$-Singer property, and that $Z$ is a retract of a $d$-fold cover of $M$.
    Then for $k>n$,
    \[
    \ub_{k}(M; \F ) \leq \ub_{k}(Z; \F ) \leq d\cdot \ub_{k}(M; \F ),
    \]
    and
    \[
    \lb_{k}(Z; \F )\leq\lb_{k}(M; \F ).
    \]
    In particular, $M^{2n}$ has the upper $\F$-Singer property if and only if $Z^{2n}$ has.
\end{lemma}
\begin{proof}
    Let $X:=\bigcup X_{s}$ and $X_{J}:=\bigcap_{s \in J} X_{s}$.
    First we claim that $\ub_{k}^{M} ( X) =0 $ in degrees $k<n$.

    To see this, consider the spectral sequence for the homology of the union $X=\bigcup X_{s}$.
    Then the $k$-th Betti number of $X$ is bounded from above by the sum of dimensions of the terms $E_{i, j}^{1} = \bigoplus_{\abs{J}=j+1} H_{i}(X_{J}; \F) $ with $i+j=k$.
    Taking covers, normalizing and taking the upper limit over $C_{M}$ gives
    \[
    \ub_{k}^{M} ( X) \leq \sum_{i+j=k} \sum_{\abs{J}=j+1} \ub_{i}^{M}(X_{J}).
    \]
    Since all $X_{J}$ are virtual retracts of $M$, we have
    \[
    \ub_{k}^{M} ( X) \leq \sum_{i+j=k} \sum_{\abs{J}=j+1} \ub_{i}(X_{J}).
    \]
    Since all $X_{s}$ are odd-dimensional and have the upper $\F$-Singer property, all the terms in this sum with $j=0$ are $0$.
    Similarly, since the components of a $(j+1)$-fold intersection $X_{J}$ are manifolds of dimension $\geq 2n-(j+1)$ and have the upper $\F$-Singer property, for $j\geq 1$, $ \ub_{i}(X_{J})=0$ for $i< (2n - (j+1))/2$.
    This inequality is equivalent to $i+j < n + (j -1)/2$, and the claim follows.

    Next, let $N(X)$ denote a closed regular neighborhood of $X$ and note that the pair $(M, N(X))$ excises to $(Z, \d Z)$, so by Lemma \ref{l:posetineq} and the long exact sequence
    \[
    \to H_{k}(X ) \to H_{k}(M) \to H_{k}(Z, \d Z) \to H_{k-1}(X ) \to
    \]
    we have
    \begin{align*}
        \ub_k(M) &\leq \ub_{k}^{M}( X ) + \ub_k^{M}(Z,\d Z),\ \intertext{and} \ub_k^{M}(Z,\d Z) &\leq \ub_k(M) + \ub_{k-1}^{M} ( X),\\
        \lb_k^{M}(Z,\d Z) &\leq \lb_k(M) + \ub_{k-1}^{M} ( X).
    \end{align*}
    By the above claim the terms with $X$ vanish for $k<n$.
    Thus, for $k<n$,
    \[
    \ub_k^{M}(Z,\d Z) = \ub_k(M), \quad \lb_k^{M}(Z,\d Z) \leq \lb_k(M).
    \]
    Therefore by Poincar\'e duality, for $k>n$,
    \[
    \ub_k^{M}(Z) = \ub_k(M), \quad \lb_k^{M}(Z) \leq \lb_k(M).
    \]
    Since $Z$ is a virtual retract of $M$, Corollary~\ref{c:virtual} implies $\ub_k(M) \leq \ub_k(Z)$ and $\lb_k(Z) \leq \lb_k(M)$ for $k>n$.

    For the remaining inequality, we first pass to a degree $d$ cover $M^{r}$ which retracts onto $Z$.
    By Lemma~\ref{l:vr}(1) the preimage of $\{X_{s}\}$ is a tractable collection in $M^{r}$.
    It cuts down $M^{r}$ to the preimage $Z'$ of $Z$, which consists of $Z$ and other components.
    A similar argument as above shows that for $k > n$,
    \[
    \ub_k(Z) = \ub_{k}^{M^{r}} (Z) \leq \ub_k^{M^{r}}(Z') = \ub_k(M^{r}) = d\cdot\ub_k(M).
    \]
\end{proof}

\subsection*{A $\ub$-acyclic covering lemma}

There is a class of spaces, generalizing Salvetti complexes of right-angled Artin groups, for which a Mayer--Vietoris argument gives a complete computation of both upper and lower homology growth, and shows that these computations agree.
We record this here.
The proof is essentially the one given in \cite{aos21}.
\begin{lemma}\label{ubacyclic}
    Suppose $X$ is a complex covered by finitely many subcomplexes $\{U_i\}$ and let $\cn$ be the nerve of this cover.
    Suppose each non-empty intersection $U_{\gs}:=\bigcap_{i\in\gs}U_i$ is either $\ub^X$-acyclic or a point.
    Let $\cl$ be the subcomplex of the nerve consisting of simplices $\gs$ with $\ub^X$-acyclic $U_{\gs}$.
    Then
    \[
    \lb_k(X;\F)=b_{k}(\cn,\cl;\F)=\ub_k(X;\F).
    \]
\end{lemma}
\begin{proof}
    It is enough to consider connected covers.
    Let $X'\to X$ be a finite connected cover of $X$.
    The Mayer--Vietoris spectral sequence for the finite covering $\{U_i'\}$ of $X'$ by preimages of the $U_i$ converges to $H_*(X';\F)$ and the assumptions imply that its $E^1$ page
    \[
    E^1_{j,k}(X')=C_j(\cn;H_k(U'_{\gs};\F))
    \]
    is concentrated on the $E^1_{j,0}$ line up to an error that is sublinear in the degree $|X'\to X|$.
    Set
    \[
    V'_{\gs}:=
    \begin{cases}
        H_0(U'_{\gs};\F) &\text{ if } U_\gs \text{ is a point},\\
        0 &\text{ if } U_{\gs} \text{ is }\ub^X\text{-acyclic}.
    \end{cases}
    \]
    The chain map $C_j(\cn;H_0(U'_{\gs};\F))\to C_j(\cn;V'_{\gs})$ which collapses the $\ub^X$-acyclic part of the coefficients is onto and has kernel of dimension that is sublinear in $|X'\to X|$.
    Putting these two observations together shows that $\ub(X;\F)$ and $\lb(X;\F)$ are the upper and lower limits of the normalized Betti numbers $b_*(\cn;V'_{\gs})/|X'\to X|$.
    Finally, the chain complex $C_*(\cn;V'_{\gs})$ can be identified with $C_*(\cn,\cl;\F)\otimes_{\F}\F[\pi_1(X)/\pi_1(X')]$, so these normalized Betti numbers are all equal to $b_*(\cn,\cl;\F)$ independent of the choice of $X'\to X$.
    This implies the equations in the statement of the lemma.
\end{proof}

If the $U_{\gs}$ are virtual retracts of $X$ then Corollary \ref{c:virtual} implies that we can replace $\ub^X$-acyclic by $\ub$-acyclic.
The basic example to which this Lemma applies is the cover of the Salvetti complex of a right angled Artin group $A_L$ by maximal tori.
In that case all intersections are retracts, either non-trivial tori (hence $\ub$-acyclic) or points, the nerve is contractible, and the subcomplex $\cl$ is homotopy equivalent to $L$.
This gives the computation from \cite{aos21} mentioned in the introduction:
\[
\lb_k(A_L;\F)=\tilde b_{k-1}(L;\F)=\ub_k(A_L;\F).
\]
If one only has $\ub^X$-acyclicity up to an error $\gd$, then the equalities in Lemma \ref{ubacyclic} hold up to an error on the order of $|\cn|\gd$.
Applying this to the classifying space of a graph product $G_L$ of finite groups $\zz/m$ (and noting that each of the vertex groups $G_v\cong\zz/m$ is a retract of $G_L$ and $\ub$-acyclic up to error $1/m$) gives an alternative way of obtaining estimates for homology growth of these graph products similar to the ones we got in Section \ref{s:graphproducts}.

\section{Inductive structure of \texorpdfstring{$\F$}{F}-Singer}\label{s:induction}

Suppose that $(N,\d N)$ is a compact $n$-manifold with boundary, $\d$ is a flag triangulation of the boundary and $\M=hP^N_{\d}$ is the closed $n$-manifold constructed from this data via the hyperbolic reflection group trick.
\begin{proposition}\label{p:singerdavistrick}
    Either $\lb_k(N;\F )\leq\lb_k(\M;\F )$ or there is a full subcomplex $L$ of a flag triangulated $S^{n-2}$ such that $\ub_{k}(hP_L;\F )\neq 0$.
\end{proposition}
\begin{proof}
    Removing vertices of $\d$ one at time (and using the unnatural embedding of $hP_{\Lk(v)}$ to cut) shows, by Corollary \ref{homologygrowthmv}, that $\lb_{k}(\M;\F )\geq \lb_{k}(N;\F )>0$ as long as $\ub_k(hP_{L};\F )=0$ for all full subcomplexes $L$ of links of vertices $\Lk(v)=S^{n-2}$ in $\d$.
\end{proof}

So, to prove Theorem \ref{maintheorem}(1), we need to understand the upper homology growth of the hyperbolizations $hP_L$.
That is one of the main subjects of this section.
But, before we get there, we first analyze some special hyperbolic manifolds.

\subsection*{On $\F$-Singer for hyperbolic manifolds of simplest type}

Let us start by showing for the class of closed arithmetic hyperbolic manifolds of simplest type that if there is a counterexample to upper $\F$-Singer property, the smallest one occurs in odd dimension.

We recall the definition.
Let $k < \rr$ be a totally real algebraic number field, let $\co$ denote its ring of integers.
Suppose $(V,Q)$ a quadratic vector space over $k$ of dimension $(n+1)$, such that $(V \otimes_{k^{\gs}} \rr, Q^{\gs})$ is positive definite for each nontrivial place $\gs$, and have signature $(n,1)$ for the trivial one.
Let $L$ be a $\zz$-lattice in $V$, i.e. $L$ is a free abelian subgroup $V$ such that $L\otimes \qq = V$.
Following \cite{v93}*{p.
217} we will call a group $\gG$ an \emph{arithmetic group of simplest type}\footnote{Also called of simple type, Type I, or standard.}
if $\gG$ is commensurable to the stabilizer $\Stab_{O(Q)}(L)$.
In fact, this definition is independent of the choice of the lattice $L$ as the following lemma shows.
\begin{lemma}
    If $L$ and $L'$ are $\zz$-lattices in $V$, then $\Stab_{O(Q)}(L)$ is commensurable to $\Stab_{O(Q)}(L')$.
\end{lemma}
\begin{proof}
    By taking intersections, it is enough to consider the case when $L'<L$.
    Then a finite index subgroup of $\Stab_{O(Q)}(L)$ stabilizes $L'$.
    Since there exists $m \in \nn$ such that $mL < L'$, and $\Stab_{O(Q)}(L)=\Stab_{O(Q)}(mL)$, a finite index subgroup of $\Stab_{O(Q)}(L')$ stabilizes $L$, and the claim follows.
\end{proof}

We also need the following.
\begin{lemma}
    Let $\gG< O(Q)$ be an arithmetic group of simplest type.
    If $U$ is a $k$-subspace of $V$ such that $Q|_{U^{\perp}}$ is positive definite, then $\Stab_{\gG}(U)$ is an arithmetic group of simplest type.
\end{lemma}
\begin{proof}
    Let $L'<U$ and $L''< U^{\perp}$ be $\zz$-lattices.
    Then $L=L'\oplus L''$ is a $\zz$-lattice in $V$.
    Therefore $\gG':=\Stab_{O(Q)}(L)$ is commensurable with $\gG$.
    We have $\Stab_{\gG'}(U)=\Stab_{O(Q|_{U} ) }(L') \times \Stab_{O( Q|_{U^{\perp}} ) }(L'') $.
    Since $Q|_{U^{\perp}}$ is positive definite, the second factor is finite, and we are done.
\end{proof}

In particular, if we identify $V$ with $k^{n+1}$, scale the form $Q$ so that it has coefficients in $\co$ and take $L=\co^{n+1}$, the stabilizer $\Stab_{O(Q)}(\co^{n+1})$ gets identified with the orthogonal group $O(Q,\co)$ of matrices preserving the form with entries in $\co$.
Thus, up to commensurability, we recover the definition in \cite{hw12}.

Note that, by taking the trivial place, $O(Q)< O(n,1)$, and $\Stab_{O(Q)}(L)$ is a lattice in $O(n,1)$.
A standard application of Mahler's compactness criterion, cf. \cite{m15}*{Proposition 5.3.4}, shows that this lattice is uniform if and only if $0$ is the only solution in $\co^{n+1}$ to the equation $Q(x)=0$.

We will say that a hyperbolic manifold $H =\hh^n/\gG$ is \emph{of simplest type} if $\gG$ is an arithmetic group of simplest type.
\begin{lemma}\label{l:hw}
    Let $H$ be a compact hyperbolic manifold of simplest type.
    Then there is a finite cover of $H$ with a tractable collection of codimension-one submanifolds of simplest type whose complement is a disjoint union of open discs, and each intersection in this collection is also a manifold of simplest type.
\end{lemma}
\begin{proof}
    Suppose $\gG=\pi_{1}(H)$ is associated to a quadratic space $(V,Q)$ over a field $k$.
    Choose $v \in V$ with $Q(v)< 0 $, and identify $\hh^{n}$ with the sheet of the hyperboloid $Q(x)=Q(v)$ in $V \otimes_{k} \rr$ containing $v$.
    Consider the Dirichlet domain corresponding to the orbit $\gG v$ which contains $v$.
    The bounding hyperplanes of this domain come from linear hyperplanes in $V \otimes_{k} \rr$ whose normal vectors has the form $\gg v - v$.
    Therefore these linear hyperplanes are $k$-rational subspaces (have the form $U \otimes_{k} \rr$ for $U$ a subspace of $V$).
    Hence the $\gG$-stabilizers of the bounding hyperplanes are cocompact arithmetic groups of simplest type and the bounding hyperplanes project to immersed compact submanifolds of $H$.
    Hence, taking the full preimage i.e. all $\gG$-translates of these bounding hyperplanes, gives a locally finite $\gG$-invariant collection of $k$-rational hyperplanes.

    Since the stabilizers of hyperplanes are quasiconvex subgroups, by Theorem \ref{t:hwvr} each is a virtual retract of $\gG'$.
    This implies that by passing to a further finite cover $H'$, we can arrange that each hyperplane from this collection maps to an embedded submanifold and is a virtual retract.
    Note that the components of the complement of this collection in $\hh^n$ are all disks as they are finite intersections of halfspaces.
    Since each component is contained in a translate of the Dirichlet domain, it is bounded, and hence each component of the complement downstairs in $H'$ is also a disk.
    Thus we obtain a tractable collection in $H'$.
    Finally, we note that the intersection of $k$-rational hyperplanes is $k$-rational, so all intersections in this collection are hyperbolic manifolds of simplest type.
\end{proof}
\begin{remark}
    As explained in~\cite{bbks21}, in general not all totally geodesic closed submanfolds of $H$ arise from $k$-rational subspaces.
    Such submanifolds are still of simplest type, but the field of definition of the restricted quadratic form may increase.
\end{remark}

Applying Lemma \ref{l:ss} and noting that vanishing of $\overline\beta$ is a commensurability invariant gives
\begin{corollary}\label{c:odddim}
    If there is a closed arithmetic hyperbolic manifold of simplest type that does not have the upper $\F$-Singer property, then there is such a manifold of odd dimension.
\end{corollary}

By Agol's fibering theorem for closed hyperbolic $3$-manifolds, the dimension of a potential counterexample is at least 5. Now we are ready to discuss hyperbolizations.

\subsection*{On $\F $-Singer property for hyperbolizations}

If $L$ is a full subcomplex of a flag triangulated sphere $S^{m-1}$, then $hP_{L}$ has a natural thickening in $hP_{S^{m-1}}$, which is an $m$-manifold with boundary.
We do not know whether or not there are hyperbolizations $hP_{L}$ that do not have upper $\F$-Singer property, but we will show that if there are such examples then the lowest dimensional ones can be realized by closed manifolds of the form $hP_{S^{n-1}}$.
Moreover, if the lowest dimension $n$ is even, then $hP_{O\gD^{n-1}}$ does not have upper $\F$-Singer property.
To that end fix an initial Charney--Davis piece of some large dimension $N$, a field $\F$, and for $n\leq N$ consider the progressively stronger statements
\begin{itemize}
    \item[($1_n$)] $\ub_k(hP_{O\gD^{n-1}};\F )=0$ for $k>n/2$,
    \item[($ 2_n$)] $\ub_k(hP_{S^{n-1}};\F )=0$ for $k>n/2$, for flag triangulated spheres $S^{n-1}$,
    \item[($ 3_n$)] $\ub_k(hP_{L};\F )=0$ for $k>n/2$, for full subcomplexes $L$ in flag $S^{n-1}$.
\end{itemize}

Statement $ 2_n$ is the upper $\F $-Singer property for the closed manifolds $hP_{S^{n-1}}$, while $3_n$ is the upper $\F$-Singer property for a thickening of $hP_L$ inside of $hP_{S^{n-1}}$.

Finally, since $hP_{O\gD^{n-1}}$ is just a finite cover of the initial Charney--Davis manifold $A^{n}$, statement $1_n$ is the $\F $-Singer property for a particular arithmetic hyperbolic $n$-manifold of simplest type.

As a consequence of the Mayer--Vietoris inequalities given in Corollary \ref{c:mvinequalities} we get the following proposition.
\begin{proposition}
    In all dimensions $n$ we have
    \[
    2_{\leq n}\iff 3_{\leq n},
    \]
    and in even dimensions we have
    \[
    3_{<2d} \text{ and } 1_{2d} \iff 3_{\leq 2d}.
    \]
\end{proposition}
\begin{proof}

    First, note that for a full subcomplex $L$ of $S^{n-1}$, removing the vertices of $S^{n-1}$ that are not contained in $L$ one at a time and applying Corollary \ref{c:mvinequalities}(2) at each step shows that
    \begin{itemize}
        \item if $hP_{S^{n-1}}$ satisfies upper $\F$-Singer, and
        \item for full subcomplexes $L'$ of links in $S^{n-1}$, $hP_{L'}$ satisfies upper $\F$-Singer,
    \end{itemize}
    then $hP_L$ does as well.
    In other words,
    \[
    3_{<n} \text{ and } 2_n\implies 3_{\leq n}.
    \]

    Inductively, starting with the fact that $3_1$ is true, this shows
    \[
    2_{\leq n}\implies 3_{\leq n}.
    \]
    The other direction is clear.
    This proves the first part of the proposition.

    Now, look in even dimensions $n=2d$.
    Suppose $3_{<2d}$ is true, and observe two things.

    First, if there is a counterexample $L\subset S^{2d-1}$ to $3_{2d}$ then there is a counterexample to $3_{2d}$ that is a $(2d-1)$-simplex: If $L$ is not a simplex, then it contains a vertex $v\in L$ such that $\St(v)$ is a proper subcomplex of $L$.
    Now, $3_{<2d}$ and Corollary \ref{c:mvinequalities}(1) implies that either $\St(v)$ or $L-v$ is a counterexample to $3_{2d}$ with fewer vertices.
    Repeating this, we obtain a counterexample that is a simplex.
    Since $3_{<2d}$ is true, the simplex must have dimension $2d-1$.
    Second, removing vertices from $O\gD^{2d-1}$ that are not contained in $\gD^{2d-1}$ one at a time and applying Corollary \ref{c:mvinequalities}(2) using $3_{<2d}$ at each step, we conclude that $hP_{\gD^{2d-1}}$ satisfies upper $\F$-Singer if $hP_{O\gD^{2d-1}}$ does.
    Thus,
    \[
    3_{<2d} \text{ and } 1_{2d} \implies 3_{2d}.
    \]

    The converse is clear.
\end{proof}

Since $hP_{O\gD^{2d-1}}$ is an arithmetic hyperbolic $(2d)$-manifold of simplest type, Corollary \ref{c:odddim} implies that if $1_{2d}$ does not hold, then there is an odd-dimensional arithmetic hyperbolic manifold of simplest type that does not satisfy upper $\F$-Singer.
Together with Agol's fibering theorem this implies:
\begin{corollary}\label{odddim}
    If there is a full subcomplex $L$ in a flag triangulated sphere $S^{m-1}$ such that $\ub_k(hP_L;\F )>0$ for some $k>m/2$, then there is a closed $n$-manifold either of the form $hP_{S^{n-1}}$, or arithmetic hyperbolic of simplest type, that does not satisfy the upper $\F $-Singer property, for some odd $n$ in the interval $[5,m]$.
\end{corollary}

\subsection*{Proofs of Theorems \ref{maintheorem}(1) and \ref{nofibertheorem}}

Now we can prove the first part of our main theorem:
\begin{prevthm}
    {\ref{maintheorem}(1)} For any odd prime $p$, there is a closed, aspherical, $n$-manifold $\M^n$ of dimension either $n=5$ or $n=7$ with special hyperbolic fundamental group such that $\lb_k(\M;\Fp)>0$ for some $k$.
\end{prevthm}
\begin{proof}
    Fix an odd prime $p$.
    Let $(N,\d N)$ be the compact aspherical $7$-manifold with boundary provided by Theorem \ref{t:hypmanifold}.
    It has special hyperbolic fundamental group and $\lb_4(N;\Fp)>0$.
    Pick a flag triangulation of the boundary $\d$ and let $\M=hP^N_{\d}$ be the result of applying the hyperbolic reflection group trick using this triangulation.
    We showed that this is a closed aspherical $7$-manifold with virtually special hyperbolic fundamental group.

    By Proposition \ref{p:singerdavistrick}, either this manifold has $\lb_4(\M;\Fp)\geq\lb_4(N;\Fp)>0$, or there is a full subcomplex $L$ of a flag triangulation of $S^5$, such that $\ub_4(hP_L; \Fp)>0$.
    In the latter case, by Corollary \ref{odddim}, there is a closed aspherical $5$-manifold $H^{5}$ with virtually special hyperbolic fundamental group, which doesn't have upper $\Fp$-Singer property.
    So, in summary, after passing to a special finite cover, we obtain a closed aspherical manifold of dimension $7$ (if it is $\M$) or $5$ (if it is $H^{5}$) with special hyperbolic fundamental group and non-vanishing lower $\Fp$ homology growth (if it is $\M$) or upper $\Fp$-homology growth (if it is $H^{5}$).
    Therefore, Corollaries \ref{nosup} and \ref{c:Babove} imply Theorem \ref{maintheorem}(1).
\end{proof}

Theorem \ref{mappingtorustheorem} then immediately implies:
\begin{prevthm}
    {\ref{nofibertheorem}} There exists a closed, odd-dimensional, aspherical manifold $\M$ with special hyperbolic fundamental group that does not virtually fiber over a circle.
\end{prevthm}

One somewhat unsatisfying aspect of this line of argument is that it does not say a given odd dimensional manifold does not virtually fiber, but only that there is a such a manifold of a potentially lower odd dimension.
The last section of the paper is occupied with addressing this problem.

\section{Hyperbolic reflection group trick via barycentric subdivisions\label{s:proofs}
}

We have enough information from the skew field approach to give---for large primes $p$---explicit $7$-dimensional, aspherical examples that do not satisfy $\Fp$-Singer, and also $7$-dimensional $\Q$-aspherical examples that do not satisfy $\Q$-Singer.

The tool for doing this is a version of the hyperbolic reflection group trick that preserves $\Q$-homology growth above the middle dimension, and also $\Fp$-homology growth above the middle dimension for large primes.
\begin{prevthm}
    {\ref{t:barytrick}}[Better hyperbolic reflection group trick]

    For each dimension $n$ there is a choice of Charney--Davis piece $CD^n$ and a corresponding finite collection of exceptional primes $S_n$, such that for any compact $n$-manifold with boundary $(N,\d N)$ and any triangulation $\d$ which is a barycentric subdivision of a triangulation of the boundary, the hyperbolic reflection group trick $\M=hP_{\d}^{N}$ satisfies the following inequalities for $k>n/2$:
    \begin{enumerate}
        \item $b_k^{(2)}(N)\leq b_k^{(2)}(\M)$,

        \item $\lb_{k}(N;\Q)\leq \lb_k(\M;\Q)$ and $\ub_{k}(N;\Q)\leq \ub_k(\M;\Q)$,

        \item $\lb_k(N;\Fp)\leq \lb_k(\M;\Fp)$ and $\ub_k(N;\Fp)\leq \ub_k(\M;\Fp)$ for $p \notin S_n$.
    \end{enumerate}
\end{prevthm}
\begin{proof}

    When the triangulation of $\d N$ is a barycentric subdivision of another triangulation, i.e. $\d=bT$, then the vertex removal process in Proposition \ref{p:singerdavistrick} can be carried out by removing centers $v_{\gs}$ of simplices $\gs$ of $T$, starting by removing barycenters of all $n-1$ simplices of $T$, then barycenters of all $n-2$ simplices of $T$, and so on.
    The links that appear in this process are precisely the barycentric subdivisions of the boundaries of these simplices $\gs$ embedded as full subcomplexes in the link of the vertex $v_{\gs}$ in the ambient manifold $\d$ i.e.     
    \[
    b\d\gs^{i}=\Lk_{bT^{(i)}}(v_{\gs^{i}})\subset \Lk_{\d}(\gs^{i})=S^{n-2}
    \]
    where $T^{(i)}$ is the $i$-skeleton of the triangulation $T$, and $bT^{(i)}$ is its barycentric subdivision, and $\gs^{i}$ is some $i$-simplex.

    Note that the $hP_{b\d\gs}$ are hyperbolizations of a particular finite collection of manifolds (they are called Tomei manifolds) of dimension $\leq n-1$.
    It follows from the smooth hyperbolization technology of Ontaneda \cite{o20}*{Main Theorem} that for any $\ge>0$ we can pick an appropriate\footnote{The normal injectivity radius of all walls needs to be large, the existence of such a $CD^n$ is in \cite{o20}*{Theorem 9.1}.}
    initial Charney--Davis piece $CD^n$ for which all these manifolds will have smooth Riemannian metrics whose sectional curvature is pinched between $-1$ and $-1-\ge$.
    When $\ge$ is sufficiently small, a result of Donnelly--Xavier \cite{dx84} implies that the even dimensional $hP_{b\d\gs}$ satisfy the Singer conjecture while the odd dimensional $hP_{b\d\gs}$ satisfy it except possibly in the middle two dimensions.
    Consequently for all these manifolds of dimension $\leq n-1$, we have
    \[
    b^{(2)}_{>n/2}(hP_{b\d\gs})=0,
    \]
    We conclude by Corollaries \ref{specialp} and \ref{c:Babove} and virtual specialness of the $\pi_1(hP_{b\d\gs})$ that for large primes $p$,
    \[
    \ub_{>n/2}(hP_{b\d\gs};\Fp)=\lb_{>n/2}(hP_{b\d\gs};\Fp)=\lb_{>n/2}(hP_{b\d\gs};\Q)=0.
    \]

    After cutting along the $hP_{b\d\gs}$, we are left with copies of $N$ and copies of $hP_{b\gs}$.
    Each copy of $N$ is a retract of $\M$ by Theorem \ref{t:hypdavistrick}, and $hP_{b\d\gs}$ and $hP_{b\gs}$ are quasiconvex in $\M$ by Lemma \ref{l:quasiconvex}.
    Hence the end result of the cutting procedure and the components of $hP_{b\d\gs}$ are virtual retracts of $\M$ by Corollary \ref{c:vretract}.
    We apply Corollary \ref{homologygrowthmv} to finish the proof of \eqref{i:lbQ} and \eqref{i:lb}.

    For \eqref{i:b2} we use an $L^{2}$-version of Corollary \ref{c:rtfnthmv}, which works for \emph{any} fundamental group.
\end{proof}
\begin{prevthm}
    {\ref{maintheorem} (2)}

    For large primes, there is a closed, aspherical, $7$-manifold $\M^7$ with special hyperbolic fundamental group such that $\lb_k(\M;\Fp)>0$ for some $k$.
\end{prevthm}
\begin{proof}

    Let $N^7$ be the seed manifold obtained form Theorem \ref{t:hypmanifold} for a prime $p\notin S_7$.
    It has special hyperbolic fundamental group and $\lb_4(N;\Fp)\neq 0$, so applying the hyperbolic reflection group trick with barycentrically subdivided boundary gives a closed, aspherical $7$-manifold $\M$ with virtually special hyperbolic fundamental group and $\lb_4(\M; \Fp) \neq 0$.
    Passing to a finite cover, if necessary, gives such a manifold $\M'$ with special fundamental group and with $\lb_4(\M';\Fp)\neq 0$.
\end{proof}
Note that this proof avoids the inductive arguments of Section \ref{s:induction}.
We now prove Theorem {\ref{rational}}.
\begin{prevthm}
    {\ref{rational}} There is a closed, rationally aspherical $7$-manifold $\M$ with virtually special hyperbolic fundamental group and $b_4^{(2)}(\M)\neq 0$.
\end{prevthm}
\begin{proof}
    By Theorem \ref{t:nosquare} there is a flag no-square triangulation of the $3$-sphere.
    Denote it by $S^3$.
    By Corollary \ref{c:modpl2}, if the finite groups used to define the graph product $G_{S^3}$ are large enough then $\lb_4(G_{S^3};\Q)>0$.
    Fix such a graph product $G_{S^3}$.
    Note that it is virtually of finite type, and a virtual $4$-dimensional duality group by Theorem \ref{dualitycriterion}.
    Therefore, \cite{a18}*{Theorem 18} implies that there is a compact $\Q$-aspherical $7$-manifold with boundary $(N^7,\d)$ whose fundamental group is a finite index torsion-free subgroup $\gG <G_{S^3}$.
    Let $\M^7=hP^N_{\d}$ be the result of applying the hyperbolic reflection group trick (with barycentric subdivisions) to this manifold.
    Then
    \[
    \lb_4(\M;\Q)\geq \lb_4(N;\Q)=\lb_4(\gG ;\Q) = [G_{S^3}:\gG ] \lb_4(G_{S^3};\Q)>0,
    \]
    where the first inequality is Theorem \ref{t:barytrick}\eqref{i:b2} and the first equality follows because the rational homology of $\gG $ can be computed from the action on a $\Q$-acyclic instead of contractible complex.
    Finally, since the triangulation we picked was no-square, the group $G_{S^3}$ (and hence $\gG $) is hyperbolic and virtually special.
    So by Theorem \ref{t:manifoldspecial}, the fundamental group of $\M$ is hyperbolic and virtually special.
    In particular, it is residually finite, so L\"uck approximation implies $b^{(2)}_4(\M)=\lb_4(\M;\Q)>0$.
    Replacing $\M$ by a finite cover if necessary, we obtain such a manifold with special fundamental group.
\end{proof}

\appendix
\renewcommand{\thesection}{\Roman{section}}

\section{Rational homology growth and \texorpdfstring{$L^2$}{L\^2}, qualitatively} The connection between $\Q$-Betti number growth and $L^2$ rests on two pillars.
The first is Kazhdan's crucial observation that if the $\Q$-Betti numbers grow fast in a chain of regular finite covers converging to the universal cover, then the universal cover has a non-vanishing $L^2$-harmonic cycle.
The second is L\"uck's converse to Kazhdan's criterion.
L\"uck showed, by an ingenious approximation argument, that $L^2$-harmonic cycles in the universal cover lead to fast $\Q$-Betti number growth in \emph{any} chain of regular finite covers converging to it.
He also obtained a quantitative statement, showing that the analytically defined $L^2$-Betti numbers agree with the limit of the normalized $\Q$-Betti numbers of any such chain \cite{l94}.

In this appendix we sketch qualitative versions of these two arguments, in order to highlight the basic ideas behind them.
The natural setting for these arguments are not chains of regular finite covers, but covers of sufficiently large universal\footnote{The \emph{universal injectivity radius} is the smallest $R$ such that the map from the universal cover $\widetilde X\to X$ is injective on $R$-balls.}
injectivity radius.

\subsection*{Kazhdan's criterion}

We start with the following (obvious) linear algebraic lemma.
\begin{lemma}
    Any $n$-dimensional subspace of $\R^N$ has a unit vector with a coordinate of size $\geq n/N$.
\end{lemma}
\begin{proof}
    Let $v_i$ be an orthonormal basis for $\R^N$ and $P$ be the orthogonal projection to the $n$-dimensional subspace.
    Its trace is $n=\sum_{i=1}^N\langle Pv_i,v_i\rangle $, so at least one of vectors $Pv_i$ has a coordinate of size $\geq n/N$.
    The unit vector $Pv_i/||Pv_i||$ has a bigger coordinate.
\end{proof}
The Lemma implies that if the dimensions of spaces of harmonic cycles grow linearly in the degree of the cover, then in each cover we can find such cycles of unit norm and uniformly bounded away from zero on some cell.

Now, suppose we have a sequence of finite covers $X_k\to X$, each with a unit norm harmonic cycle $z_k$ and a cell $e_k$ such that $\langle z_k,e_k\rangle \geq c>0$.
Pick a subsequence so that all the cells $e_k$ lie over the same cell $e$ in $X$.
Suppose, in addition, that the covering map $\widetilde X\to X_k$ is injective on balls of radius $k$.
Then, we have an isometry of $k$-neighborhoods $\phi:B_k(\widetilde e)\cong B_k(\widetilde e_k)\cong B_k(e_k)$, where the first map is a covering translation and the second map is the projection $\widetilde X\to X_k$.
Using this isometry we pick lifts $\widetilde z_k=\phi^{-1}(z_k)$.
They have $L^2$-norm $=1$ and satisfy $\langle \widetilde z_k,\widetilde e\rangle =\langle z_k,e_k\rangle \geq c$.
The $\widetilde z_k$ may not be harmonic (in fact, they are just chains, not cycles), but since $z_k$ is harmonic, this failure happens outside the $(k-1)$-neighborhood of $\widetilde e$.
So, we can pick a subsequence of $\widetilde z_k$ that converges to a harmonic cycle of $L^2$-norm $\leq 1$ and non-vanishing at $\widetilde e$.
We have proved
\begin{theorem}[Kazhdan's criterion \cite{k75}] Let $X$ be a finite complex and $X_k\to X$ a sequence of finite covers whose universal injectivity radius goes to infinity.
    If the $i^{th}$-Betti number grows linearly, i.e.     
    \[
    \limsup_k \frac{b_i(X_k;\Q)}{\abs{X_k\to X}}>0,
    \]
    then the universal cover $\widetilde X$ has an non-zero $L^2$ harmonic $i$-cycle.
\end{theorem}

\subsection*{L\"uck approximation}

L\"uck proved a converse to Kazhdan's criterion:
\begin{itemize}
    \item \emph{$L^2$ harmonic cycles on the universal cover produce linear Betti number growth.}
\end{itemize}
The goal is to bound below the trace of the orthogonal projection to the kernel of the Laplacian in covers, and L\"uck's idea is to approximate the orthogonal projection by polynomials in the combinatorial Laplacian, and exploit that this Laplacian is a bounded operator with integer entries.

Let $X$ be a finite complex with fundamental group $G=\pi_1(X)$ and $X_k\to X$ a finite cover for which $\pi:\widetilde X\to X_k$ is injective on $k$-balls.
Running our earlier argument backwards, if there are $L^2$-harmonic cycles, then there is a $i$-cell $\widetilde e$ such that, for the orthogonal projection $\widetilde P$ to the space of harmonic $L^2$ chains we have $\langle \widetilde P\widetilde e,\widetilde e\rangle > c>0$.
Approximate $\widetilde P$ by a polynomial $f(\widetilde \gD)$ in the combinatorial Laplacian $\widetilde \gD$.
Since it is a bounded operator, we can do this with $f(0)=1$ and $f\geq 0$ on the spectrum of $\widetilde\gD$, so we also get $\langle f(\widetilde \gD)\widetilde e,\widetilde e\rangle \geq c$.
The advantage of polynomials in $\widetilde\gD$ over the orthogonal projection $\widetilde P$ is that $f(\widetilde\gD)\widetilde e$ is a \emph{finite} chain, so if the injectivity radius $k$ of $X_k$ is greater than the diameter of $f(\widetilde\gD)\widetilde e$, then all the translates $f(\widetilde\gD)\gamma\widetilde e$ embed isometrically in $X_k$.
This implies $\langle f(\gD_k)\pi(\gamma \widetilde e),\pi(\gamma\widetilde e)\rangle \geq c$ for all $\gamma\in G$.
There are $\abs{X_k\to X}$ different $i$-cells of the form $\pi(\gamma\widetilde e)$ in $X_k$.
Summing over all of them, we get a lower bound for the normalized trace
\[
\frac{\tr(f(\gD_k))}{\abs{X_k\to X}}\geq c.
\]
We are almost done.
We just need to get back to the orthogonal projection $P_k$ in the finite cover.
For this, one needs an estimate the number $N_{\ge, k}$ of small but nonzero eigenvalues ($\leq\ge$) of $\gD_k$ compared to the number of $i$-cells $N_k$ in the cover.
This is given by another linear algebraic lemma, namely Lemma \ref{smalleigs} in Section \ref{s:homologygrowth}.
The final important point is that---for a finite complex $X$---the norms of both the combinatorial Laplacian in the universal cover and the combinatorial Laplacians in all the finite covers are bounded by a single finite constant $D$ (see \cite{l94}*{Lemma 2.5.}).
Therefore, the upper bound on the normalized number of $\ge$-small, nonzero eigenvalues coming from Lemma \ref{smalleigs} does not depend on the cover.
From here it is easy to see that the normalized trace of $f(\gD_k)$ is close to the normalized trace of the orthogonal projection $P_k$ once $f$ is close enough to the characteristic function $\chi_0$ on the spectrum of $\gD_k$.

\section{Residual finiteness}

In this appendix, we show that if our seed manifold $(N, \d N)$ has residually finite fundamental group then the output $\M$ of our hyperbolic reflection group trick has residually finite fundamental group as well.
In particular, we don't need to assume that $\pi_1(N)$ is hyperbolic and virtually special.
Along the way, we give a general criterion for basic constructions of mirrored spaces to be residually finite.
\subsection*{Profinite topology}

The profinite topology on a group $G$ has as a basis the collection of cosets of finite index normal subgroups of $G$.
Multiplication by a group element and taking inverses induce continuous maps with this topology.

A subset $C$ of $G$ is called \emph{separable} if it is closed in the profinite topology, i.e. for any $g \in G - C$ there is a finite index normal subgroup $N$ such that $Cg^{-1} \cap N = \emptyset$.
We will express this by saying $g$ can be separated from $C$.

If $H$ is a subgroup of $G$ then separability is equivalent to the statement that $H$ is an intersection of finite index subgroups of $G$, i.e. for any $g \in G - H$ there is a finite index subgroup $N$ such that $H \subset N$ and $g \notin N$.
In particular, $G$ is residually finite if and only if $\{1\}$ is separable (or, equivalently, the profinite topology is Hausdorff.)

We shall need the following theorem, which combines work of Haglund--Wise \cite{hw08} and Minasyan \cite{m06}.
\begin{theorem}[\cite{hw08}+\cite{m06}]\label{t:hwm}
    If $G$ is a hyperbolic, virtually special group, then every quasiconvex subgroup is separable.
    Furthermore, if $H$ and $K$ are two quasiconvex subgroups, then the double coset $HK$ is separable.
\end{theorem}
Note that, since conjugation preserves quasiconvexity in hyperbolic groups \cite{g97a}*{Lemma 1.9}, this theorem implies that the double cosets $HgK$ are also separable for each $g \in G$.

We shall also need the following easy lemma.
\begin{lemma}\label{l:separable}
    Suppose that $G$ is a residually finite group, and $H$ is a closed subset of $G$ in the profinite topology.
    Let $\gF$ be a finite subset of $G - H$.
    Then there is a finite index normal subgroup $\gG \lhd G$ such that $\gF \cap H\gG = \emptyset$.
\end{lemma}
\begin{proof}
    Since $H$ is closed, for each $\gf \in \gF$ the translate $H\gf^{-1}$ is closed and does not contain the identity.
    Therefore, there is a finite index normal subgroup $\gG_{\gf}$ such that $H\gf^{-1} \cap \gG_{\gf} = \emptyset$.
    Since $\gG_{\gf}$ is normal, $\gG_{\gf} H = H \gG_{\gf}$, hence $\gf \notin H\gG_{\gf}$.
    Taking $\gG = \bigcap_{\gF} \gG_{\gf}$ works.
\end{proof}

\subsection*{On residual finiteness of basic constructions}

Recall from Section \ref{s:hrt} that if $X$ is a mirrored complex, then the universal cover of $X$ is also a mirrored complex.
The mirrors in the universal cover are precisely the path components of preimages of mirrors in $X$.
Let $(W, S)$ be the RACG corresponding to the mirrored complex $X$, and $(\widetilde W, \widetilde S)$ the RACG corresponding to the universal cover.
\begin{theorem}\label{t:davisresf}
    Let $X$ be a finite mirrored complex and $W$ the associated right-angled Coxeter group.
    Let $\widetilde S$ and $\widetilde W$ as above.
    Suppose $G = \pi_1(X)$ is residually finite, and that all the double cosets $\Stab(\tilde s)\Stab(\tilde t)$ for $\tilde s, \tilde t \in \widetilde S$ are separable in $G$.
    Then $\widetilde W \rtimes G$ is residually finite.
\end{theorem}
\begin{proof}
    Let $g \in \widetilde W \rtimes G$.
    If $g$ maps nontrivially to $G$, then since $G$ is residually finite we can find a finite quotient of $G$ where $g$ survives.
    So, we can assume that $g \in \widetilde W$.
    Write $g$ as a product of generators of $\widetilde W$, and let $\widetilde T \subset \widetilde S$ denote the finite set of generators appearing in the product.
    Hence $g$ is contained in the finitely generated special subgroup $W_{\widetilde T}$.

    We claim that there is a finite index subgroup $\gG$ of $G$ which preserves disjointness of the mirrors in this collection, i.e. if $\tilde s, \tilde t \in \widetilde T$ and $\tilde s \cap \tilde t = \emptyset$ then $\tilde s \cap \gg\tilde t = \emptyset$ for all $\gg \in \gG$.

    Since $\widetilde T$ is finite, it is enough to find such a subgroup for each pair $\tilde s, \tilde t \in \widetilde T$, then we can take $\gG$ to be the intersection of all these subgroups.

    So, suppose we have two disjoint mirrors $\tilde s$ and $\tilde t$.
    Since the stabilizer $\Stab(\tilde s)$ acts cocompactly on $\tilde s$, up to $\Stab(\tilde s)$-action there are at most finitely many $G$-translates of $\tilde t$ which intersect $\tilde s$.
    Choose a finite set $\gF \subset G$ of group elements representing these translates.

    Note that $\gF \cap \Stab(\tilde s)\Stab(\tilde t) = \emptyset$, since if $\gf = g_{\tilde s}g_{\tilde t}$ where $g_{\tilde s} \in \Stab(\tilde s)$ and $g_{\tilde t} \in \Stab(\tilde t)$, then $\gf\tilde t = g_{\tilde s}\tilde t$, and $g_{\tilde s}$ preserves disjointness between $\tilde s$ and $\tilde t$.

    Since the double coset $\Stab(\tilde s)\Stab(\tilde t)$ is separable by assumption, Lemma \ref{l:separable} implies that we can find a finite index $\gG \lhd G$ so that
    \[
    \gF \cap \Stab(\tilde s) \Stab(\tilde t) \gG = \emptyset.
    \]
    Since $\gG$ is normal, $\Stab(\tilde s) \Stab(\tilde t) \gG = \Stab(\tilde s) \gG \Stab(\tilde t) $, and hence
    \[
    \gG \cap \Stab(\tilde s) \gF \Stab(\tilde t) = \emptyset.
    \]

    We claim that $\tilde s \cap \gg\tilde t = \emptyset$ for each $\gg \in \gG$.
    Indeed, if $\gg\tilde t$ intersects $\tilde s$, then $\gg \tilde t = g \gf \tilde t$ for some $\gf \in \gF$, and $g \in \Stab(\tilde s)$.
    Therefore, $\gf^{-1}g^{-1} \gg \in \Stab(\tilde t)$, so $\gg$ is contained in the double coset $\Stab(\tilde s) \gf \Stab(\tilde t)$.

    So, now we have a finite index normal subgroup $\gG$ of $G$ which preserves disjointness of mirrors corresponding to generators in the finitely generated Coxeter subgroup $W_{\widetilde T}$.
    Let $\overline S$ be the set of $\gG$-orbits of elements of $\widetilde S$.
    For $\bar s\neq \bar t \in \overline S$ define $m_{\bar s \bar t} = 2$ if there exist representatives of $\bar s$ and $\bar t$ which intersect, and set $m_{\bar s \bar t} = \infty$ otherwise.
    Let $\overline W$ be the corresponding finitely generated RACG.

    The natural quotient homomorphism $\widetilde W \to \overline W$ is $\gG$-equivariant with respect to the left action of $\gG$ on $\widetilde W$ and the trivial action on $\overline W$.
    Therefore, it induces a map $\widetilde W \rtimes \gG \to \overline W \times \gG$.
    Composing with the projection onto the first factor gives a map $\widetilde W \rtimes \gG \to \overline W$, which by construction is injective on $W_{\widetilde T}$, and in particular maps $g$ nontrivially.
    Since $\overline W$ is a finitely generated RACG, it is residually finite, hence we can detect the image of $g$ in a finite quotient.
    This induces a finite quotient of $\widetilde W \rtimes \gG$ where $g$ survives.
    Since $\widetilde W \rtimes \gG$ is a finite index subgroup of $\widetilde W \rtimes G$, this induces a finite quotient of $\widetilde W \rtimes G$ where $g$ survives, so we are done.
\end{proof}
\begin{remark}
    If $X$ has only one mirror $X_s$, then $\cu(W, X)$ is the double of $X$ over $X_s$.
    Therefore, this theorem can be seen as a generalization of the classical fact that the double $G *_H G$ of a residually finite group over a separable subgroup is residually finite.
\end{remark}

\subsection*{Davis trick preserves residual finiteness}

If $W$ is infinite, then $\pi_1(\cu(W,X))$ is infinitely generated if $X$ is not simply connected.
In this case, a finite index torsion-free subgroup $\gG$ of $W$ still acts properly and cocompactly on $\cu(W,X)$, and we can form the quotient $Y = \cu(W,X)/\gG$.
If $X$ is a compact aspherical manifold with boundary and the $X_s$ are dual cells to a flag $PL$ triangulation of $\d X$, then $Y$ is a closed aspherical manifold.
This construction of $Y$ is called the \emph{Davis reflection group trick} applied to $X$.
In any case, $\pi_1(Y)$ is finite index in $\widetilde W \rtimes \pi_1(X)$.
Explicitly, let $\pi : \widetilde W \to W$ as above, and set $\widetilde \gG = \pi^{-1}(\gG)$.
Then $\widetilde \gG$ is torsion-free, finite index in $\widetilde W$, and stable under the action of $\pi_1(X)$.
It turns out that $\pi_1(Y)$ is precisely the semi-direct product $\widetilde \gG \rtimes \pi_1(X)$.

Note that for the usual reflection group trick, the mirrors $X_s$ are contractible, so we get an immediate corollary:
\begin{corollary}
    Let $X$ be a compact aspherical manifold with boundary with $\pi_1(X)$ residually finite.
    Then the fundamental group of any Davis reflection trick $Y$ applied to $X$ is residually finite.
\end{corollary}

\subsection*{Hyperbolic reflection group trick preserves residual finiteness}

Let $L$ be a flag complex.
We know from Theorem \ref{t:hcone special} that $\hc{L}$ has hyperbolic and virtually special fundamental group (we'll assume that $\hc{L}$ is connected, otherwise take the component containing the cone vertex).

Now, let $N$ be a simplicial complex containing $L$.
Similarly to Section \ref{s:special}, let $hK_L^N$ be obtained from the hyperbolized cone $hK_L$ by removing a small enough $\ge$-ball centered at the cone vertex $o$ and gluing in $N$.
As in the proof of Theorem~\ref{t:hypdavistrick}(1), $hK_L^N$ retracts onto $N$, so its fundamental group splits as a semidirect product:

\[
\pLhc{N}{L}= \pi_1(hK_{\widehat L}^{\widetilde N}) \rtimes \pi_1 (N),
\]
where $\widetilde N\to N$ is the universal cover of $N$ and $\widehat L\to L$ is the induced cover of $L$.
Note that $\pi_1(\hc{\widehat L}^{\widetilde N})=\pi_1(hK_{\widehat L})$, since $\widetilde N$ is simply connected.

The space $hK_{\widehat L}^{\widetilde N}$ looks like a bunch of hyperbolic row-houses sitting on $\widehat L$ in $\widetilde N$, see Figure~\ref{f:cover}.
Note that a cover $N'\to N$ induces a cover $hK^{N'}_{L'}\to hK^N_L$ of the same degree.
\begin{figure}

    \tikzset{ house/.pic={
    \draw (0,-1)--(4,-1)--(4,3) to [out=180, in=-45] (2,4) to [out=-135, in=0] (0,3)--cycle;
    \draw (1,1) to [out=0, in=-45] (1.1,2) to [out=135, in=180] (2,3) to [out=0, in=45] (2.9,2) to [out=-135, in=180] (3,1) ;
    \draw (1.5,2) to [out=-30, in=-150] (2.5,2) ;
    \draw (1.6,1.95) to [out=30, in=150] (2.4,1.95) ; }}

    \tikzset{ rhouse/.pic={

    \draw (0,0)--++(54:6)--++(126:6)--++(-126:6)--cycle;

    \draw (-1,5) to [out=0, in=-45] (-.9,7) to [out=135, in=180] (0,8) to [out=0, in=45] (.9,7) to [out=-135, in=180] (1,5) ;
    \draw (-.5,7) to [out=-30, in=-150] (.5,7) ;
    \draw (-.4,6.95) to [out=30, in=150] (.4,6.95) ; } }
    \begin{tikzpicture}[scale=.2]
        \coordinate (shift) at (-18,20);
        \begin{scope}[scale=1.5, shift=(shift)]

            \foreach \i in {0,4,...,20}
            \path (\i,0) pic[ transform shape] {house};

            \draw[fill = lightgray] (0,-1) -- (24,-1) -- (24, -3) -- (0, -3) -- (0,-1);
            \node at (12,-2) {\footnotesize $\widetilde N$};
            \node at (26, -1) {$\Lhc{\widetilde N}{\hat L}$};
            \node at (-1, -1) {$\hat L$};
        \end{scope}

        \foreach \i in {0,72,...,360}
        \path (0,0) pic[transform shape, rotate=\i] {rhouse};

        \draw[fill = lightgray] (0,0) circle [radius = 3];
        \draw[fill = white] (0,0) circle [radius = 1];
        \node at (0,1.75) {\footnotesize $N$};
        \node at (11, 1) {$\Lhc{N}{L}$};
        \draw [->] (0, 15) -- (0, 10);
    \end{tikzpicture}

    \caption{The space $hK_L^N$ and its $\pi_1(N)$-cover.}\label{f:cover}
\end{figure}
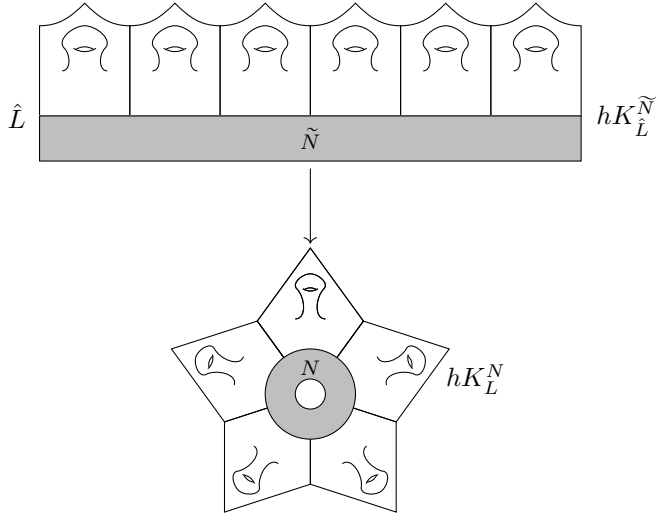
\begin{proposition}\label{p:branched}
    Let $N$ be a finite complex with residually finite fundamental group and $L\subset N$ a flag subcomplex.
    Let $s, t$ be vertices of $L$, and let $H_s$ and $H_t$ denote any conjugates of the fundamental groups of components of the hyperbolized mirrors $h(K_{s})$ and $h(K_{t})$ respectively.
    Then we have:
    \begin{enumerate}
        \item $\pLhc{N}{L}$ is residually finite.
        \item $H_s, H_t$ and $H_sH_t$ are separable subsets in $\pLhc{N}{L}$.
    \end{enumerate}
\end{proposition}
\begin{proof}
    We first prove statement (1).
    Let $\gg \in \pi_1(\hc{L}^N)$, $\gg \neq 1$.
    We want to separate $\gg$ from 1 by a finite index subgroup.
    Let $p: \pi_1(\hc{L}^N) \to \pi_{1}(N)$ denote the retraction.
    We have two cases:

    Case 1. $\gg \notin \ker p$.
    Then we can separate $p(\gg)$ from 1 in $\pi_{1}(N)$, which is residually finite, and take preimages.

    Case 2. $\gg \in \ker p$.

    Then $\gg$ lifts to a nontrivial loop $\widehat\gg$ in $\pi_1(\hc{\widehat L}^{\widetilde N})=\pi_1(hK_{\widehat L})$, which is contained in a finite union of row houses.
    We think of this union as $\hc{Q}$ for a finite full subcomplex $Q$ of $\widehat L$.
    Since $\pi_1(N)$ is residually finite, we can choose a finite cover $N'\to N$, so that $Q$ projects injectively onto a full subcomplex of the induced cover $L' \to L$.

    Let $\gg''$ denote the image of $\widehat\gg$ in $\pLhc{N'}{ L'}$.
    Then we have a commutative diagram:
    \[
    \begin{tikzcd}[ in/.style = {draw=none,"\in" description, sloped}, ni/.style = {draw=none,"\ni" description, sloped}, md/.style = {mapsto, dashed}, ] \overline\gg \ar[dr, in] \ar[dd, md] \ar[rrr, md]&[-1em] & &[-3em] \gg'\neq 1 \ar[ni]{ld} \\[-1em]
        & \phc{Q} \ar[d, hook] \ar[hook]{r} & \phc{ L'} \\
        \widehat\gg \ar[r, in] \ar[d, md] & \phc{\widehat L} \ar{ur} \\
        [-1em] \widehat\gg \ar[r, in] \ar[dd, md] \ar[ddrrr, md, bend right=3] &\pLhc{\widetilde N}{\widehat L} \ar[rd,hook] \ar[equal]{u} \ar[d,hook] \\
        & \pLhc{N}{L} & \pLhc{N'}{ L'} \ar[hook']{l}{\text{f.i.}
        } \ar{uuu} \\[-1em]
        \gg \ar[in]{ur} &&& \gg'' \ar[ni]{ul} \ar[uuuuu, md] \ar[lll, md]
    \end{tikzcd}
    \]

    By construction of $Q$ the element $\widehat\gg$ is in the image of the top left vertical map, let $\overline\gg$ denote its preimage.
    The top horizontal map is injective since $Q$ maps injectively to a full subcomplex of $L'$, therefore $\overline\gg$ maps to a non-trivial element $\gg'$ in $\phc{ L'}$.
    By Theorem \ref{t:hcone special} $\phc{ L'}$ is virtually special, hence residually finite, so we can separate $\gg'$ from $1$.
    Since the diagram commutes, $\gg'$ is also the image of $\gg''$ under the right vertical map, so taking preimages we get a finite index subgroup of $\pLhc{N'}{ L'}$ separating $\gg''$ from $1$.
    Finally, since $\pLhc{N'}{ L'}$ is a finite index subgroup of $\pLhc{N}{ L}$ and $\gg''$ maps to $\gg$ under the bottom map, the same subgroup separates $\gg$ from $1$ in $\pLhc{N}{ L}$.

    For statement (2), we will only prove the statement for $H_s$, as the double coset argument is identical.
    We have that $H_s$ is contained in $\ker p$ and $\gg \in \pLhc{N}{L} - H_s$.
    Again, we have two cases:

    If $\gg \notin \ker p$, then $\gg$ is separated from $H_s$ by $\ker q p$, where $q$ is a finite quotient of $\pi_1(N)$ such that $q(p(\gg)) \ne 1$.

    So, we assume that $\gg \in \ker p$.
    Since $H_s$ is generated by finitely many loops, by the same argument as above we have that $H_s$ and $\gg$ are contained in $\phc{Q}$ for some finite, full subcomplex $Q$ of $\widehat L$, Hence, $H_s$ and $\gg$ map injectively with distinct images to $\phc{L'}$ for a certain finite cover $L'$ of $L$.
    The image of $h(K_{s})$ is totally geodesic in $\hc{L'}$, hence the image $H_s'$ of $H_s$ is a quasiconvex subgroup of $\phc{L'}$.
    By Theorem \ref{t:hwm} the image $ \gg'$ can be separated from $H_s'$ (or from the image of a double coset) by a finite index subgroup of $\phc{L'}$.
    This pulls back to a finite index subgroup of $\pLhc{N'}{L'}$ which separates $\gg$ and $H_s$, and the same subgroup separates $\gg$ from $H_s$ in $\pLhc{N}{L}$.
\end{proof}

\subsection*{Proof of Theorem \ref{t:hypdavistrick}(6)}

We now prove
\begin{prevthm}
    {\ref{t:hypdavistrick}(6)} Suppose $N$ is a compact aspherical manifold with a flag triangulation $\d$ of the boundary, and $\pi_1(N)$ is residually finite.
    Let $\M$ be the result of applying the hyperbolic reflection group trick.
    Then $\pi_1(\M)$ is residually finite.
\end{prevthm}
\begin{proof}
    The manifold $\M$ is a basic construction, where the seed manifold is $\Lhc{N}{\d}$, the Coxeter group is $(\zz/2)^{\abs{\d^{(0)}}}$, and for each $s \in \d^{(0)}$, the $s$-mirror is $\hc{\Lk s}$.
    By Proposition \ref{p:branched} $\Lhc{N}{\d}$ has residually finite fundamental group, and each component of the mirrors has separable fundamental group inside of $\pLhc{N}{\d}$.
    The stabilizers of components of lifts of mirrors in the universal cover of $\Lhc{N}{\d}$ are conjugates of these subgroups; therefore we can apply Theorem \ref{t:davisresf} to conclude that $\pi_1(\M)$ is residually finite.
\end{proof}

\section{Embedding Octahedralizations}\label{a:embedding}

Recall that the octahedralization $OL$ of a flag complex $L$ is the complex which has $2^{k+1}$ $k$-simplices $v_0^{\pm}*\dots*v_k^{\pm}$ for each $k$-simplex $v_0*\dots*v_k$ of $L$.
Alternatively, it is the link of the special vertex in the Salvetti complex $\bigcup_{v_0*\dots*v_k\subset L} S^1_{v_0}\times\dots\times S^1_{v_k}$ of the right-angled Artin group $A_L$.
In this appendix, we will sketch an alternate proof of the following result from \cite{ados16}, which is the main step in constructing low dimensional thickenings of Salvetti complexes.
\begin{theorem}[\cite{ados16}]\label{octaemb}
    If $d\neq 2$ and $L$ is a $d$-dimensional flag complex with $b_d(L;\F _2)=0$, then $OL$ piecewise linearly embeds in $S^{2d}$.
\end{theorem}
From here, a ``codimension three local unknotting'' result of Akin \cite{a69} shows that the triangulation of $OL$ can be extended to a triangulation of the $S^{2d}$, and then partially subdivided to obtain $OL$ embedded inside a flag triangulation of $S^{2d}$ as a full subcomplex.
That is precisely the ``if'' direction of Theorem \ref{t:manifold}.

The proof given here avoids configuration spaces in favor of van Kampen's earlier (equivalent) approach to embedding obstructions, and clarifies the role that the moment map immersion plays in \cite{ados16}, replacing it with a special class of immersions of $OL$ that are obtained from a generic, linear immersion of the underlying complex $L$ by a small linear perturbation.

\subsection*{Classical embedding theory (van Kampen + Whitney trick)}

Everywhere in this section all maps are piecewise linear and we assume that $d\neq 2$.
(When $d=2$ things are more subtle.)

Let $L$ be a $d$-dimensional complex.
We want to determine whether it embeds in $\R^{2d}$.
To that end, start with a generic immersion $f:L\to \R^{2d}$, pick orientations on all the $d$-simplices in $L$ and look at the (signed) intersection numbers of images of disjoint $d$-simplices $\gs$ and $\tau$ in $L$:
\[
V_{f,\gs,\tau}:=f(\gs)\cap f(\tau).
\]
This is called the \emph{intersection vector} of $f$.
The intersection vector is symmetric when $d$ is even and anti-symmetric when $d$ is odd, i.e. $V_{f,\gs,\tau}=(-1)^dV_{f,\tau,\gs}$.
\begin{proposition}[\cites{k33,fkt94}]\label{alg}
    If $V_f=0$ then there is an embedding of $L$ in $\R^{2d}$.
\end{proposition}
\begin{proof}[Proof idea] For disjoint $d$-simplices $\gs$ and $\tau$, one cancels pairs of intersections with opposite intersection numbers using the Whitney trick (this requires $d\neq 2$).
    At the end, one is only left with self intersections of $f(\gs)$ and intersections between adjacent simplices.
    It turns out that these can be cancelled as well (but this requires extra arguments, see \cite{k33} and the erratum \cite{k33a}\footnote{For an annotated English translation, see \url{https://sites.google.com/site/tutamnguyenphan/home-1}}or \cite{fkt94}).
\end{proof}

So, we need ways to modify an immersion $f$ so as to make the intersection vector equal to zero.

\subsection*{Finger moves}

We can push a $d$-simplex $\gs$ with a finger to make it go around a $(d-1)$-simplex $e$.
In fact, we can do this to several $(d-1)$-simplices or do it several times to the same simplex.
In general, a function assigning an integer to every $(d-1)$-simplex $\rho:C_{d-1}(L) \to \Z$ specifies a way to modify an immersion $f$ on a given $d$-simplex $\gs$ (while keeping it fixed on all other simplices).
Call the resulting immersion $\widetilde f$.
The effect the modification has on the intersection vector is given by the following formula
\[
V_{\widetilde f,\gs,\tau}=V_{f,\gs,\tau}+\rho(\d\tau).
\]
If we interchange $\sigma$ and $\tau$ then (anti)-symmetry of intersection vectors gives
\[
V_{\widetilde f,\tau,\gs}=V_{f,\gs,\tau}+(-1)^d\rho(\d\tau).
\]
Finally, for a pair of simplices $\tau$ and $\tau'$ that are both different from $\sigma$, the intersection vector is unaffected by a finger move applied to $\sigma$:
\[
V_{\widetilde f,\tau,\tau'}=V_{f,\tau,\tau'}.
\]
Note that $\rho(\d\tau)=(\gd\rho)(\tau)$ where $\gd$ is the coboundary operator in cohomology.
So, the above formulas hint at a close connection between embedding theory and cohomology.
In fact, a quick consequence is the following.
\begin{proposition}
    If $H^d(L)=0$ then $L$ embeds in $\R^{2d}$.\label{cohvanish}
\end{proposition}
\begin{proof}
    Fix a $d$-simplex $\gs_1$ in $L$ and look at the function given by intersections with the image of the simplex $f(\gs_1)\cap f(-):C_{d}(L) \to \Z$.
    Since $H^d(L)=0$ this function is a coboundary, i.e. $f(\gs_1)\cap f(-)=\gd\rho$ for some $\rho$.
    Now, do the finger move specified by $-\rho$ to $\gs_1$ to get a new immersion $\widetilde f$ with
    \begin{equation}\label{newint}
        V_{\widetilde f,\gs,\tau}=
        \begin{cases}
            0&\text{ if } \gs=\gs_1,\\
            0 &\text{ if } \tau=\gs_1,\\
            V_{f,\gs,\tau} &\text{ else.}
        \end{cases}
    \end{equation}
    Now repeat for all $d$-simplices of $L$ to obtain an immersion with zero intersection vector and apply the previous proposition.
\end{proof}
The next proposition says that on the level of intersection vectors we can get from any immersion to any other one via finger moves.
\begin{proposition}[\cite{k33}] For two immersions $f$ and $g$ there is a sequence of finger moves which changes $f$ to an immersion with the same intersection vector as $g$.
\end{proposition}
\begin{proof}[Proof idea] By general position, we can assume that $f$ and $g$ agree on the $(d-1)$-skeleton of $L$.
    Since we can move between any two such immersions by modifying them one $d$-simplex at a time, it is enough to prove the proposition when $f$ and $g$ only differ on a single $d$-simplex $\gs$.
    In that situation, there is a homotopy $F:L\times I \to \R^{2d}$ with $F|_{L\times 0}=f$ and $F|_{L\times 1}=g$ and $F(p, t)$ a constant function of $t$ for any point $p$ not in the interior of $\gs$.
    Now, the finger move we need to do to $\gs$ is given by looking at the intersection number $F(\gs\times I)\cap f(-):C_{d-1}(L) \to \Z$.
\end{proof}
\begin{example}
    The utilities graph $K_{3,3}$ does not embed in the plane: Pick an immersion $f:K_{3,3} \to \R^2$ for which $\sum f(\gs)\cap f(\tau)=1$ (mod $2$), where the sum is taken over all unordered, disjoint pairs of edges $\{\gs,\tau\}$.
    Note that this expression does not change if we apply finger moves to $f$.\footnote{The reason is that the collection of edges disjoint from a given edge $\gs$ form a cycle.}
    Therefore, there is no immersion with vanishing intersection vector.
\end{example}

\subsection*{An odd fact.}

The reason $\F _2$ is special in embedding theory is because we can replace the immersion $f$ by another immersion $f'$ whose intersection vector is any odd multiple of the original intersection vector:
\begin{proposition}
    $(2k+1)V_f=V_{f'}$ for some immersion $f'$.
\end{proposition}
\begin{proof}
    Let $r:\R^{2d}\to \R^{2d}$ be a reflection.
    The immersion $r\circ f$ has intersection vector $V_{r\circ f}=-V_f$ since changing the orientation of the ambient space changes the sign of the intersection number.
    On the other hand, there is a sequence of finger moves that takes the immersion $r\circ f$ to one with the same intersection vector as $f$.
    Algebraically this sequence of finger moves takes $-V_f$ to $V_f$.
    Applying this sequence of finger moves to $f$ produces an immersion with intersection vector $3V_f$.
    Applying it to $f$ $k$-times produces an immersion $f'$ with intersection vector $(2k+1)V_f$.
\end{proof}
\begin{corollary}
    If $H^d(L)$ is finite and of odd order, then $L$ embeds in $\R^{2d}$.
\end{corollary}
\begin{proof}
    Let $2k+1=\abs{H^d(L)}$ be the order of the cohomology group.
    By the odd fact, we can choose an initial immersion $f$ whose intersection vector is divisible by $2k+1$.
    Then we have $f(\gs_1)\cap f(-)=(2k+1)\phi$ for some cocycle $\phi:C_d(L) \to \Z$, which implies that $f(\gs_1)\cap f(-)$ is a coboundary, and we can proceed as in the proof of Proposition \ref{cohvanish}.
\end{proof}

\subsection*{Some special generic immersions for octahedralizations}

All this embedding theory is very classical.
It doesn't directly help embed the octahedralization $OL$ because $H^d(OL)$ is never finite: The octahedralizations of top dimensional simplices give infinite order homology, and thus also cohomology, in the top dimension.
The extra idea we used in \cite{ados16} was to start with an initial immersion $f:OL\to S^{2d}$ that is amenable to computation, namely the moment map specified by a particular\footnote{One first picks an ordering $<$ on $L$ and then an ordering $<$ on $OL$ for which the projection $\pi:OL\to L$ is an order-preserving map.
The moment map immersion $OL \to \R^{2d}$ is then defined on vertices by sending the $i$-th vertex in this ordering to $(i, i^2,\dots, i^{2d})$ and extending linearly to all of $OL$.}
ordering on $OL$.
Re-examining the proof in \cite{ados16}, we discovered that all we used was that it is a generic immersion with the following invariance property: Let $\pi:OL\to L$ be the projection map.
Then, for any $d$-simplices $\gs,\tau,$ and $\tau'$ such that $\gs$ is disjoint from both $\tau$ and $\tau'$
\begin{equation}\label{special}
    f(\gs)\cap f(\tau)=f(\gs)\cap f(\tau')\quad \text{ whenever }\quad \pi(\tau)=\pi(\tau').
\end{equation}
Given any such generic immersion, we can prove the first main result of this Appendix.

\subsection*{Proof of Theorem \ref{octaemb}}

The assumption $b_d(L;\F _2)=0$ implies that the top cohomology of $L$ is odd torsion, i.e. $\abs{H^d(L)}=2k+1$.
Let $\pi:OL\to L$ be the projection map.
Start with an immersion of $L$ in $S^{2d}$ whose intersection vector is divisible by $2k+1$ and perturb it to get an immersion $f:OL\to S^{2d}$.
Its intersection vector is divisible by $2k+1$ and has the additional property \eqref{special}.

Look at a $d$-simplex $\gs_1$ of $OL$.
Property \eqref{special} implies that
\[
f(\gs_1)\cap f(-):C_d(OL) \to \Z
\]
factors through the projection to $L$ as $C_d(OL)\stackrel{\pi}\to C_d(L)\stackrel{\varphi} \to \Z$.
The map $\varphi$ is divisible by $2k+1$ so, since $\abs{H^d(L)}=2k+1$, it is a coboundary (i.e $\phi=\gd\rho$), and thus $f(\gs_1)\cup f(-)=\pi^*\varphi$ is a coboundary, as well.
Applying the finger move $-\pi^*\rho$ to $\gs_1$ we obtain a new immersion $\widetilde f$ whose intersection vector is given by \eqref{newint}.
The new immersion $\widetilde f$ no longer satisfies property \eqref{special} for all $\gs$ in $OL$,\footnote{To see why, pick $\gs_1'\neq \gs_1$ with $\pi(\gs_1')=\pi(\gs_1)$ and look at a simplex $\gs$ disjoint from both $\gs_1$ and $\gs_1'$.
Then $\widetilde f(\gs)\cap\widetilde f(\gs_1)=0$ while $\widetilde f(\gs)\cap\widetilde f(\gs_1')=f(\gs)\cap f(\gs_1')=f(\gs)\cap f(\gs_1)$ may not be.}
but it does so for all $d$-simplices $\gs$ projecting to $\pi(\gs_1)$.
(The key point here is that there is only one $d$-simplex lying over $\pi_1(\gs_1)$ and disjoint from $\gs_1$.) So, we can repeat the same argument for all the $d$-simplices projecting to $\pi(\gs_1)$ to obtain a new immersion $\hat f$ with intersection vector
\[
V_{\hat f,\gs,\tau}=
\begin{cases}
    0&\text{ if } \pi(\gs)=\pi(\gs_1),\\
    0&\text{ if } \pi(\tau)=\pi(\gs_1),\\
    V_{f,\gs,\tau}&\text{ else.}
\end{cases}
\]
The immersion $\hat f$ does satisfy \eqref{special} for all $d$-simplices $\gs$ in $OL$, so we can repeat the argument for all simplices of $OL$, obtain an immersion with zero intersection vector and finally apply Proposition \ref{alg} to get an embedding.
This finishes the proof of Theorem \ref{octaemb}.

\subsection*{$2$-dimensional octahedralizations}

One might wonder if there is a more natural way to embed $OL$ in $S^{2d}$.
Recently \foreignlanguage{vietnamese}{T\^am Nguy\~\ecircumflex{}n Phan} and the first author showed that the $d\neq 2$ restriction in Theorem \ref{octaemb} cannot be removed.
\begin{theorem}[\cite{ap21}] For a sufficiently fine flag triangulation of the $2$-complex $D^2\cup_3S^1$, the octahedralization $O(D^2\cup_3S^1)$ does not $PL$ embed in $S^4$.
\end{theorem}
This seems to be evidence against the existence of a more natural or canonical embedding of $OL$ in higher dimensions.

\subsection*{Immersions of $OL$ via perturbation}

The goal of the remainder of this appendix is to show that we don't really need orderings or the moment map to obtain a generic immersion satisfying the invariance property \eqref{special}.
A small, generic, linear perturbation of a generic linear immersion of $L$ will do!

Suppose $L$ is linearly immersed in $\R^{2d}$.
For each vertex $v$ pick a vector $X_v$ in $\R^{2d}$.
For a simplex $\gs$ and subset of its vertices $A$, we will denote by $\gs_{\ge A}$ the linear simplex in $\R^{2d}$ obtained from the image of $\gs$ by moving each vertex $v$ in $A$ by the vector $\ge X_v$.
Denote by $\overline{\gs_{\ge A}}$ the plane spanned by this simplex.

The collection of vectors $\{X_v\}$ is a \emph{generic linear perturbation} of $L \to \R^{2d}$ if for every pair of $d$-simplices $\gs$ and $\tau$ in $L$ and partition containing their common vertices
\[
L^{(0)}\supset A\coprod B\supset (\gs\cap\tau)^{(0)}
\]
there is an $\ge'>0$ such that for each $0<\ge\leq\ge'$ the following two conditions hold.
\begin{itemize}
    \item $\overline{\gs_{\ge A}}$ and $\overline{\tau_{\ge B}}$ are $d$-planes that intersect in a single point $p_{\ge}$, and
    \item $p_{\ge}$ does not lie on the boundaries of the simplicies: $p_{\ge}\notin\d\gs_{\ge A}\cup\d\tau_{\ge B}$.
\end{itemize}
In particular, for a pair of simplices $\gs$ and $\tau$ that are disjoint in $L$, this is saying that their intersection in $\R^{2d}$ is transverse.

It is easy to see that generic linear perturbations of generic immersions exist, and in fact form an open dense subset of all possible choices of vectors $\{X_v\}$.

Genericity has several important consequences.
\begin{itemize}
    \item[(1)] The intersections form a half-open linear interval $p_{[\ge',0)}$ (because everything defining it is linear) converging to a point $p_s\to p_0\in\rr^{2d}$.
    In other words, there is a vector $W$ in $\R^{2d}$ such that for $0<\ge\leq \ge'$ we have $p_{\ge}=p_0+\ge W$.
    Next, pick a vertex $v$ of $\tau$ that is not contained in $A\coprod B$ and look at what happens to the intersection as we vary $\tau_{\ge B}$ by moving the vertex $v$ in the direction $tX_v$.
    This varies the vector $W$ linearly in $t$, i.e. we have for small enough $t$ that
    \[
    \overline{\gs_{\ge A}}\cap\overline{\tau_{tv\cup\ge B}}=p_0+\ge(W+tV)=p_{\ge}+\ge tV
    \]
    for some vector $V$ which can be expressed infinitesimally as
    \[
    V:=\frac{1}{\ge}\frac{d(\overline {\gs_{\ge A}}\cap\overline{\tau_{tv\cup\ge B}})}{dt} \big\rvert_{t=0}.
    \]
    In particular, for small enough $\ge$, the intersection between the planes spanned by $\gs_{\ge A}$ and $\tau_{\ge (v\cup B)}$ occurs at $p_{\ge}+\ge^2V$.
    \item[(2)] There is linear isomorphism $\pi:\overline{\gs_{\ge A}}\cong\overline{\gs}$ that identifies $\gs_{\ge A}$ with $\gs$ and is bilipschitz, with bilipschitz constant tending to one as $\ge\to 0$.
    The image $\pi(p_{[\ge',0)})$ is a linear family in $\overline\gs$ that does not meet $\d\gs$, which implies there is a constant $C'$ such that $d(\pi(p_{\ge}),\d\gs)>C'\ge$ for all $0<\ge\leq\ge'$.
    The same statement holds with $\gs_{\ge A}$ and $\gs$ replaced by $\tau_{\ge B}$ and $\tau$.
    Therefore, there is a positive constant $C$ such that
    \begin{equation}\label{far}
        d(p_{\ge},\d\gs_{\ge A}\cup\d\tau_{\ge B})>C\ge.
    \end{equation}
\end{itemize}
Now, let $P:=P^{d+1}$ be the $(d+1)$-plane spanned by $\tau_{\ge B}$ and $\tau_{\ge(v\cup B)}$, (or, equivalently, by $\tau_{\ge B}$ and $X_v$).
Parametrize $P\cap\overline{\gs_{\ge A}}$ as a line $L(s)$ with $L(0)=p_{\ge}$ and $L(1)=p_{\ge}+\ge^2 V$.
Note that
\begin{itemize}
    \item for small enough $\ge$ it follows from \eqref{far} that $L(0)\in\gs_{\ge A}$ if and only if $L(1)\in\gs_{\ge A}$.
\end{itemize}

We will need an analogous statement involving $\tau_{\ge B}$ and $\tau_{\ge(v\cup B)}$.
To that end, note that $P$ has linear retractions
\[
\overline{\tau_{\ge B}} \xleftarrow{r_1} P_{\ge}\xrightarrow{r_2} \overline{\tau_{\ge(v\cup B)}}
\]
whose Lipschitz constants are uniformly bounded as $\ge\to 0$ and which restrict to inverse isomorphisms $r_2|_{\overline{\tau_{\ge B}}}=r_1^{-1}|_{\overline{\tau_{\ge(v\cup B)}}}$ identifying $\tau_{\ge B}$ with $\tau_{\ge(v\cup B)}$.
\begin{lemma}
    For small enough $\ge>0$, we have $L(0)\in\tau_{\ge B}$ if and only if $L(1)\in\tau_{\ge(v\cup B)}$.
\end{lemma}
\begin{proof}
    Note that $p_{\ge}=L(0)=r_1(L(0))$ and that $d(r_1(L(0)),r_1(L(1)))\leq\abs{r_1\abs{\abs{V}\ge^2$.
    Since $\abs{r_1}}V}$ is uniformly bounded independent of $\ge$, \eqref{far} implies for small enough $\ge$ that $p_{\ge}$ is in $\tau_{\ge B}$ if and only if $r_1(L(1))$ is in $\tau_{\ge B}$.
    This happens if and only if $r_2r_1(L(1))=L(1)$ is in $\tau_{\ge(v\cup B)}$.
    So, we are done.
\end{proof}
\begin{corollary}
    For small enough $\ge>0$, $\gs_{\ge A}$ intersects $\tau_{\ge B}$ if and only if $\gs_{\ge A}$ intersects $\tau_{\ge(B\cup v)}$.
    Moreover, the signs of the intersections are the same.
\end{corollary}
\begin{proof}
    The bullet point and lemma above imply that for small enough $\ge$ we have $L(0)\in\gs_{\ge A}\cap\tau_{\ge B}$ if and only if $L(1)\in \gs_{\ge A}\cap\tau_{\ge(v\cup B)}$, which proves the first part.
    The second statement is clear because (for small enough $\ge$ and $0\leq t\leq\ge$) the signed intersection number $\overline{\gs_{\ge A}}\cap\overline{\tau_{tv\cup \ge B}}$ is defined and independent of $t$.
\end{proof}
\begin{proposition}[Invariance property] Suppose $f:L\to \R^{2d}$ is a generic linear immersion.
    Then for any generic linear perturbation $\{X_v\}$ there is an $\varepsilon>0$ such that the linear immersion $f_{\varepsilon}:OL \to \R^{2d}$ defined on vertices by $v^+\mapsto f(v),v^-\mapsto f(v)+\ge X_v$ is generic and its signed intersection numbers satisfy
    \[
    f_{\ge}(\gs)\cap f_{\ge}(\tau)=f_{\ge}(\gs)\cap f_{\ge}(\tau')
    \]
    whenever $\gs,\tau$ and $\tau'$ are $d$-simplices, $\tau$ and $\tau'$ are disjoint from $\gs$ and $\pi(\tau)=\pi(\tau')$.
\end{proposition}
\begin{proof}
    Given $\gs$, $\tau$ and $\tau'$ as in the statement of the proposition, $f_{\ge}(\gs)=\gs_{\ge A},$ $f_{\ge}(\tau)=\tau_{\ge B}$ and $f_{\ge}(\tau')=\tau_{\ge B'}$ for some $A\coprod B\supset(\gs\cap\tau)^{(0)}\subset A\coprod B'$.
    Repeatedly using the Corollary we conclude that for small enough $\ge$ there are equalities of signed intersection numbers
    \[
    \gs_{\ge A}\cap\tau_{\ge B}=\gs_{\ge A}\cap\tau_{\ge(B\cup B')}=\gs_{\ge A}\cap\tau_{\ge B'}
    \]
    which proves the proposition.
\end{proof}
\begin{remark}
    The embedding $OL\into S^{2d}$ produced by the methods of this appendix appears quite exotic from a metric perspective, because it starts with an immersion of $OL$ obtained by perturbing an immersion of $L$.
    In this immersion, two vertices $v^+$ and $v^-$ that correspond to the positive and negative directions of the loop $S^1_v$ are put very close together, as opposed to being antipodal in the sphere $S^{2d}$.
\end{remark}

\end{document}